\documentclass[12pt,thmsa]{article}
\usepackage{amssymb}
\usepackage{amsfonts}
\usepackage{amsmath}
\usepackage{cite}
\usepackage[top=3.8cm,bottom=3.8cm,textwidth=16cm,centering,a4paper]{geometry}

\numberwithin{equation}{section}
\newtheorem{theorem}{Theorem}[section]
\newtheorem{proposition}[theorem]{Proposition}
\newtheorem{corollary}[theorem]{Corollary}
\newtheorem{remark}{Remark}[section]
\newtheorem{lemma}[theorem]{Lemma}

\newenvironment{proof}[1][Proof]{\noindent\textbf{#1.} }{\hfill $\square$}

\begin{document}

\title{On the solvability of an indefinite nonlinear Kirchhoff equation via
associated eigenvalue problems}
\date{}
\author{Han-Su Zhang$^{a}$\thanks{%
~\textit{E-mail addresses}: hansu\_zhang@foxmail.com}, Tiexiang Li$^{a}$%
\thanks{%
~\textit{E-mail addresses}: txli@seu.edu.cn}, Tsung-fang Wu$^{b}$\thanks{%
~\textit{E-mail addresses}: tfwu@nuk.edu.tw} \\
%EndAName
{\footnotesize $^{a}$School of mathematics, Southeast University, Nanjing
211189, PR. China}\\
$^{{\footnotesize b}}${\footnotesize Department of Applied Mathematics,
National University of Kaohsiung, Kaohsiung 811, Taiwan}}
\maketitle

\begin{abstract}
We study the non-existence, existence and multiplicity of positive solutions
to the following nonlinear Kirchhoff equation:%
\begin{equation*}
\left\{
\begin{array}{l}
-M\left( \int_{\mathbb{R}^{3}}\left\vert \nabla u\right\vert ^{2}dx\right)
\Delta u+\mu V\left( x\right) u=Q(x)\left\vert u\right\vert ^{p-2}u+\lambda
f\left( x\right) u\text{ in }\mathbb{R}^{N}, \\
u\in H^{1}\left( \mathbb{R}^{N}\right) ,%
\end{array}%
\right.
\end{equation*}%
where $N\geq 3,2<p<2^{\ast }:=\frac{2N}{N-2},M\left( t\right) =at+b$ $\left(
a,b>0\right) ,$ the potential $V$ is a nonnegative function in $\mathbb{R}%
^{N}$ and the weight function $Q\in L^{\infty }\left( \mathbb{R}^{N}\right) $
with changes sign in $\overline{\Omega }:=\left\{ V=0\right\} .$ We mainly
prove the existence of at least two positive solutions in the cases that $%
\left( i\right) $ $2<p<\min \left\{ 4,2^{\ast }\right\} $ and $0<\lambda <%
\left[ 1-2\left[ \left( 4-p\right) /4\right] ^{2/p}\right] \lambda
_{1}\left( f_{\Omega }\right) ;$ $\left( ii\right) $ $p\geq 4,\lambda \geq
\lambda _{1}\left( f_{\Omega }\right) $ and near $\lambda _{1}\left(
f_{\Omega }\right) $ for $\mu >0$ sufficiently large, where $\lambda
_{1}\left( f_{\Omega }\right) $ is the first eigenvalue of $-\Delta $ in $%
H_{0}^{1}\left( \Omega \right) $ with weight function $f_{\Omega }:=f|_{%
\overline{\Omega }},$ whose corresponding positive principal eigenfunction
is denoted by $\phi _{1}.$ Furthermore, we also investigated the
non-existence and existence of positive solutions if $a,\lambda $ belongs to
different intervals.
\end{abstract}


\section{Introduction}

In this paper we are concerned the following nonlinear Kirchhoff equation:%
\begin{equation}
\left\{
\begin{array}{l}
-M\left( \int_{\mathbb{R}^{3}}\left\vert \nabla u\right\vert ^{2}dx\right)
\Delta u+\mu V\left( x\right) u=g(x,u)\text{ in }\mathbb{R}^{N}, \\
u\in H^{1}\left( \mathbb{R}^{N}\right) ,%
\end{array}%
\right.  \label{1-1}
\end{equation}%
where $N\geq 3,g\in \mathbb{R}^{N}\times \mathbb{R}\rightarrow \mathbb{R}\,$
being continuous, $M(s)=as+b\,\left( a,b>0\right) $ and the parameter $\mu
>0.$ We assume that the potential function $V$ satisfies the following
conditions:

\begin{itemize}
\item[$\left( V_{1}\right) $] $V$ is a nonnegative continuous function on $%
\mathbb{R}^{N};$

\item[$\left( V_{2}\right) $] there exists $c>0$ such that the set $\left\{
V<c\right\} :=\left\{ x\in \mathbb{R}^{N}\ |\ V\left( x\right) <c\right\} $
is nonempty and has finite Lebesgue measure;

\item[$(V_{3})$] $\Omega =\mathrm{int}\left\{ x\in \mathbb{R}^{N}\ |\
V\left( x\right) =0\right\} $ is nonempty bounded domain and has a smooth
boundary with $\overline{\Omega }=\left\{ x\in \mathbb{R}^{N}\ |\ V\left(
x\right) =0\right\} .$
\end{itemize}

The hypotheses $\left( V_{1}\right) -\left( V_{3}\right) $ imply that $\mu V$
represents a potential well whose depth is controlled by $\mu .$ $\mu V$ is
called a steep potential well if $\mu $ is sufficiently large and one
expects to find solutions which localize near its bottom $\Omega .$ This
problem has found much interest after being first introduced by Bartch and
Wang \cite{BW} in the study of the existence of positive solutions for
nonlinear Schr\"{o}dinger equations and has been attracting much attention,
see \cite{AN1,BPW,BT,SZ,WZ} and the references therein.

Kirchhoff type equations, of the form similar to Equation $(\ref{1-1}),$
originate from physics. Indeed, if we set $V(x)\equiv 0$ and replace $%
\mathbb{R}^{N}$ by a bounded domain $\Omega \subset \mathbb{R}^{N}$ in
Equation $(\ref{1-1}),$ then it becomes the following Dirichlet problem of
Kirchhoff type:
\begin{equation}
\left\{
\begin{array}{ll}
-\left( a\int_{\Omega }|\nabla u|^{2}dx+b\right) \Delta u=g(x,u) & \text{ in
}\Omega , \\
u=0 & \ \text{on }\partial \Omega ,%
\end{array}%
\right.  \label{1-3}
\end{equation}%
which is analogous to the stationary case of equations that arise in the
study of string or membrane vibrations, namely,
\begin{equation}
u_{tt}-\left( a\int_{\Omega }|\nabla u|^{2}dx+b\right) \Delta u=g(x,u),
\label{1-2}
\end{equation}%
where$\ u$ denotes the displacement, $g$ is the external force and $b$ is
the initial tension while $a$ is related to the intrinsic properties of the
string (such as Young's modulus). Equation $(\ref{1-2})$ was first proposed
by Kirchhoff \cite{K} in 1883 to describe the transversal oscillations of a
stretched string, particularly, taking into account the subsequent change in
string length caused by oscillations. It is notable that Equation $(\ref{1-2}%
)$ is often referred to as being nonlocal because of the presence of the
integral over the domain $\Omega .$

After the pioneering work by Pohozaev \cite{P} and Lions \cite{L}, the
qualitative analysis of nontrivial solutions for the nonlinear Kirchhoff
type equations, similar to Equation $(\ref{1-1})$, has begun to receive much
attention in recent years. We refer the reader to \cite%
{AF,CKW,DPS,DS1,FIJ,G,HZ,I,JL,LY,SCWF,SW1,SW3,SW4,TC,WHL,XM} and the
references therein.

Let us briefly comment on some of the things that are relevant to our work.
In \cite{SW1} introduced the steep potential well $V$ to the Kirchhoff type
equations. When the potential $V$ satisfies the hypotheses $(V_{1})-(V_{3}),$
the following results were obtained.\newline
$(i)$ $N\geq 3:$ if $0<a<a^{\ast }$ and $\mu >0$ sufficiently large, then
Equation $(\ref{1-1})$ has at least one positive solution, when $g(x,u)\ $is
asymptotically linear at infinity on $u$ and $b\lambda _{1}^{(1)}<1;$\newline
$(ii)$ $N=3:$ if $0<a<\lambda _{1}^{(3)}$ and $\mu >0$ sufficiently large,
then Equation $(\ref{1-1})$ has at least one positive solution, when $%
g(x,u)\ $is asymptotically $3$-linear at infinity on $u;$\newline
$(iii)$ $N=3:$ for any $a>0$ and $\mu >0$ sufficiently large, Equation $(\ref%
{1-1})$ has at least one positive solution, when $g(x,u)\ $is asymptotically
$4$-linear at infinity on $u.$

After that, Xie and Ma \cite{XM} obtained the existence and concentration of
positive solutions for Equation $(\ref{1-1})$ with $N=3$ when potential $V$
satisfies conditions $(V_{1})-(V_{3})$ and nonlinearity $g$ satisfies the
following conditions:

\begin{itemize}
\item[$(G_{1})$] there exists $\rho >4$ such that $0<\rho G(x,u)\leq g(x,u)u$
for $u>0,$ where $G(x,u)=\int_{0}^{u}g(x,s)ds;$

\item[$(G_{2})$] $\frac{G(x,u)}{u^{3}}$ is increasing for $u>0.$
\end{itemize}

In our recent papers \cite{SCWF, SW4}, we concluded that when $N\geq 3$ and $%
g(x,u)$ is superlinear and subcritical on $u$, the geometric structure of
the functional $J$ related to Equation $(\ref{1-1})$ is known to have a
global minimum and a mountain pass, owing to the fourth power of the
nonlocal term. By using the standard variational methods, two different
positive solutions can be found, since some embedding inequalities are
proved with the help of the fact of $2^{\ast }:=\frac{2N}{N-2}\leq 4.$

In simple terms, when $g(x,u)=Q(x)|u|^{\,p-2}u$ and $Q\in L^{\infty }\left(
\mathbb{R}^{N}\right) $ is sign-changing, the current progress through the
above literature is as follows:

\begin{itemize}
\item[$\left( I\right) $] $N=3$ and $4<p<6:$ for any $a>0$ and $\mu >0$
sufficiently large, Equation $(\ref{1-1})$ has at least one positive
solution;

\item[$\left( II\right) $] $N=3$ and $2<p\leq 4:$ for $a>0$ small enough and
$\mu >0$ sufficiently large, Equation $(\ref{1-1})$ has at least one
positive solution;

\item[$\left( III\right) $] $N\geq 4$ and $2<p<2^{\ast }:$ for $a>0$ small
enough and $\mu >0$ sufficiently large, Equation $(\ref{1-1})$ has at least
two positive solution.
\end{itemize}

Motivated by these findings, we now extend the analysis to the Kirchhoff
type equation with combination of a superlinear term and a linear term, that
is $g(x,u)=Q(x)|u|^{\,p-2}u+\lambda f(x)u$. Our intension here is to
illustrate the difference in the solution behavior which arises from the
consideration of the nonlocal and eigenvalue problem effects. The problem we
consider is thus%
\begin{equation}
\left\{
\begin{array}{l}
-M\left( \int_{\mathbb{R}^{N}}\left\vert \nabla u\right\vert ^{2}dx\right)
\Delta u+\mu V\left( x\right) u=Q(x)\left\vert u\right\vert ^{p-2}u+\lambda
f\left( x\right) u\text{ in }\mathbb{R}^{N}, \\
u\in H^{1}\left( \mathbb{R}^{N}\right) ,%
\end{array}%
\right.  \tag{$E_{\mu ,\lambda }$}
\end{equation}%
where $N\geq 3,2<p<2^{\ast }:=\frac{2N}{N-2},M\left( t\right) =at+b$ $\left(
a,b>0\right) $ and the parameters $\mu ,\lambda >0.$ We are interested in
the case the weight functions $f$ and $Q\ $are sign-changing in $\mathbb{R}%
^{N},$ which is why we call indefinite nonlinear Kirchhoff equation in the
title.

To go further, let us give some notations first. For the sake of simplicity,
we always assume that $b=1$ in Equation $(E_{\mu ,\lambda }).$ Let $%
D^{1,2}\left( \mathbb{R}^{N}\right) $ be the completing of $C_{0}^{\infty
}\left( \mathbb{R}^{N}\right) $ with respect to the norm $\left\Vert
u\right\Vert _{D^{1,2}}^{2}=\int_{\mathbb{R}^{N}}\left\vert \nabla
u\right\vert ^{2}dx.$ Denote by $S_{p},S_{p}(\Omega )$ and $S$ the best
constants for the embeddings of $H^{1}(\mathbb{R}^{N})$ in $L^{p}(\mathbb{R}%
^{N}),H_{0}^{1}(\Omega )$ in $L^{p}(\Omega )$ and $D^{1,2}(\mathbb{R}^{N})$
in $L^{2^{\ast }}(\mathbb{R}^{N})$, respectively. We denote a strong
convergence by \textquotedblleft $\rightarrow $\textquotedblright\ and a
weak convergence by \textquotedblleft $\rightharpoonup $\textquotedblright .

Throughout this paper, $u\in H^{1}(\mathbb{R}^{N})$ is a solution of
Equation $(E_{\mu ,\lambda })$ if for any $v\in H^{1}(\mathbb{R}^{N})$ there
holds%
\begin{equation*}
M\left( \int_{\mathbb{R}^{N}}\left\vert \nabla u\right\vert ^{2}dx\right)
\int_{\mathbb{R}^{N}}\nabla u\nabla v+\mu \int_{\mathbb{R}^{N}}V\left(
x\right) uv=\int_{\mathbb{R}^{N}}\left( Q(x)\left\vert u\right\vert
^{p-2}uv+\lambda f\left( x\right) uv\right) dx.
\end{equation*}%
And $u$ is called a positive solution if $u$ is a solution and $u>0$ in $%
\mathbb{R}^{N}.$

It is well known that Equation $\left( E_{\mu ,\lambda }\right) $ is
variational, and its solutions correspond to the critical point of the
energy functional $J_{\mu ,\lambda }:X_{\mu }\rightarrow \mathbb{R}$%
\begin{equation*}
J_{\mu ,\lambda }\left( u\right) =\frac{a}{4}\left\Vert u\right\Vert
_{D^{1,2}}^{4}+\frac{1}{2}\left\Vert u\right\Vert _{\mu }^{2}-\frac{1}{p}%
\int_{\mathbb{R}^{N}}Q\left\vert u\right\vert ^{p}dx-\frac{\lambda }{2}\int_{%
\mathbb{R}^{N}}fu^{2}dx.
\end{equation*}%
where $\left\Vert u\right\Vert _{\mu }=\left[ \int_{\mathbb{R}^{N}}\left(
\left\vert \nabla u\right\vert ^{2}+\mu Vu^{2}\right) dx\right] ^{1/2}$ is
the standard norm in $X_{\mu }$ and $X_{\mu }$ is a subspace of $H^{1}\left(
\mathbb{R}^{N}\right) $ (see below). Thus, if $u$ is a critical point of $%
J_{\mu ,\lambda }$ on $X_{\mu },$ then $u$ is a solution of Equation $\left(
E_{\mu ,\lambda }\right) .$

Assume the following hypotheses $\left( D\right) :$

\begin{itemize}
\item[$\left( D_{1}\right) $] $f\in L^{N/2}\left( \mathbb{R}^{N}\right) $
which $f^{+}:=\max \left\{ f,0\right\} \not\equiv 0$ in $\Omega ;$

\item[$\left( D_{2}\right) $] $Q\in L^{\infty }\left( \mathbb{R}^{N}\right) $
which $Q^{+}:=\max \left\{ Q,0\right\} \not\equiv 0$ in\ $\Omega .$
\end{itemize}

\begin{remark}
\label{r1}Since $\left\{ f>0\right\} \cap \Omega $ has positive Lebesgue
measure, we\ can assume that $\lambda _{1}\left( f_{\Omega }\right) $ denote
the positive principal eigenvalue of the problem%
\begin{equation}
-\Delta u(x)=\lambda f_{\Omega }(x)u(x)\text{ for}\;x\in \Omega ;\text{ }%
u(x)=0\text{ for}\;x\in \partial \Omega ,  \label{2}
\end{equation}%
where $f_{\Omega }$ is a restriction of $f$ on $\overline{\Omega }.$
Clearly, $\lambda _{1}\left( f_{\Omega }\right) $ has a corresponding
positive principal eigenfunction $\phi _{1}$ with $\int_{\Omega }f_{\Omega
}\phi _{1}^{2}dx=1$ and $\int_{\Omega }\left\vert \nabla \phi
_{1}\right\vert ^{2}dx=\lambda _{1}\left( f_{\Omega }\right) .$
\end{remark}

We now summarize our main results as follows.

\begin{theorem}
\label{t1}Suppose that $N=3,4<p<6$ and conditions $\left( V_{1}\right)
-\left( V_{3}\right) $ and $\left( D_{1}\right) -\left( D_{2}\right) $ hold.
Then for each $a>0$ and $0<\lambda <\lambda _{1}\left( f_{\Omega }\right) ,$
Equation $\left( E_{\mu ,\lambda }\right) $ has a positive solution $u_{\mu
}^{-}$ satisfying $J_{\mu ,\lambda }\left( u_{\mu }^{-}\right) >0$ for $\mu
>0$ sufficiently large.
\end{theorem}

\begin{theorem}
\label{t2}Suppose that $N=3,$ $4<p<6,$ conditions $\left( V_{1}\right)
-\left( V_{3}\right) $ and $\left( D_{1}\right) -\left( D_{2}\right) $ hold
and $\int_{\Omega }Q\phi _{1}^{p}dx<0.$ Then for each $a>0$ there exists $%
\delta _{0}$ such that for every $\lambda _{1}\left( f_{\Omega }\right) \leq
\lambda <\lambda _{1}\left( f_{\Omega }\right) +\delta _{0},$ Equation $%
\left( E_{\mu ,\lambda }\right) $ has at least two positive solutions $%
u_{\mu }^{-}$ and $u_{\mu }^{+}$ satisfying $J_{\mu ,\lambda }\left( u_{\mu
}^{+}\right) <0<J_{\mu ,\lambda }\left( u_{\mu }^{-}\right) $ for $\mu >0$
sufficiently large.
\end{theorem}

To consider the case $N=3$ and $p=4,$ we need the following minimum problem%
\begin{equation*}
\Gamma _{0}:=\sup_{u\in X}\frac{\int_{\mathbb{R}^{3}}Q|u|^{4}dx}{\left\Vert
u\right\Vert _{D^{1,2}}^{4}}>0.
\end{equation*}%
Then we have the following results.

\begin{theorem}
\label{t3}Suppose that $N=3,$ $p=4$ and conditions $\left( V_{1}\right)
-\left( V_{3}\right) $ and $\left( D_{1}\right) -\left( D_{2}\right) $ hold.
Then we have the following results.\newline
$\left( i\right) $ For each $0<a<\Gamma _{0}$ and $0<\lambda <\lambda
_{1}\left( f_{\Omega }\right) $, Equation $\left( E_{\mu ,\lambda }\right) $
has a positive solution $u_{\mu }^{-}$ satisfying $J_{\mu ,\lambda }\left(
u_{\mu }^{-}\right) >0$ for $\mu >0$ sufficiently large.\newline
$\left( ii\right) $ If $\Gamma _{0}<\infty ,$ then for each $a>\Gamma _{0}$
and $0<\lambda <\lambda _{1}\left( f_{\Omega }\right) ,$ Equation $\left(
E_{\mu ,\lambda }\right) $ does not admits nontrivial solution for $\mu >0$
sufficiently large.\newline
$\left( iii\right) $ If $\Gamma _{0}<\infty ,$ then for each $a>\Gamma _{0}$
and $\lambda \geq \lambda _{1}\left( f_{\Omega }\right) ,$ Equation $\left(
E_{\mu ,\lambda }\right) $ has a positive solution $u_{\mu }^{+}$ satisfying
$J_{\mu ,\lambda }\left( u_{\mu }^{+}\right) <0$ for $\mu >0$ sufficiently
large.
\end{theorem}

\begin{theorem}
\label{t4}Suppose that $N=3,$ $p=4$ and conditions $\left( V_{1}\right)
-\left( V_{3}\right) $ and $\left( D_{1}\right) -\left( D_{2}\right) $ hold.
Then for each $\lambda _{1}^{-2}\left( f_{\Omega }\right) \int_{\Omega
}Q\phi _{1}^{4}dx<a<\Gamma _{0}$ there exists $\delta _{0}$ such that for
every $\lambda _{1}\left( f_{\Omega }\right) \leq \lambda <\lambda
_{1}\left( f_{\Omega }\right) +\delta _{0},$ Equation $\left( E_{\mu
,\lambda }\right) $ has two positive solutions $u_{\mu }^{-}$ and $u_{\mu
}^{+}$ satisfying $J_{\mu ,\lambda }\left( u_{\mu }^{+}\right) <0<J_{\mu
,\lambda }\left( u_{\mu }^{-}\right) $ for $\mu >0$ sufficiently large.
\end{theorem}

To consider the case $2<p<\min \left\{ 4,2^{\ast }\right\} ,$ we first show
that the nonexistence of solutions.

\begin{theorem}
\label{t5-0}Suppose that $N\geq 4,$ $2<p<2^{\ast }$ and conditions $\left(
V_{1}\right) -\left( V_{3}\right) $ and $\left( D_{1}\right) -\left(
D_{2}\right) $ hold. Then for each $0<\lambda <\lambda _{1}\left( f_{\Omega
}\right) $ there exists%
\begin{equation*}
0<\overline{\mathbf{A}}_{\lambda }<\frac{1}{2}\left( \frac{\left( 4-p\right)
\lambda _{1}\left( f_{\Omega }\right) }{p\left( \lambda _{1}\left( f_{\Omega
}\right) -\lambda \right) }\right) ^{(4-p)/(p-2)}\left( \frac{\left\Vert
Q\right\Vert _{\infty }\left\vert \left\{ V<c\right\} \right\vert ^{\frac{%
2^{\ast }-p}{2^{\ast }}}}{S^{p}}\right) ^{2/(p-2)}
\end{equation*}%
such that for every $a>\overline{\mathbf{A}}_{\lambda },$ Equation $\left(
E_{\mu ,\lambda }\right) $ does not admits nontrivial solution for $\mu >0$
sufficiently large.
\end{theorem}

To prove the existence of positive solution, we need the following
conditions:

\begin{itemize}
\item[$\left( D_{3}\right) $] There exist two numbers $c_{\ast },R_{\ast }>0$
such that%
\begin{equation*}
\left\vert x\right\vert ^{p-2}Q\left( x\right) \leq c_{\ast }\left[ V\left(
x\right) \right] ^{4-p}\text{ for all }\left\vert x\right\vert >R_{\ast }.
\end{equation*}

\item[$\left( D_{4}\right) $] $\left\vert \left\{ V<c\right\} \right\vert
^{\left( 6-p\right) /6}\leq \frac{S^{p}Q_{\Omega ,\min }}{S_{p}^{p}\left(
\Omega \right) \left\Vert Q\right\Vert _{\infty }},$ where $Q_{\Omega ,\min
}=\inf_{x\in \overline{\Omega }}Q\left( x\right) >0.$
\end{itemize}

Then we have the following results.

\begin{theorem}
\label{t5}Suppose that $N=3,$ $2<p<4$ and conditions $\left( V_{1}\right)
-\left( V_{3}\right) $ and $\left( D_{1}\right) -\left( D_{3}\right) $ hold.
Then we have the following results.\newline
$\left( i\right) $ There exists $a_{0}>0$ such that for every $0<a<a_{0}$
and $0<\lambda <\lambda _{1}\left( f_{\Omega }\right) $, Equation $\left(
E_{\mu ,\lambda }\right) $ has a positive solution $u_{\mu }^{+}$ satisfying
$J_{\mu ,\lambda }\left( u_{\mu }^{+}\right) <0$ for $\mu >0$ sufficiently
large.\newline
$\left( ii\right) $ For each $\lambda \geq \lambda _{1}\left( f_{\Omega
}\right) $ and $a>0$, Equation $\left( E_{\mu ,\lambda }\right) $ has a
positive solution $u_{\mu }^{+}$ satisfying $J_{\mu ,\lambda }\left( u_{\mu
}^{+}\right) <0$ for $\mu >0$ sufficiently large.
\end{theorem}

\begin{theorem}
\label{t5-2}Suppose that $N\geq 4,$ $2<p<2^{\ast }$ and conditions $\left(
V_{1}\right) -\left( V_{3}\right) $ and $\left( D_{1}\right) -\left(
D_{2}\right) $ hold. Then we have the following results.\newline
$\left( i\right) $ There exists $a_{0}>0$ such that for every $0<a<a_{0}$
and $0<\lambda <\lambda _{1}\left( f_{\Omega }\right) $, Equation $\left(
E_{\mu ,\lambda }\right) $ has a positive solution $u_{\mu }^{+}$ satisfying
$J_{\mu ,\lambda }\left( u_{\mu }^{+}\right) <0$ for $\mu >0$ sufficiently
large.\newline
$\left( ii\right) $ For each $a>0$ and $\lambda \geq \lambda _{1}\left(
f_{\Omega }\right) $, Equation $\left( E_{\mu ,\lambda }\right) $ has a
positive solution $u_{\mu }^{+}$ satisfying $J_{\mu ,\lambda }\left( u_{\mu
}^{+}\right) <0$ for $\mu >0$ sufficiently large.
\end{theorem}

\begin{theorem}
\label{t6}Suppose that $N\geq 3,$ $2<p<\min \left\{ 4,2^{\ast }\right\} $
and conditions $\left( V_{1}\right) -\left( V_{3}\right) ,\left(
D_{1}\right) -\left( D_{2}\right) $ and $\left( D_{4}\right) $ hold. Then
there exists $a_{0}>0$ such that for every $0<a<a_{0}$ and $0<\lambda <\left[
1-2\left( \frac{4-p}{4}\right) ^{2/p}\right] \lambda _{1}\left( f_{\Omega
}\right) $, Equation $\left( E_{\mu ,\lambda }\right) $ has a positive
solution $u_{\mu }^{-}$ satisfying $J_{\mu ,\lambda }\left( u_{\mu
}^{-}\right) >0$ for $\mu >0$ sufficiently large.
\end{theorem}

Combining the theorems \ref{t5}, \ref{t6} results, we have the following
multiplicity result.

\begin{corollary}
Suppose that $N=3,$ $2<p<4$ and conditions $\left( V_{1}\right) -\left(
V_{3}\right) $ and $\left( D_{1}\right) -\left( D_{4}\right) $ hold. Then
there exists $a_{0}>0$ such that for every $0<a<a_{0}$ and $0<\lambda <\left[
1-2\left( \frac{4-p}{4}\right) ^{2/p}\right] \lambda _{1}\left( f_{\Omega
}\right) $, Equation $\left( E_{\mu ,\lambda }\right) $ has two positive
solutions $u_{\mu }^{-}$ and $u_{\mu }^{+}$ satisfying $J_{\mu ,\lambda
}\left( u_{\mu }^{+}\right) <0<J_{\mu ,\lambda }\left( u_{\mu }^{-}\right) $
for $\mu >0$ sufficiently large.
\end{corollary}

Combining the theorems \ref{t5-2}, \ref{t6} results, we have the following
multiplicity result.

\begin{corollary}
Suppose that $N\geq 4,$ $2<p<2^{\ast }$ and conditions $\left( V_{1}\right)
-\left( V_{3}\right) ,\left( D_{1}\right) -\left( D_{2}\right) $ and $\left(
D_{4}\right) $ hold. Then there exists $a_{0}>0$ such that for every $%
0<a<a_{0}$ and $0<\lambda <\left[ 1-2\left( \frac{4-p}{4}\right) ^{2/p}%
\right] \lambda _{1}\left( f_{\Omega }\right) $, Equation $\left( E_{\mu
,\lambda }\right) $ has two positive solutions $u_{\mu }^{-}$ and $u_{\mu
}^{+}$ satisfying $J_{\mu ,\lambda }\left( u_{\mu }^{+}\right) <0<J_{\mu
,\lambda }\left( u_{\mu }^{-}\right) $ for $\mu >0$ sufficiently large.
\end{corollary}

To study the mainly Theorems, we shall establish their result by considering
minimization on two distinct components of the Nehari manifold corresponding
to Equation $\left( E_{\mu ,\lambda }\right) $. We are likewise interested
in the conditions of $M$ and $g$ that subsequently gives rise to the
non-existence and existence of positive solutions. Our focus here, however,
is on a given set of $M$ and $g$ so that it is possible to examine in detail
the number of solutions admitted subject to the variations of parameters
imbedded in these functions. A similar analysis has been carried out on
other elliptic equations with interesting results. Amann and Lopez-Gomez
\cite{ALG}, Binding \emph{et. al.} \cite{BDH1,BDH2}, and Brown and Zhang
\cite{BZ}, for example, studied the following semilinear boundary value
problem:

\begin{equation}
\left\{
\begin{array}{l}
-\Delta u=\lambda f_{\Omega }\left( x\right) u+b(x)\left\vert u\right\vert
^{p-2}u\text{ in }\Omega ; \\
u=0\text{ on }\partial \Omega ,%
\end{array}%
\right.  \label{1}
\end{equation}%
where $\Omega $ is a bounded domain with smooth boundary in $\mathbb{R}%
^{N},\lambda >0$ is a real parameter, $2<p<2^{\ast }$ and $f_{\Omega },b:%
\overline{\Omega }\rightarrow \mathbb{R}$ are smooth functions which change
sign in $\overline{\Omega }.$ In \cite{BDH1,BDH2} by using variational
methods, in Brown and Zhang \cite{BZ} by using Nehari manifold and fibrering
maps, and in Amann and Lopez-Gomez \cite{ALG} by using global bifurcation
theory. The existence and multiplicity results can be summarized as follows.
It is known that

\begin{itemize}
\item[$\left( A\right) $] there exists a positive solution to Equation $%
\left( \ref{1}\right) $ whenever $0<\lambda <\lambda _{1}\left( f_{\Omega
}\right) ;$

\item[$\left( B\right) $] if $\int_{\Omega }b\phi _{1}^{p}dx<0$, there
exists $\delta _{0}>0$ such that Equation $\left( \ref{1}\right) $ has at
least two positive solutions whenever $\lambda _{1}\left( f_{\Omega }\right)
<\lambda <\lambda _{1}\left( f_{\Omega }\right) +\delta _{0}.$
\end{itemize}

Results $(A)$ and $(B)$ can be understood in term of global bifurcation
theory as the sign of $\int_{\Omega }b\phi _{1}^{p}dx$ determines the
direction of bifurcation from the branch of zero solutions at the
bifurcation point at $\lambda =\lambda _{1}\left( f_{\Omega }\right) $ so
that bifurcation is to the left when $\int_{\Omega }b\phi _{1}^{p}dx>0$ and
to the right when $\int_{\Omega }b\phi _{1}^{p}dx<0$; the corresponding
bifurcation diagrams are shown in Fig.1 of \cite{BZ}. Furthermore, some
who's been done for this type of problem in $\mathbb{R}^{N}.$ We are only
aware of the works Chabrowski and Costa \cite{CC} and Costa and Tehrani \cite%
{CT} which also studied existence and multiplicity of positive solutions for
Schr\"{o}dinger type equations in $\mathbb{R}^{N}$%
\begin{equation}
-\Delta _{p}u=\lambda \widehat{f}\left( x\right) u+\widetilde{Q}%
(x)\left\vert u\right\vert ^{p-2}u\text{ in }\mathbb{R}^{N},  \label{5}
\end{equation}%
where $\lambda $ is a real parameter and $p<q<Np/(N-p)$ and $1<p<N.$ The
functions $\widetilde{f}$ and $\widetilde{Q}$ denote sign-changing
potentials such that $\widetilde{f}\in L^{N/p}(\mathbb{R}^{N})\cap L^{\infty
}(\mathbb{R}^{N})$ and $\widetilde{Q}\in L^{\infty }(\mathbb{R}^{N}).$ Let $%
\lambda _{1}\left( \widetilde{f}\right) $ denote the lowest positive
eigenvalue of $-\Delta _{p}$ and let $\varphi _{1}>0$ be the associated
eigenfunction. When $p=2,$ \cite{CT} by using the Mountain-Pass Theorem and
variational methods which under a slightly more general assumption on the
nonlinearity appearing on the right-hand side of $\left( \ref{5}\right) $.
However, in order to apply their result to Equation $\left( \ref{5}\right) $
they needed a \textquotedblleft thickness\textquotedblright\ condition on
the set $\Omega _{o}=\left\{ x:\widetilde{Q}\left( x\right) =0\right\} .$
\cite{CC} by using Nehari manifold and fibrering maps which under a limits
condition $\lim_{\left\vert x\right\vert \rightarrow \infty }\widetilde{Q}%
\left( x\right) =\widetilde{Q}_{\infty }<0.$ Their main result is almost the
same as in results $(A)$ and $(B)$ above. However, the principal eigenvalue
and eigenfunction are replaced by the problem $-\Delta u(x)=\lambda
\widetilde{f}(x)u(x)$ for$\;x\in \mathbb{R}^{N}.$

The approach to Equation $\left( E_{\mu ,\lambda }\right) $ has been
inspired by the papers of \cite{BZ,CC} without any condition on $\Omega _{o}$
or $\lim_{\left\vert x\right\vert \rightarrow \infty }Q\left( x\right)
=Q_{\infty }<0.$ Moreover, since Equation $\left( E_{\mu ,\lambda }\right) $
is on $\mathbb{R}^{N}$, its variational setting is characterized by a lack
of compactness. To overcome this difficulty we apply a simplified version of
the steep well method of \cite{BW} and concentration compactness principle
of \cite{Li}. Furthermore, the first eigenvalue of problem $-\Delta u+\mu
V\left( x\right) u=\lambda f\left( x\right) u$ in $\mathbb{R}^{N}$ is less
than $\lambda _{1}\left( f_{\Omega }\right) ,$ which indicates that the
original method at \cite{BZ,CC} cannot be directly applied, thus we provide
an approximation estimate of eigenvalue to prove that our main results.

The plan of the paper is as follows. In Section 2, we discuss the Nehari
manifold and examine carefully the connection between the Nehari manifold
and the fibrering maps. In Section 3, we establish the non-emptiness of
submanifolds and the proofs of the main theorems are given in the remaining
sections. In section 4, we discuss the Nehari manifold when $4<p<6.$ In
particular, we prove that Theorems \ref{t1}, \ref{t2}$.$ In Section 5, we
discuss the case when $p=4$ and prove that Theorems \ref{t3}, \ref{t4}. In
section 6, we discuss the case when $p<4$ and prove that Theorems \ref{t5}, %
\ref{t5-2} and \ref{t6}.

\section{Preliminaries}

In this section, we give the variational setting for Equation $\left( E_{\mu
,\lambda }\right) .$ Let%
\begin{equation*}
X=\left\{ u\in H^{1}\left( \mathbb{R}^{N}\right) \ |\ \int_{\mathbb{R}%
^{N}}Vu^{2}dx<\infty \right\}
\end{equation*}%
be equipped with the inner product and norm%
\begin{equation*}
\left\langle u,v\right\rangle =\int_{\mathbb{R}^{N}}\nabla u\nabla v+Vuvdx,\
\left\Vert u\right\Vert =\left\langle u,u\right\rangle ^{1/2}.
\end{equation*}%
For $\mu >0,$ we also need the following inner product and norm%
\begin{equation*}
\left\langle u,v\right\rangle _{\mu }=\int_{\mathbb{R}^{N}}\nabla u\nabla
v+\mu Vuvdx,\ \left\Vert u\right\Vert _{\mu }=\left\langle u,u\right\rangle
_{\mu }^{1/2}.
\end{equation*}%
It is clear that $\left\Vert \cdot \right\Vert \leq \left\Vert \cdot
\right\Vert _{\mu }$ for $\mu \geq 1$ and set $X_{\mu }=\left( X,\left\Vert
\cdot \right\Vert _{\mu }\right) .$ It follows from conditions $\left(
V_{1}\right) $ and $\left( V_{2}\right) $ and similar to the argument in
\cite{SW1}, one has%
\begin{equation*}
\int_{\mathbb{R}^{N}}(\left\vert \nabla u\right\vert ^{2}+u^{2})dx\leq
\left( 1+S^{-2}\left\vert \left\{ V<c\right\} \right\vert ^{\frac{2}{N}%
}\right) \left\Vert u\right\Vert _{\mu }^{2}
\end{equation*}%
for all $\mu \geq \mu _{0}:=\frac{S^{2}}{c}\left\vert \left\{ V<c\right\}
\right\vert ^{-\frac{2}{N}}$, which implies that the imbedding $X_{\mu
}\hookrightarrow H^{1}(\mathbb{R}^{N})$ is continuous. Moreover, for any $%
r\in \lbrack 2,2^{\ast }],$ there holds
\begin{equation}
\int_{\mathbb{R}^{N}}\left\vert u\right\vert ^{r}dx\leq S^{-r}\left\vert
\left\{ V<c\right\} \right\vert ^{\frac{2^{\ast }-r}{2^{\ast }}}\left\Vert
u\right\Vert _{\mu }^{r}\text{ for }\mu \geq \mu _{0}.  \label{11}
\end{equation}

Because the energy functional $J_{\mu ,\lambda }$ is not bounded below on $%
X, $ it is useful to consider the functional on the Nehari manifold (see
\cite{N})
\begin{equation*}
\mathbf{N}_{\mu ,\lambda }=\left\{ u\in X\setminus \left\{ 0\right\}
:\left\langle J_{\mu ,\lambda }^{\prime }\left( u\right) ,u\right\rangle
=0\right\} .
\end{equation*}%
Thus, $u\in \mathbf{N}_{\mu ,\lambda }$ if and only if
\begin{equation*}
a\left\Vert u\right\Vert _{D^{1,2}}^{4}+\left\Vert u\right\Vert _{\mu
}^{2}=\int_{\mathbb{R}^{N}}Q\left\vert u\right\vert ^{p}dx+\lambda \int_{%
\mathbb{R}^{N}}fu^{2}dx.
\end{equation*}%
Note that $\mathbf{N}_{\mu ,\lambda }$ contains every nonzero solution of
Equation $\left( E_{\mu ,\lambda }\right) .$ It is useful to understand $%
\mathbf{N}_{\mu ,\lambda }$ in terms of the stationary points of mappings of
the form $h_{u}(t)=J_{\mu ,\lambda }(tu)(t>0).$ Such a map is known as the
fibrering map. It was introduced by Dr\'{a}bek and Pohozaev \cite{DP}, and
further discussed by Brown and Zhang \cite{BZ}. It is clear that, if $u$ is
a local minimizer of $J_{\mu ,\lambda }$, then $h_{u}$ has a local minimum
at $t=1$. Thus, $tu\in \mathbf{N}_{\mu ,\lambda }$ if and only if $%
h_{u}^{\prime }(t)=0$ for $u\in X\setminus \{0\}.$ Thus points in $\mathbf{N}%
_{\mu ,\lambda }$ correspond to stationary points of the maps $h_{u}$ and so
it is natural to divide $\mathbf{N}_{\mu ,\lambda }$ into three subsets $%
\mathbf{N}_{\mu ,\lambda }^{+}$, $\mathbf{N}_{\mu ,\lambda }^{-}$ and $%
\mathbf{N}_{\mu ,\lambda }^{0}$ corresponding to local minima, local maxima
and points of inflexion of fibrering maps. We have%
\begin{equation*}
h_{u}^{\prime }(t)=at^{3}\left\Vert u\right\Vert _{D^{1,2}}^{4}+t\left(
\left\Vert u\right\Vert _{\mu }^{2}-\lambda \int_{\mathbb{R}%
^{N}}fu^{2}dx\right) -t^{p-1}\int_{\mathbb{R}^{N}}Q\left\vert u\right\vert
^{p}dx
\end{equation*}%
and%
\begin{equation*}
h_{u}^{\prime \prime }(t)=3at^{2}\left\Vert u\right\Vert
_{D^{1,2}}^{4}+\left( \left\Vert u\right\Vert _{\mu }^{2}-\lambda \int_{%
\mathbb{R}^{N}}fu^{2}dx\right) -\left( p-1\right) t^{p-2}\int_{\mathbb{R}%
^{N}}Q\left\vert u\right\vert ^{p}dx.
\end{equation*}%
Hence if we define%
\begin{gather*}
\mathbf{N}_{\mu ,\lambda }^{+}=\left\{ u\in \mathbf{N}_{\mu ,\lambda
}:h_{u}^{\prime \prime }(1)>0\right\} ; \\
\mathbf{N}_{\mu ,\lambda }^{0}=\left\{ u\in \mathbf{N}_{\mu ,\lambda
}:h_{u}^{\prime \prime }(1)=0\right\} ; \\
\mathbf{N}_{\mu ,\lambda }^{-}=\left\{ u\in \mathbf{N}_{\mu ,\lambda
}:h_{u}^{\prime \prime }(1)<0\right\} ,
\end{gather*}%
which indicates that for $u\in \mathbf{N}_{\mu ,\lambda }$, we have $%
h_{u}^{\prime }(1)=0$ and $u\in \mathbf{N}_{\mu ,\lambda }^{+},\mathbf{N}%
_{\mu ,\lambda }^{0},\mathbf{N}_{\mu ,\lambda }^{-}$ if $h_{u}^{\prime
\prime }(1)>0,h_{u}^{\prime \prime }(1)=0,h_{u}^{\prime \prime }(1)<0$,
respectively. Note that for all $u\in \mathbf{N}_{\mu ,\lambda },$
\begin{eqnarray}
h_{u}^{\prime \prime }(1) &=&-\left( p-2\right) \left( \left\Vert
u\right\Vert _{\mu }^{2}-\lambda \int_{\mathbb{R}^{N}}fu^{2}dx\right)
-a\left( p-4\right) \left\Vert u\right\Vert _{D^{1,2}}^{4}  \notag \\
&=&2a\left\Vert u\right\Vert _{D^{1,2}}^{4}-\left( p-2\right) \int_{\mathbb{R%
}^{N}}Q\left\vert u\right\vert ^{p}dx  \notag \\
&=&-2\left( \left\Vert u\right\Vert _{\mu }^{2}-\lambda \int_{\mathbb{R}%
^{N}}fu^{2}dx\right) -\left( p-4\right) \int_{\mathbb{R}^{N}}Q\left\vert
u\right\vert ^{p}dx.  \label{2.2}
\end{eqnarray}%
Now, we define%
\begin{equation*}
\Lambda _{\mu }^{+}=\left\{ u\in X:\left\Vert u\right\Vert _{\mu
}=1,\left\Vert u\right\Vert _{\mu }^{2}-\lambda \int_{\mathbb{R}%
^{N}}fu^{2}dx>0\right\}
\end{equation*}%
and $\Lambda _{\mu }^{-}$ and $\Lambda _{\mu }^{0}$ similarly by replacing $%
> $ by $<$ and $=$ respectively. We also define%
\begin{equation*}
\Theta _{\mu }^{+}\left( p\right) =\left\{ u\in X:\left\Vert u\right\Vert
_{\mu }=1,\Phi _{p}\left( u\right) >0\right\}
\end{equation*}%
and $\Theta _{\mu }^{-}\left( p\right) $ and $\Theta _{\mu }^{0}\left(
p\right) $ analogously, where
\begin{equation*}
\Phi _{p}\left( u\right) =\left\{
\begin{array}{ll}
\int_{\mathbb{R}^{N}}Q|u|^{p}dx & \text{ for }2<p<2^{\ast }\text{ and }p\neq
4, \\
\int_{\mathbb{R}^{N}}Q|u|^{p}dx-a\left\Vert u\right\Vert _{D^{1,2}}^{4} &
\text{for }p=4.%
\end{array}%
\right.
\end{equation*}%
Thus, if $u\in \Lambda _{\mu }^{+}\cap \Theta _{\mu }^{+}\left( p\right) $
and $p\geq 4,h_{u}(t)>0$ for $t$ small and positive but $h_{u}(t)\rightarrow
-\infty $ as $t\rightarrow \infty ;$ also $h_{u}(t)$ has a unique (maximum)
stationary point $t_{\max }(u)$ and $t_{\max }(u)u\in \mathbf{N}_{\mu
,\lambda }^{-}$. Similarly, if $u\in \Lambda _{\mu }^{-}\cap \Theta _{\mu
}^{-}\left( p\right) $ and $2<p<2^{\ast },h_{u}(t)<0$ for $t$ small and
positive, $h_{u}(t)\rightarrow \infty $ as $t\rightarrow \infty $ and $%
h_{u}(t)$ has a unique minimum $t_{\min }(u)$ so that $t_{\min }(u)u\in
\mathbf{N}_{\mu ,\lambda }^{+}$. Finally, if $u\in \Lambda _{\mu }^{+}\cap
\Theta _{\mu }^{-}\left( p\right) $, $h_{u}$ is strictly increasing for all $%
t>0.$ Thus, we have the following results.

\begin{lemma}
\label{g2-2}Suppose that $N=3$ and $4<p<6.$ If $\overline{\Lambda _{\mu }^{-}%
}\cap \overline{\Theta _{\mu }^{+}}\left( p\right) =\emptyset $ and $u\in
X_{\mu }\backslash \{0\}$, then\newline
$(i)$ a multiple of $u$ lies is $\mathbf{N}_{\mu ,\lambda }^{-}$ if and only
if $\frac{u}{\Vert u\Vert _{\mu }}$ lies in $\Lambda _{\mu }^{+}\cap \Theta
_{\mu }^{+}\left( p\right) ;$\newline
$(ii)$ a multiple of $u$ lies is $\mathbf{N}_{\mu ,\lambda }^{+}$ if and
only if $\frac{u}{\Vert u\Vert _{\mu }}$ lies in $\Lambda _{\mu }^{-}\cap
\Theta _{\mu }^{-}\left( p\right) ;$\newline
$(iii)$ when $u\in \Lambda _{\mu }^{+}\cap \Theta _{\mu }^{-}\left( p\right)
$, no multiple of $u$ lies in $\mathbf{N}_{\mu ,\lambda }.$
\end{lemma}

\begin{lemma}
\label{g2-3}Suppose that $N=3$ and $p=4.$ If $u\in X_{\mu }\backslash \{0\}$%
, then\newline
$(i)$ a multiple of $u$ lies is $\mathbf{N}_{\mu ,\lambda }^{-}$ if and only
if $\frac{u}{\Vert u\Vert _{\mu }}$ lies in $\Lambda _{\mu }^{+}\cap \Theta
_{\mu }^{+}\left( p\right) ;$\newline
$(ii)$ a multiple of $u$ lies is $\mathbf{N}_{\mu ,\lambda }^{+}$ if and
only if $\frac{u}{\Vert u\Vert _{\mu }}$ lies in $\Lambda _{\mu }^{-}\cap
\Theta _{\mu }^{-}\left( p\right) ;$\newline
$(iii)$ when $u\in \Lambda _{\mu }^{+}\cap \Theta _{\mu }^{-}\left( p\right)
$ or $\Lambda _{\mu }^{-}\cap \Theta _{\mu }^{+}\left( p\right) $, no
multiple of $u$ lies in $\mathbf{N}_{\mu ,\lambda }.$
\end{lemma}

\begin{lemma}
\label{g2-6}Suppose that $N\geq 3$ and $2<p<\min \left\{ 4,2^{\ast }\right\}
.$ If $u\in X_{\mu }\backslash \{0\}$, then\newline
$(i)$ if $\frac{u}{\Vert u\Vert _{\mu }}$ lies in $\overline{\Lambda _{\mu
}^{-}}\cap \Theta _{\mu }^{+}\left( p\right) $ or $\Lambda _{\mu }^{-}\cap
\overline{\Theta _{\mu }^{-}}\left( p\right) ,$ then a multiple of $u$ lies
is $\mathbf{N}_{\mu ,\lambda }^{+};$\newline
$(ii)$ when $u\in \Lambda _{\mu }^{+}\cap \Theta _{\mu }^{-}\left( p\right) $%
, no multiple of $u$ lies in $\mathbf{N}_{\mu ,\lambda }.$
\end{lemma}

The following Lemma shows that minimizers on $\mathbf{N}_{\mu ,\lambda }$
are critical points for $J_{\mu ,\lambda }$ in $X.$

\begin{lemma}
\label{g2-1}Suppose that $u_{0}$ is a local minimizer for $J_{\mu ,\lambda }$
on $\mathbf{N}_{\mu ,\lambda }$ and that $u_{0}\notin \mathbf{N}_{\mu
,\lambda }^{0}$. Then $J_{\mu ,\lambda }^{\prime }(u_{0})=0.$
\end{lemma}

\begin{proof}
The proof of Lemma \ref{g2-1} is essentially same as that in Brown and Zhang
\cite[Theorem 2.3]{BZ} (or see Binding et al. \cite{BDH1}), so we omit it
here.
\end{proof}

Finally, we investigate the compactness condition for the functional $J_{\mu
,\lambda }.$ Here we call that a $C^{1}$-functional $J_{\mu ,\lambda }$
satisfies Palais-Smale condition at level $\beta $ ((PS)$_{\beta }$%
-condition for short) in $\mathbf{N}_{\mu ,\lambda },$ if any sequence $%
\{u_{n}\}\subset \mathbf{N}_{\mu ,\lambda }$ is uniformly bounded which
satisfy $J_{\mu ,\lambda }\left( u_{n}\right) =\beta +o\left( 1\right) $ and
$J_{\mu ,\lambda }^{\prime }\left( u_{n}\right) =o\left( 1\right) $ has a
convergent subsequence.

\begin{proposition}
\label{p1}Suppose that conditions $(V_{1})-(V_{2})$ and $(D_{1})-(D_{2})$
hold. Then there exists $\widehat{D}_{0}\in \mathbb{R}$ independent of $\mu $
such that $J_{\mu ,\lambda }$ satisfies (PS)$_{\beta }$--condition in $%
\mathbf{N}_{\mu ,\lambda }$ with $\beta <\widehat{D}_{0}$ for $\mu >0$
sufficiently large.
\end{proposition}

\begin{proof}
Let $\{u_{n}\}\subset \mathbf{N}_{\mu ,\lambda }$ be a (PS)$_{\beta }$%
--sequence for $J_{\mu ,\lambda }$ with $\beta <\widehat{D}_{0}.$ Since $%
\left\{ u_{n}\right\} \subset X_{\mu }$ is uniformly bounded, i.e., there
exists $d_{0}>0$ such that
\begin{equation}
\left\Vert u_{n}\right\Vert _{\mu }<d_{0}.  \label{8-4}
\end{equation}%
Then there exist a subsequence $\{u_{n}\}$ and $u_{0}$ in $X_{\mu }$ such
that%
\begin{eqnarray*}
u_{n} &\rightharpoonup &u_{0}\text{ weakly in }X_{\mu }; \\
u_{n} &\rightarrow &u_{0}\text{ strongly in }L_{loc}^{r}(\mathbb{R}^{N})%
\text{ for }2\leq r<2^{\ast }.
\end{eqnarray*}%
Then by condition $\left( D_{1}\right) ,$%
\begin{equation}
\lim_{n\rightarrow \infty }\int_{\mathbb{R}^{N}}fu_{n}^{2}dx=\int_{\mathbb{R}%
^{N}}fu_{0}^{2}dx.  \label{8-1}
\end{equation}%
Now, we prove that $u_{n}\rightarrow u_{0}$ strongly in $X_{\mu }.$ Let $%
v_{n}=u_{n}-u_{0}.$ By $(\ref{8-4})$ one has%
\begin{equation*}
\left\Vert u_{0}\right\Vert _{\mu }\leq \liminf_{n\rightarrow \infty
}\left\Vert u_{n}\right\Vert _{\mu }\leq d_{0},
\end{equation*}%
leading to%
\begin{equation}
\left\Vert v_{n}\right\Vert _{\mu }=\left\Vert u_{n}-u_{0}\right\Vert _{\mu
}\leq 2d_{0}.  \label{8-5}
\end{equation}%
It follows from the condition $(V_{1})$ that%
\begin{equation*}
\int_{\mathbb{R}^{N}}v_{n}^{2}dx=\int_{\left\{ V\geq c\right\}
}v_{n}^{2}dx+\int_{\left\{ V<c\right\} }v_{n}^{2}dx\leq \frac{1}{\mu c}%
\left\Vert v_{n}\right\Vert _{\mu }^{2}+o\left( 1\right) ,
\end{equation*}%
which implies that
\begin{eqnarray}
\int_{\mathbb{R}^{N}}\left\vert v_{n}\right\vert ^{p}dx &\leq &\left( \frac{1%
}{\mu c}\left\Vert v_{n}\right\Vert _{\mu }^{2}\right) ^{\frac{2^{\ast }-p}{%
2^{\ast }-2}}\left( S^{-2^{\ast }}\left\Vert v_{n}\right\Vert
_{D^{1,2}}^{2^{\ast }}\right) ^{\frac{p-2}{2^{\ast }-2}}+o(1)  \notag \\
&\leq &\left( \frac{1}{\mu c}\right) ^{\frac{\left( 2^{\ast }-p\right)
\left( N-2\right) }{4}}S^{-\frac{N\left( p-2\right) }{2}}\left\Vert
v_{n}\right\Vert _{\mu }^{p}+o(1),  \label{8-6}
\end{eqnarray}%
where we have used the H\"{o}lder and Sobolev inequalities. Moreover, by
Brezis-Lieb Lemma \cite{BL}, we have%
\begin{equation}
\int_{\mathbb{R}^{N}}Q\left\vert v_{n}\right\vert ^{p}dx=\int_{\mathbb{R}%
^{N}}Q\left\vert u_{n}\right\vert ^{p}dx-\int_{\mathbb{R}^{N}}Q\left\vert
u_{0}\right\vert ^{p}dx+o(1).  \label{8-7}
\end{equation}%
Since the sequence $\left\{ u_{n}\right\} $ is bounded in $X_{\mu },$ there
exists a constant $A>0$ such that%
\begin{equation*}
\int_{\mathbb{R}^{N}}\left\vert \nabla u_{n}\right\vert ^{2}dx\rightarrow A%
\text{ as }n\rightarrow \infty .
\end{equation*}%
It indicates that for any $\varphi \in C_{0}^{\infty }(\mathbb{R}^{N}),$%
\begin{eqnarray*}
o(1) &=&\left\langle J_{\mu ,\lambda }^{\prime }\left( u_{n}\right) ,\varphi
\right\rangle \\
&\rightarrow &\int_{\mathbb{R}^{N}}\nabla u_{0}\nabla \varphi dx+\int_{%
\mathbb{R}^{N}}\mu Vu_{0}\varphi dx+aA\int_{\mathbb{R}^{N}}\nabla
u_{0}\nabla \varphi dx \\
&&-\int_{\mathbb{R}^{N}}fu_{0}\varphi dx-\int_{\mathbb{R}^{N}}Q\left\vert
u_{0}\right\vert ^{p-2}u_{0}\varphi dx\text{ as }n\rightarrow \infty ,
\end{eqnarray*}%
which shows that%
\begin{equation}
\left\Vert u_{0}\right\Vert _{\mu }^{2}+aA\left\Vert u_{0}\right\Vert
_{D^{1,2}}^{2}-\int_{\mathbb{R}^{N}}fu_{0}^{2}dx-\int_{\mathbb{R}%
^{N}}Q\left\vert u_{0}\right\vert ^{p}dx=0.  \label{8-8}
\end{equation}%
Note that%
\begin{equation}
\left\Vert u_{n}\right\Vert _{\mu }^{2}+a\left\Vert u_{n}\right\Vert
_{D^{1,2}}^{4}-\int_{\mathbb{R}^{N}}fu_{n}^{2}dx-\int_{\mathbb{R}%
^{N}}Q\left\vert u_{n}\right\vert ^{p}dx=0.  \label{8-9}
\end{equation}%
Then by $\left( \ref{8-1}\right) $ and $(\ref{8-6})-(\ref{8-9})$ one has%
\begin{eqnarray}
o\left( 1\right) &=&\left\Vert v_{n}\right\Vert _{\mu }^{2}+a\left\Vert
u_{n}\right\Vert _{D^{1,2}}^{4}-aA\left\Vert u_{0}\right\Vert
_{D^{1,2}}^{2}-\int_{\mathbb{R}^{N}}Q\left\vert v_{n}\right\vert ^{p}dx
\notag \\
&=&\left\Vert v_{n}\right\Vert _{\mu }^{2}+a\left\Vert u_{n}\right\Vert
_{D^{1,2}}^{2}\left( \left\Vert u_{n}\right\Vert _{D^{1,2}}^{2}-\left\Vert
u_{0}\right\Vert _{D^{1,2}}^{2}\right) -\int_{\mathbb{R}^{N}}Q\left\vert
v_{n}\right\vert ^{p}dx  \notag \\
&=&\left\Vert v_{n}\right\Vert _{\mu }^{2}+a\left\Vert u_{n}\right\Vert
_{D^{1,2}}^{2}\left\Vert v_{n}\right\Vert _{D^{1,2}}^{2}-\int_{\mathbb{R}%
^{N}}Q\left\vert v_{n}\right\vert ^{p}dx.  \label{8-10}
\end{eqnarray}%
It follows from $(\ref{11}),(\ref{8-5}),(\ref{8-6}),(\ref{8-10})$ that%
\begin{eqnarray*}
o\left( 1\right) &=&\left\Vert v_{n}\right\Vert _{\mu }^{2}+a\left\Vert
u_{n}\right\Vert _{D^{1,2}}^{2}\left\Vert v_{n}\right\Vert
_{D^{1,2}}^{2}-\int_{\mathbb{R}^{N}}Q\left\vert v_{n}\right\vert ^{p}dx \\
&\geq &\left\Vert v_{n}\right\Vert _{\mu }^{2}-Q_{\max }\left( \int_{\mathbb{%
R}^{N}}|v_{n}|^{p}dx\right) ^{\frac{p-2}{p}}\left( \int_{\mathbb{R}%
^{N}}|v_{n}|^{p}dx\right) ^{\frac{2}{p}} \\
&\geq &\left[ 1-Q_{\max }\left[ \frac{(2d_{0})^{\frac{p}{2}}\left\vert
\left\{ V<c\right\} \right\vert ^{\frac{2^{\ast }-p}{2^{\ast }}}}{S^{p}}%
\right] ^{\frac{p-2}{p}}\left( \frac{1}{\mu c}\right) ^{\frac{\left( 2^{\ast
}-p\right) \left( N-2\right) }{2p}}S^{-\frac{N\left( p-2\right) }{p}}\right]
\left\Vert v_{n}\right\Vert _{\mu }^{2}+o\left( 1\right) ,
\end{eqnarray*}%
which implies that $v_{n}\rightarrow 0$ strongly in $X_{\mu }$ for $\mu >0$
sufficiently large. Consequently, this completes the proof.
\end{proof}

\section{Non-emptiness of submanifolds}

First, we need the following result.

\begin{theorem}
\label{g4-0}Let $\mu _{n}\rightarrow \infty $ as $n\rightarrow \infty $ and $%
\{v_{n}\}\subset X$ with $\Vert v_{n}\Vert _{\mu _{n}}\leq c_{0}$ for some $%
c_{0}>0.$ Then there exist subsequence $\left\{ v_{n}\right\} $ and $%
v_{0}\in H_{0}^{1}\left( \Omega \right) $ such that $v_{n}\rightharpoonup
v_{0}$ in $X$ and $v_{n}\rightarrow v_{0}$ in $L^{r}\left( \mathbb{R}%
^{N}\right) $ for all $2\leq r<2^{\ast }.$
\end{theorem}

\begin{proof}
Since $\left\Vert v_{n}\right\Vert \leq \left\Vert v_{n}\right\Vert _{\mu
_{n}}\leq c_{0}.$ We may assume that there exists $v_{0}\in X$ such that
\begin{eqnarray*}
v_{n} &\rightharpoonup &v_{0}\text{ in }X, \\
v_{n} &\rightarrow &v_{0}\text{ a.e. in }\mathbb{R}^{N}, \\
v_{n} &\rightarrow &v_{0}\text{ in }L_{loc}^{r}\left( \mathbb{R}^{N}\right)
\text{ for }2\leq r<2^{\ast }.
\end{eqnarray*}%
By Fatou's Lemma, we have
\begin{equation*}
\int_{\mathbb{R}^{N}}Vv_{0}^{2}dx\leq \liminf_{n\rightarrow \infty }\int_{%
\mathbb{R}^{N}}Vv_{n}^{2}dx\leq \liminf_{n\rightarrow \infty }\frac{%
\left\Vert v_{n}\right\Vert _{\mu _{n}}^{2}}{\mu _{n}}=0,
\end{equation*}%
this implies that $\int_{\mathbb{R}^{N}}Vv_{0}^{2}dx=0$ or $v_{0}=0$ a.e. in
$\mathbb{R}^{N}\setminus \overline{\Omega }$ and $v_{0}\in H_{0}^{1}\left(
\Omega \right) $ by condition $\left( V_{3}\right) .$ We now show that $%
v_{n}\rightarrow v_{0}$ in $L^{p}\left( \mathbb{R}^{N}\right) .$ Suppose on
the contrary. Then by Lions vanishing lemma (see \cite[Lemma I.1]{Li} or
\cite[Lemma 1.21]{Wi}), there exist $d_{0}>0,R_{0}>0$ and $x_{n}\in \mathbb{R%
}^{N}$ such that
\begin{equation*}
\int_{B\left( x_{n},R_{0}\right) }\left( v_{n}-v_{0}\right) ^{2}dx\geq d_{0}.
\end{equation*}%
Moreover, $x_{n}\rightarrow \infty $, and hence, $\left\vert B\left(
x_{n},R_{0}\right) \cap \left\{ x\in \mathbb{R}^{N}:V<c\right\} \right\vert
\rightarrow 0$. By the H\"{o}lder inequality, we have
\begin{equation*}
\int_{B\left( x_{n},R_{0}\right) \cap \left\{ V<c\right\} }\left(
v_{n}-v_{0}\right) ^{2}dx\rightarrow 0.
\end{equation*}%
Consequently,
\begin{eqnarray*}
c_{0} &\geq &\left\Vert v_{n}\right\Vert _{\mu _{n}}^{2}\geq \mu
_{n}c\int_{B\left( x_{n},R_{0}\right) \cap \left\{ V\geq c\right\}
}v_{n}^{2}dx=\mu _{n}c\int_{B\left( x_{n},R_{0}\right) \cap \left\{ V\geq
c\right\} }\left( v_{n}-v_{0}\right) ^{2}dx \\
&=&\mu _{n}c\left( \int_{B\left( x_{n},R_{0}\right) }\left(
v_{n}-v_{0}\right) ^{2}dx-\int_{B\left( x_{n},R_{0}\right) \cap \left\{
V<c\right\} }\left( v_{n}-v_{0}\right) ^{2}dx\right) \\
&\rightarrow &\infty ,
\end{eqnarray*}%
which a contradiction. Thus, $v_{n}\rightarrow v_{0}$ in $L^{r}\left(
\mathbb{R}^{N}\right) $ for all $2\leq r<2^{\ast }.$ This completes the
proof.
\end{proof}

Next, we consider the following eigenvalue problem%
\begin{equation}
-\Delta u(x)+\mu V\left( x\right) u\left( x\right) =\lambda f(x)u(x)\text{
for}\;x\in \mathbb{R}^{N}.  \label{4}
\end{equation}%
We can approach this problem by a direct method and attempt to obtain
nontrivial solutions of problem $\left( \ref{4}\right) $ as relative minima
of the functional%
\begin{equation*}
I_{\mu }\left( u\right) =\frac{1}{2}\int_{\mathbb{R}^{N}}\left\vert \nabla
u\right\vert ^{2}+\mu Vu^{2}dx,
\end{equation*}%
on the unit sphere in $\mathbb{B}=\left\{ u\in X:\int_{\mathbb{R}%
^{N}}fu^{2}dx=1\right\} .$ Equivalently, we may seek to minimize a quotient
as follows%
\begin{equation}
\widetilde{\lambda }_{1,\mu }\left( f\right) =\inf_{u\in X\backslash \left\{
0\right\} }\frac{\int_{\mathbb{R}^{N}}\left\vert \nabla u\right\vert
^{2}+\mu Vu^{2}dx}{\int_{\mathbb{R}^{N}}fu^{2}dx}.  \label{3}
\end{equation}%
Then, by $\left( \ref{11}\right) ,$
\begin{equation*}
\frac{\int_{\mathbb{R}^{N}}\left\vert \nabla u\right\vert ^{2}+\mu Vu^{2}dx}{%
\int_{\mathbb{R}^{N}}fu^{2}dx}\geq \frac{S^{2}}{\left\Vert f\right\Vert
_{\infty }\left\vert \left\{ V<c\right\} \right\vert ^{\frac{2}{3}}}\text{
for all }\mu \geq \mu _{0},
\end{equation*}%
this implies that $\widetilde{\lambda }_{1,\mu }\left( f\right) \geq \frac{%
S^{2}}{\left\Vert f\right\Vert _{\infty }\left\vert \left\{ V<c\right\}
\right\vert ^{\frac{2}{3}}}>0.$ Moreover, by condition $\left( V_{3}\right)
, $%
\begin{equation*}
\inf_{u\in X\backslash \left\{ 0\right\} }\frac{\int_{\mathbb{R}%
^{N}}\left\vert \nabla u\right\vert ^{2}+\mu Vu^{2}dx}{\int_{\mathbb{R}%
^{N}}fu^{2}dx}\leq \inf_{u\in H_{0}^{1}\left( \Omega \right) \backslash
\left\{ 0\right\} }\frac{\int_{\mathbb{R}^{N}}\left\vert \nabla u\right\vert
^{2}+\mu Vu^{2}dx}{\int_{\mathbb{R}^{N}}fu^{2}dx}=\inf_{u\in H_{0}^{1}\left(
\Omega \right) \backslash \left\{ 0\right\} }\frac{\int_{\Omega }\left\vert
\nabla u\right\vert ^{2}}{\int_{\Omega }f_{\Omega }u^{2}dx},
\end{equation*}%
which indicates that $\widetilde{\lambda }_{1,\mu }\left( f\right) \leq
\lambda _{1}\left( f\right) $ for all $\mu \geq \mu _{0}.$ Then we have the
following results.

\begin{lemma}
\label{g3-0}For each $\mu \geq \mu _{0}$ there exists a positive function $%
\varphi _{\mu }\in X$ with $\int_{\mathbb{R}^{N}}f\varphi _{\mu }^{2}dx=1$
such that
\begin{equation*}
\widetilde{\lambda }_{1,\mu }\left( f\right) =\int_{\mathbb{R}%
^{N}}\left\vert \nabla \varphi _{\mu }\right\vert ^{2}+\mu V\varphi _{\mu
}^{2}dx<\lambda _{1}\left( f_{\Omega }\right) .
\end{equation*}%
Furthermore, $\widetilde{\lambda }_{1,\mu }\left( f\right) \rightarrow
\lambda _{1}^{-}\left( f_{\Omega }\right) $ and $\varphi _{\mu }\rightarrow
\phi _{1}$ as $\mu \rightarrow \infty ,$ where $\phi _{1}$ is positive
principal eigenfunction of problem $\left( \ref{2}\right) .$
\end{lemma}

\begin{proof}
Let $\left\{ u_{n}\right\} \subset X$ with $\int_{\mathbb{R}%
^{N}}fu_{n}^{2}dx=1$ be a minimizing sequence of $\left( \ref{3}\right) ,$
that is%
\begin{equation*}
\int_{\mathbb{R}^{N}}\left\vert \nabla u_{n}\right\vert ^{2}+\mu
Vu_{n}^{2}dx\rightarrow \widetilde{\lambda }_{1,\mu }\left( f\right) \text{
as }n\rightarrow \infty .
\end{equation*}%
Since $\widetilde{\lambda }_{1,\mu }\left( f\right) \leq \lambda _{1}\left(
f_{\Omega }\right) $ for all $\mu \geq \mu _{0},$ there exists $C_{0}>0$
independent of $\mu $ such that $\left\Vert u_{n}\right\Vert _{\mu }\leq
C_{0}.$ Thus, there exist a subsequence $\left\{ u_{n}\right\} $ and $%
\varphi _{\mu }\in X$ such that
\begin{eqnarray*}
u_{n} &\rightharpoonup &\varphi _{\mu }\text{ in }X_{\mu }, \\
u_{n} &\rightarrow &\varphi _{\mu }\text{ a.e. in }\mathbb{R}^{N}, \\
u_{n} &\rightarrow &\varphi _{\mu }\text{ in }L_{loc}^{r}\left( \mathbb{R}%
^{N}\right) \text{ for }2\leq r<2^{\ast }.
\end{eqnarray*}%
Moreover, by condition $\left( D_{1}\right) ,$%
\begin{equation*}
\int_{\mathbb{R}^{N}}fu_{n}^{2}dx\rightarrow \int_{\mathbb{R}^{N}}f\varphi
_{\mu }^{2}dx=1.
\end{equation*}%
Now we show that $u_{n}\rightarrow \varphi _{\mu }$ in $X_{\mu }.$ Suppose
on the contrary. Then
\begin{equation*}
\int_{\mathbb{R}^{N}}\left\vert \nabla \varphi _{\mu }\right\vert ^{2}+\mu
V\varphi _{\mu }^{2}dx<\liminf_{n\rightarrow \infty }\int_{\mathbb{R}%
^{N}}\left\vert \nabla u_{n}\right\vert ^{2}+\mu Vu_{n}^{2}dx=\widetilde{%
\lambda }_{1,\mu }\left( f\right) ,
\end{equation*}%
which is impossible. Thus, $u_{n}\rightarrow \varphi _{\mu }$ in $X_{\mu },$
which implies that $\int_{\mathbb{R}^{N}}f\varphi _{\mu }^{2}dx=1$ and $%
\int_{\mathbb{R}^{N}}\left\vert \nabla \varphi _{\mu }\right\vert ^{2}+\mu
V\varphi _{\mu }^{2}dx=\widetilde{\lambda }_{1,\mu }\left( f\right) .$ Since
$\left\vert \varphi _{\mu }\right\vert \in X$ and
\begin{equation*}
\widetilde{\lambda }_{1,\mu }\left( f\right) =\int_{\mathbb{R}%
^{N}}\left\vert \nabla \varphi _{\mu }\right\vert ^{2}+\mu V\varphi _{\mu
}^{2}dx=\int_{\mathbb{R}^{N}}\left\vert \nabla \left\vert \varphi _{\mu
}\right\vert \right\vert ^{2}+\mu V\left\vert \varphi _{\mu }\right\vert
^{2}dx,
\end{equation*}%
by the maximum principle, we may assume that $\varphi _{\mu }$ is positive
eigenfunction of problem $\left( \ref{4}\right) .$ Moreover, by the Harnack
inequality due to Trudinger \cite{Tr}, we must have $\widetilde{\lambda }%
_{1,\mu }\left( f\right) <\lambda _{1}\left( f_{\Omega }\right) .$ Now, by
the definition of $\widetilde{\lambda }_{1,\mu }\left( f\right) $, there
holds $\widetilde{\lambda }_{1,\mu _{1}}\left( f\right) \leq \widetilde{%
\lambda }_{1,\mu _{2}}\left( f\right) $ for $\mu _{1}<\mu _{2}.$ Hence, for
any sequence $\mu _{n}\rightarrow \infty ,$ let $\varphi _{n}:=\varphi _{\mu
_{n}}$ be the minimizer of $\widetilde{\lambda }_{1,\mu _{n}}\left( f\right)
$. Then $\int_{\mathbb{R}^{N}}f\varphi _{n}^{2}dx=1$ and%
\begin{equation*}
\widetilde{\lambda }_{1,\mu _{n}}\left( f\right) =\int_{\mathbb{R}%
^{N}}\left\vert \nabla \varphi _{n}\right\vert ^{2}+\mu _{n}V\varphi
_{n}^{2}dx<\lambda _{1}\left( f_{\Omega }\right) ,
\end{equation*}%
that
\begin{equation*}
\widetilde{\lambda }_{1,\mu _{n}}\left( f\right) \rightarrow d_{0}\leq
\lambda _{1}\left( f_{\Omega }\right) \text{ for some }d_{0}>0
\end{equation*}%
and
\begin{equation*}
\left\Vert \varphi _{n}\right\Vert \leq \left\Vert \varphi _{n}\right\Vert
_{\mu _{n}}\leq \sqrt{\lambda _{1}\left( f_{\Omega }\right) },\text{ for }n%
\text{ sufficiently large.}
\end{equation*}%
Thus, by Theorem \ref{g4-0}, we may assume that there exists $\varphi
_{0}\in H_{0}^{1}\left( \Omega \right) $ such that $\varphi
_{n}\rightharpoonup \varphi _{0}$ in $X$ and $\varphi _{n}\rightarrow
\varphi _{0}$ in $L^{r}\left( \mathbb{R}^{N}\right) $ for all $2\leq
r<2^{\ast }.$ Then%
\begin{equation*}
\int_{\Omega }\left\vert \nabla \varphi _{0}\right\vert ^{2}dx\leq
\liminf_{n\rightarrow \infty }\int_{\mathbb{R}^{N}}\left\vert \nabla \varphi
_{n}\right\vert ^{2}+\mu _{n}V\varphi _{n}^{2}dx=d_{0}
\end{equation*}%
and
\begin{equation*}
\lim_{n\rightarrow \infty }\int_{\mathbb{R}^{N}}f\varphi
_{n}^{2}dx=\int_{\Omega }f_{\Omega }\varphi _{0}^{2}dx=1.
\end{equation*}%
Since $d_{0}\leq \lambda _{1}\left( f_{\Omega }\right) $ and $\lambda
_{1}\left( f_{\Omega }\right) $ is positive principal eigenvalue of problem $%
\left( \ref{2}\right) .$ Thus, we must has $\int_{\Omega }\left\vert \nabla
\varphi _{0}\right\vert ^{2}dx=\lambda _{1}\left( f_{\Omega }\right) $ and $%
\varphi _{0}=\phi _{1}$ a positive principal eigenfunction of problem $%
\left( \ref{2}\right) ,$ which completes the proof.
\end{proof}

By Lemma \ref{g3-0}, for each $0<\lambda <\lambda _{1}\left( f_{\Omega
}\right) $ there exists $\overline{\mu }_{0}\left( \lambda \right) \geq \mu
_{0}$\ with $\overline{\mu }_{0}\left( \lambda \right) \rightarrow \infty $
as $\lambda \rightarrow \lambda _{1}^{-}\left( f_{\Omega }\right) $ such
that for every $\mu >\overline{\mu }_{0}\left( \lambda \right) ,$ there
holds $\lambda <\widetilde{\lambda }_{1,\mu }\left( f\right) <\lambda
_{1}\left( f_{\Omega }\right) ,$ which indicates that%
\begin{equation}
\left\Vert u\right\Vert _{\mu }^{2}-\lambda \int_{\mathbb{R}%
^{N}}fu^{2}dx\geq \frac{\widetilde{\lambda }_{1,\mu }\left( f\right)
-\lambda }{\widetilde{\lambda }_{1,\mu }\left( f\right) }\left\Vert
u\right\Vert _{\mu }^{2}\text{ for all }u\in X.  \label{10}
\end{equation}%
Moreover, it is easy to show that%
\begin{equation*}
h_{u}^{\prime \prime }(1)=-\left( p-2\right) \left( \left\Vert u\right\Vert
_{\mu }^{2}-\lambda \int_{\mathbb{R}^{N}}fu^{2}dx\right) -\left( p-4\right)
\left\Vert u\right\Vert _{D^{1,2}}^{4}<0
\end{equation*}%
for all $4\leq p<6$ and $u\in \mathbf{N}_{\mu ,\lambda }.$ Furthermore, we
have the following results.

\begin{lemma}
\label{g2-4}Suppose that $N=3,4\leq p<6$ and $\Gamma _{0}=\infty $ (if $%
p=4). $ Then for each $a>0$ and $0<\lambda <\lambda _{1}\left( f_{\Omega
}\right) $, there holds $\mathbf{N}_{\mu ,\lambda }=\mathbf{N}_{\mu ,\lambda
}^{-}$ and $\mathbf{N}_{\mu ,\lambda }^{-}=\{t_{\max }(u)u:u\in \Theta _{\mu
}^{+}\left( p\right) \}$ for $\mu >0$ sufficiently large.
\end{lemma}

\begin{proof}
By $\left( \ref{10}\right) ,$ $\Lambda _{\mu }^{+}\neq \emptyset $ and $%
\Lambda _{\mu }^{-}\cup \Lambda _{\mu }^{0}=\emptyset ,$ this implies that
the submanifolds $\mathbf{N}_{\mu ,\lambda }^{+}$ and $\mathbf{N}_{\mu
,\lambda }^{0}$ are empty and
\begin{equation*}
\mathbf{N}_{\mu ,\lambda }=\mathbf{N}_{\mu ,\lambda }^{-}=\{t_{\max
}(u)u:u\in \Theta _{\mu }^{+}\left( p\right) \}
\end{equation*}%
for $\mu >0$ sufficiently large. This completes the proof.
\end{proof}

\begin{lemma}
\label{g2-5}Suppose that $N=3,p=4$ and $\Gamma _{0}<+\infty $ (if $p=4$)$.$
Then we have the following results.\newline
$\left( i\right) $ For each $a<\Gamma _{0}$ and $0<\lambda <\lambda
_{1}\left( f_{\Omega }\right) $, there holds $\mathbf{N}_{\mu ,\lambda }=%
\mathbf{N}_{\mu ,\lambda }^{-}$ and $\mathbf{N}_{\mu ,\lambda
}^{-}=\{t_{\max }(u)u:u\in \Theta _{\mu }^{+}\left( p\right) \}$ for $\mu >0$
sufficiently large.\newline
$\left( ii\right) $ For each $a>\Gamma _{0}$ and $0<\lambda <\lambda
_{1}\left( f_{\Omega }\right) ,$ there holds $\mathbf{N}_{\mu ,\lambda
}=\emptyset $ for $\mu >0$ sufficiently large.\newline
$\left( iii\right) $ For each $a>\Gamma _{0}$ and $\lambda \geq \lambda
_{1}\left( f_{\Omega }\right) ,$ there holds $\mathbf{N}_{\mu ,\lambda }=%
\mathbf{N}_{\mu ,\lambda }^{+}$ and $\mathbf{N}_{\mu ,\lambda
}^{+}=\{t_{\min }(u)u:u\in \Theta _{\mu }^{-}\left( p\right) \}$ for $\mu >0$
sufficiently large.
\end{lemma}

\begin{proof}
$\left( i\right) $ The proof is almost the same as Lemma \ref{g2-4}, and we
omit it here.\newline
$\left( ii\right) $ Since
\begin{equation*}
\int_{\mathbb{R}^{3}}Q|u|^{4}dx\leq \Gamma _{0}\left\Vert u\right\Vert
_{D^{1,2}}^{4}\text{ for all }u\in X,
\end{equation*}%
we can obtain%
\begin{equation}
\Phi _{p}\left( u\right) =\int_{\mathbb{R}^{3}}Q|u|^{4}dx-a\left\Vert
u\right\Vert _{D^{1,2}}^{4}<0  \label{3-3}
\end{equation}%
for all $a>\Gamma _{0}$ and $u\in X,$ this implies that $\Theta _{\mu
}^{+}\left( p\right) \cup \Theta _{\mu }^{0}\left( p\right) =\emptyset .$
Moreover, $\Lambda _{\mu }^{-}\cup \Lambda _{\mu }^{0}=\emptyset $ for $\mu
>0$ sufficiently large, by Lemma \ref{g2-2}, $\mathbf{N}_{\mu ,\lambda
}=\emptyset $ for $\mu >0$ sufficiently large.\newline
$\left( iii\right) $ By Lemma \ref{g3-0}, there exists a positive function $%
\varphi _{\mu }\in X$ such that $\int_{\mathbb{R}^{3}}f\varphi _{\mu
}^{2}dx=1$ and
\begin{equation*}
\widetilde{\lambda }_{1,\mu }\left( f\right) =\int_{\mathbb{R}%
^{3}}\left\vert \nabla \varphi _{\mu }\right\vert ^{2}+\mu V\varphi _{\mu
}^{2}dx<\lambda _{1}\left( f_{\Omega }\right) \text{ for }\mu >0\text{
sufficiently large.}
\end{equation*}%
If $\lambda \geq \lambda _{1}\left( f_{\Omega }\right) $, then%
\begin{equation}
\int_{\mathbb{R}^{3}}\left\vert \nabla \varphi _{\mu }\right\vert ^{2}+\mu
V\varphi _{\mu }^{2}dx-\lambda \int_{\mathbb{R}^{3}}f\varphi _{\mu }^{2}dx=%
\widetilde{\lambda }_{1,\mu }\left( f\right) -\lambda <0,  \label{3-6}
\end{equation}%
and so $\varphi _{\mu }\in \Lambda _{\mu }^{-}$ for $\mu >0$ sufficiently
large. Thus, by $\left( \ref{3-3}\right) ,$ $\Lambda _{\mu }^{-}\cap \Theta
_{\mu }^{-}\left( p\right) \neq \emptyset $ and $\Theta _{\mu }^{+}\left(
p\right) \cup \Theta _{\mu }^{0}\left( p\right) =\emptyset .$ Then by Lemma %
\ref{g2-3},
\begin{equation*}
\mathbf{N}_{\mu ,\lambda }=\mathbf{N}_{\mu ,\lambda }^{+}=\{t_{\min
}(u)u:u\in \Theta _{\mu }^{-}\left( p\right) \}
\end{equation*}%
for $\mu >0$ sufficiently large. This completes the proof.
\end{proof}

If $\lambda \geq \lambda _{1}\left( f_{\Omega }\right) $, then by $\left( %
\ref{3-6}\right) ,$ $\varphi _{\mu }\in \Lambda _{\mu }^{-}$ for $\mu >0$
sufficiently large. Moreover, if $\Phi _{p}\left( \phi _{1}\right) <0$, then
by Lemma \ref{g3-0}, $\Phi _{p}\left( \varphi _{\mu }\right) <0$ for $\mu >0$
sufficiently large, this implies that $\varphi _{\mu }\in \Lambda _{\mu
}^{-}\cap \Theta _{\mu }^{-}\left( p\right) $ and so $\mathbf{N}_{\mu
,\lambda }^{+}\neq \emptyset $ for $\mu >0$ sufficiently large. Thus, as
well shall see, $\mathbf{N}_{\mu ,\lambda }$ may consist of two distinct
components in this case which makes it possible to prove the existence of at
least two positive solutions by showing that $J_{\mu ,\lambda }$ has an
appropriate minimizer on each component.

Moreover, if $\lambda \geq \lambda _{1}\left( f_{\Omega }\right) $, then
roughly speaking $\left\Vert u\right\Vert _{\mu }^{2}\leq \lambda \int_{%
\mathbb{R}^{3}}fu^{2}dx$ if and only if $u$ is almost a multiple of $\phi
_{1}$ for $\mu $ sufficiently large. Thus, if $\phi _{1}\in \Theta _{\mu
}^{-}\left( p\right) $, it should follow that $\overline{\Lambda _{\mu }^{-}}%
\cap \overline{\Theta _{\mu }^{+}}\left( p\right) =\emptyset $ for $\mu >0$
sufficiently large. This is made precise in the following lemma and we show
subsequently that $\overline{\Lambda _{\mu }^{-}}\cap \overline{\Theta _{\mu
}^{+}}\left( p\right) =\emptyset $ is an important condition for
establishing the existence of minimizers.

\begin{theorem}
\label{g4-1}Suppose that $N=3,4\leq p<6$ and $\Phi _{p}\left( \phi
_{1}\right) <0.$ Then for each $a>0$ there exists $\delta _{0}>0$ such that
for every $\lambda _{1}\left( f_{\Omega }\right) \leq \lambda <\lambda
_{1}\left( f_{\Omega }\right) +\delta _{0},$ there holds $\overline{\Lambda
_{\mu }^{-}}\cap \overline{\Theta _{\mu }^{+}}\left( p\right) =\emptyset $
for $\mu >0$ sufficiently large.
\end{theorem}

\begin{proof}
We may divide the proof into the two cases: $\left( I\right) $ $4<p<6$ and $%
\left( II\right) $ $p=4.$\newline
Case $\left( I\right) :4<p<6.$ Suppose that the result is false. Then there
exist sequences $\{\mu _{n}\},\{\lambda _{n}\}$ and $\{u_{n}\}\subset X$
with $\lambda _{n}\rightarrow \lambda _{1}^{+}\left( f_{\Omega }\right) $
and $\mu _{n}\rightarrow \infty $ such that $\Vert u_{n}\Vert _{\mu _{n}}=1$
and%
\begin{equation}
\Vert u_{n}\Vert _{\mu _{n}}^{2}-\lambda _{n}\int_{\mathbb{R}%
^{3}}fu_{n}^{2}dx\leq 0,\quad \int_{\mathbb{R}^{3}}Q|u_{n}|^{p}dx\geq 0.
\label{e16}
\end{equation}%
By Theorem \ref{g4-0}, we may assume that there exists $u_{0}\in
H_{0}^{1}\left( \Omega \right) $ such that $u_{n}\rightharpoonup u_{0}$ in $%
X $ and $u_{n}\rightarrow u_{0}$ in $L^{r}\left( \mathbb{R}^{3}\right) $ for
all $2\leq r<6.$ Then by $\left( \ref{e16}\right) ,$%
\begin{equation}
\lim_{n\rightarrow \infty }\lambda _{n}\int_{\mathbb{R}^{3}}fu_{n}^{2}dx=%
\lambda _{1}\left( f_{\Omega }\right) \int_{\mathbb{R}^{3}}fu_{0}^{2}dx\geq 1
\label{e17}
\end{equation}%
and%
\begin{equation*}
\lim_{n\rightarrow \infty }\int_{\mathbb{R}^{3}}Q|u_{n}|^{p}dx=\int_{\mathbb{%
R}^{3}}Q|u_{0}|^{p}dx.
\end{equation*}%
Now, we show that $\lim_{n\rightarrow \infty }\int_{\mathbb{R}^{3}}|\nabla
u_{n}|^{2}dx=\int_{\Omega }|\nabla u_{0}|^{2}dx.$ Suppose on the contrary.
Then by $\left( \ref{e17}\right) ,$%
\begin{eqnarray*}
\int_{\Omega }\left( |\nabla u_{0}|^{2}-\lambda _{1}\left( f_{\Omega
}\right) f_{\Omega }u_{0}^{2}\right) dx &=&\int_{\mathbb{R}^{3}}\left(
|\nabla u_{0}|^{2}-\lambda _{1}\left( f_{\Omega }\right) fu_{0}^{2}\right) dx
\\
&<&\liminf_{n\rightarrow \infty }\left( \left\Vert u_{n}\right\Vert _{\mu
_{n}}^{2}-\lambda _{n}\int_{\mathbb{R}^{3}}fu_{n}^{2}dx\right) \leq 0,
\end{eqnarray*}%
which is impossible. Hence $\lim_{n\rightarrow \infty }\int_{\Omega }|\nabla
u_{n}|^{2}dx=\int_{\Omega }|\nabla u_{0}|^{2}dx.$ It follows that%
\begin{equation*}
(i)\int_{\Omega }(|\nabla u_{0}|^{2}-\lambda _{1}\left( f_{\Omega }\right)
f_{\Omega }u_{0}^{2})dx\leq 0,\quad (ii)\;\int_{\mathbb{R}%
^{3}}Q|u_{0}|^{p}dx\geq 0.
\end{equation*}%
But $(i)$ implies that $u_{0}=k\phi _{1}$ for some $k$ and then $(ii)$
implies that $k=0$ which is impossible as $\lambda _{1}\left( f_{\Omega
}\right) \int_{\mathbb{R}^{3}}fu_{0}^{2}dx\geq 1.$ Thus, there exists $%
\delta _{0}>0$ and $\widehat{\mu }_{0}\geq \mu _{0}$ such that $\overline{%
\Lambda _{\mu }^{-}}\cap \overline{\Theta _{\mu }^{+}}\left( p\right)
=\emptyset $ for all $\lambda _{1}\left( f_{\Omega }\right) \leq \lambda
<\lambda _{1}\left( f_{\Omega }\right) +\delta _{0}$ and $\mu >\widehat{\mu }%
_{0}.$ Moreover, if $\mathbf{N}_{\mu ,\lambda }^{0}\neq \emptyset ,$ then
there exists $u_{0}\in \mathbf{N}_{\mu ,\lambda }^{0}$ such that $\frac{u_{0}%
}{\Vert u_{0}\Vert _{\mu }}\in \Lambda _{\mu }^{0}\cap \Theta _{\mu
}^{0}\left( p\right) \subset \overline{\Lambda _{\mu }^{-}}\cap \overline{%
\Theta _{\mu }^{+}}\left( p\right) =\emptyset $ which is impossible.
Therefore, $\mathbf{N}_{\mu ,\lambda }^{0}=\emptyset $ for all $\lambda
_{1}\left( f_{\Omega }\right) \leq \lambda <\lambda _{1}\left( f_{\Omega
}\right) +\delta _{0}$ and for $\mu >0$ sufficiently large.\newline
Case $\left( II\right) :p=4.$ The proof of $p=4$ is essentially similar in
Case $\left( I\right) $, so we omit it here. This completes the proof.
\end{proof}

\begin{lemma}
\label{g4-3}Suppose that $N=3,4<p<6$ and $\int_{\Omega }Q\phi _{1}^{p}dx<0.$
Let $\delta _{0}>0$ be as in Theorem \ref{g4-1}. Then for each $a>0$ and $%
\lambda _{1}\left( f_{\Omega }\right) \leq \lambda <\lambda _{1}\left(
f_{\Omega }\right) +\delta _{0}$, there holds $\mathbf{N}_{\mu ,\lambda }=%
\mathbf{N}_{\mu ,\lambda }^{-}\cup \mathbf{N}_{\mu ,\lambda }^{+}$ for $\mu
>0$ sufficiently large. Furthermore, $\mathbf{N}_{\mu ,\lambda }^{\pm }$ are
nonempty sets for $\mu >0$ sufficiently large.
\end{lemma}

\begin{proof}
Since $\int_{\Omega }Q\phi _{1}^{p}dx<0,$ by Lemma \ref{g3-0}, there exists
a positive function $\varphi _{\mu }\in X$ such that
\begin{equation*}
\widetilde{\lambda }_{1,\mu }\left( f\right) =\int_{\mathbb{R}%
^{3}}\left\vert \nabla \varphi _{\mu }\right\vert ^{2}+\mu V\varphi _{\mu
}^{2}dx<\lambda _{1}\left( f_{\Omega }\right)
\end{equation*}%
and%
\begin{equation*}
\varphi _{\mu }\rightarrow \phi _{1}\text{ in }X\text{ as }\mu \rightarrow
\infty .
\end{equation*}%
Hence, for $\mu >0$ large enough,%
\begin{equation*}
\int_{\mathbb{R}^{3}}Q|\varphi _{\mu }|^{p}dx<0,
\end{equation*}%
this implies that $\varphi _{\mu }\in \Theta _{\mu }^{-}\left( p\right) $
for $\mu >0$ sufficiently large. Moreover, $\left\Vert \varphi _{\mu
}\right\Vert _{\mu }^{2}<\lambda \int_{\mathbb{R}^{3}}f\varphi _{\mu }^{2}dx$
for all $\lambda \geq \lambda _{1}\left( f_{\Omega }\right) ,$ we have $%
\varphi _{\mu }\in \Lambda _{\mu }^{-}\cap \Theta _{\mu }^{-}\left( p\right)
$ for $\mu >0$ sufficiently large. Next, by conditions $\left( D_{1}\right) $
and $\left( D_{2}\right) ,$ we have $\Lambda _{\mu }^{+}\cap \Theta _{\mu
}^{+}\left( p\right) \neq \emptyset .$ Thus, by Lemma \ref{g2-2} $\left(
i\right) ,\left( ii\right) ,$ $\mathbf{N}_{\mu ,\lambda }^{\pm }\neq
\emptyset .$ Moreover, by Theorem \ref{g4-1}, $\overline{\Lambda _{\mu }^{-}}%
\cap \overline{\Theta _{\mu }^{+}}\left( p\right) =\emptyset $ for $\mu >0$
sufficiently large. Then by Lemma \ref{g2-2} $\left( iii\right) $ this
implies that $\mathbf{N}_{\mu ,\lambda }^{0}=\emptyset $ or $\mathbf{N}_{\mu
,\lambda }=\mathbf{N}_{\mu ,\lambda }^{-}\cup \mathbf{N}_{\mu ,\lambda }^{+}$
for $\mu >0$ sufficiently large. This completes the proof.
\end{proof}

\begin{lemma}
\label{g4-4}Suppose that $N=3,p=4$ and $\Phi _{p}\left( \phi _{1}\right) <0.$
Then there exists $\delta _{0}>0$ such that for every $\lambda _{1}\left(
f_{\Omega }\right) \leq \lambda <\lambda _{1}\left( f_{\Omega }\right)
+\delta _{0}$, there holds $\mathbf{N}_{\mu ,\lambda }=\mathbf{N}_{\mu
,\lambda }^{-}\cup \mathbf{N}_{\mu ,\lambda }^{+}$ for $\mu >0$ sufficiently
large. Furthermore, $\mathbf{N}_{\mu ,\lambda }^{\pm }$ are nonempty sets
for $\mu >0$ sufficiently large.
\end{lemma}

\begin{proof}
The proof is almost the same as Lemma \ref{g4-3}, and we omit it here.
\end{proof}

Next, we consider the following nonlinear Schr\"{o}dinger equation%
\begin{equation}
\left\{
\begin{array}{ll}
-\Delta u=\lambda f_{\Omega }u+Q_{\Omega }|u|^{p-2}u & \text{ in }\Omega ,
\\
u\in H_{0}^{1}(\Omega ), &
\end{array}%
\right.  \tag*{$\left( E_{\lambda ,\Omega }\right) $}
\end{equation}%
where $N\geq 3,2<p<\min \left\{ 4,2^{\ast }\right\} ,0\leq \lambda <\lambda
_{1}\left( f_{\Omega }\right) $ and $Q_{\Omega }$ is a restriction of $Q$ on
$\overline{\Omega }.$

It is well-known that for each $0\leq \lambda <\lambda _{1}\left( f_{\Omega
}\right) ,$ Equation $\left( E_{\lambda ,\Omega }\right) $ has positive
ground state solution $w_{\lambda ,\Omega }$ such that
\begin{equation}
\inf_{u\in \mathbf{N}_{\lambda ,\Omega }^{\infty }}J_{\lambda ,\Omega
}^{\infty }\left( u\right) =J_{\lambda ,\Omega }^{\infty }\left( w_{\lambda
,\Omega }\right) =\alpha _{\lambda ,Q}^{\infty }(\Omega ),  \label{5-12}
\end{equation}%
and%
\begin{equation}
\int_{\Omega }\left\vert \nabla w_{\lambda ,\Omega }\right\vert
^{2}dx-\lambda \int_{\Omega }f_{\Omega }w_{\lambda ,\Omega
}^{2}dx=\int_{\Omega }Q_{\Omega }w_{\lambda ,\Omega }^{p}dx=\frac{2p}{p-2}%
\alpha _{\lambda ,Q}^{\infty }(\Omega )>0  \label{5-11}
\end{equation}%
where $J_{\lambda ,\Omega }^{\infty }$ is the energy functional of Equation $%
(E_{\lambda ,\Omega })$ in $H_{0}^{1}(\Omega )$ in the form
\begin{equation*}
J_{\lambda ,\Omega }^{\infty }\left( u\right) =\frac{1}{2}\left(
\int_{\Omega }\left\vert \nabla u\right\vert ^{2}dx-\lambda \int_{\Omega
}f_{\Omega }u^{2}dx\right) -\frac{1}{p}\int_{\Omega }Q_{\Omega }\left\vert
u\right\vert ^{p}dx,
\end{equation*}%
and
\begin{equation*}
\mathbf{N}_{\lambda ,\Omega }^{\infty }=\{u\in H_{0}^{1}(\Omega )\backslash
\left\{ 0\right\} \ |\ \left\langle \left( J_{\lambda ,\Omega }^{\infty
}\right) ^{\prime }\left( u\right) ,u\right\rangle =0\}.
\end{equation*}%
Clearly $\alpha _{\lambda ,Q}^{\infty }(\Omega )\leq \alpha _{0,Q}^{\infty
}(\Omega )\leq \frac{p-2}{2p}\left( \frac{S_{p}^{p}\left( \Omega \right) }{%
Q_{\Omega ,\min }}\right) ^{2/(p-2)}$ for all $0<\lambda <\lambda _{1}\left(
f_{\Omega }\right) $. Then we have the following nonempteness and properties
of submanifolds $\mathbf{N}_{\mu ,\lambda }^{+}$ and $\mathbf{N}_{\mu
,\lambda }^{-}.$

\begin{theorem}
\label{t6-1}Suppose that $N\geq 3$ and $2<p<\min \left\{ 4,2^{\ast }\right\}
.$ Then we have the following results.\newline
$\left( i\right) $ Let $0<\lambda <\lambda _{1}\left( f_{\Omega }\right) $
and let $w_{\lambda ,\Omega }$ be the ground state positive solution of
Equation $\left( E_{\lambda ,\Omega }\right) .$ If conditions $\left(
V_{1}\right) ,\left( V_{3}\right) ,\left( D_{1}\right) ,\left( D_{2}\right) $
and $\left( D_{4}\right) $ hold, then there exists $a_{0}>0$ independent of $%
\lambda ,\mu $ such that for every $0<a<a_{0},$ there exist two positive
constants $t_{a}^{-}$ and $t_{a}^{+}$ satisfying%
\begin{equation*}
1<t_{a}^{-}<\left( \frac{2}{4-p}\right) ^{1/(p-2)}<t_{a}^{+}
\end{equation*}%
and%
\begin{equation*}
t_{a}^{-}\rightarrow 1;t_{a}^{+}\rightarrow \infty \text{ as }a\rightarrow
0^{+}
\end{equation*}%
such that $t_{a}^{\pm }w_{\lambda ,\Omega }\in \mathbf{N}_{\mu ,\lambda
}^{\pm }.$ Furthermore, if $0<\lambda <\left[ 1-2\left( \frac{4-p}{4}\right)
^{2/p}\right] \lambda _{1}\left( f_{\Omega }\right) ,$ then%
\begin{equation*}
J_{\mu ,\lambda }\left( t_{a}^{-}w_{\Omega }\right) <\frac{p-2}{4p}\left(
\frac{\widetilde{\lambda }_{1,\mu }\left( f\right) -\lambda }{\widetilde{%
\lambda }_{1,\mu }\left( f\right) }\right) ^{p/\left( p-2\right) }\left(
\frac{2S_{p}^{p}\left( \Omega \right) }{Q_{\Omega ,\min }\left( 4-p\right) }%
\right) ^{2/(p-2)}
\end{equation*}%
for $\mu >0$ sufficiently large.\newline
$\left( ii\right) $ Let $\lambda \geq \lambda _{1}\left( f_{\Omega }\right) $
and let $\phi _{1}$ be positive principal eigenfunction of $-\Delta $ in $%
H_{0}^{1}\left( \Omega \right) $ with weight function $f_{\Omega }:=f|_{%
\overline{\Omega }}.$ Then for each $a>0$ there exists $t_{a}^{+}>0$ such
that $t_{a}^{+}\phi _{1}\in \mathbf{N}_{\mu ,\lambda }^{+}$ and
\begin{equation*}
J_{\mu ,\lambda }\left( t_{a}^{+}\phi _{1}\right) =\inf_{t\geq 0}J_{\mu
,\lambda }\left( t\phi _{1}\right) <0.
\end{equation*}%
In particular, $\mathbf{N}_{\mu ,\lambda }^{+}$ is nonempty and $\inf_{u\in
\mathbf{N}_{\mu ,\lambda }^{+}}J_{\mu ,\lambda }\left( u\right) <J_{\mu
,\lambda }\left( t_{a}^{+}\phi _{1}\right) .$
\end{theorem}

\begin{proof}
$\left( i\right) $ Define
\begin{equation*}
m\left( t\right) =t^{-2}\left( \left\Vert w_{\lambda ,\Omega }\right\Vert
_{\mu }^{2}-\lambda \int_{\Omega }f_{\Omega }w_{\lambda ,\Omega
}^{2}dx\right) -t^{p-4}\int_{\Omega }Q_{\Omega }\left\vert w_{\lambda
,\Omega }\right\vert ^{p}dx\text{ for }t>0.
\end{equation*}%
Clearly, $tw_{\lambda ,\Omega }\in \mathbf{N}_{\mu ,\lambda }$ if and only
if $m\left( t\right) +a\left( \int_{\Omega }\left\vert \nabla w_{\lambda
,\Omega }\right\vert ^{2}dx\right) ^{2}=0.$ Since $\left\Vert w_{\lambda
,\Omega }\right\Vert _{\mu }^{2}-\lambda \int_{\Omega }f_{\Omega }w_{\lambda
,\Omega }^{2}dx>0,$ similar to the arguments of \cite[Lemmas 2.4 and 2.5]%
{SW4}, there exist two positive constants $t_{a}^{-}$ and $t_{a}^{+}$
satisfying%
\begin{equation*}
1<t_{a}^{-}<\left( \frac{2}{4-p}\right) ^{1/(p-2)}<t_{a}^{+}
\end{equation*}%
and%
\begin{equation*}
t_{a}^{-}\rightarrow 1;t_{a}^{+}\rightarrow \infty \text{ as }a\rightarrow
0^{+}
\end{equation*}%
such that $t_{a}^{\pm }w_{\Omega }\in \mathbf{N}_{\mu ,\lambda }^{\pm },$%
\begin{equation*}
J_{\mu ,\lambda }\left( t_{a}^{-}w_{\lambda ,\Omega }\right) =\sup_{0\leq
t\leq t_{a}^{+}}J_{\mu ,\lambda }\left( tw_{\lambda ,\Omega }\right)
\end{equation*}%
and%
\begin{equation}
J_{\mu ,\lambda }\left( t_{a}^{+}w_{\lambda ,\Omega }\right) =\inf_{t\geq
t_{a}^{-}}J_{\mu ,\lambda }\left( tw_{\lambda ,\Omega }\right) =\inf_{t\geq
0}J_{\mu ,\lambda }\left( tw_{\lambda ,\Omega }\right) <0.  \label{5-13}
\end{equation}%
Furthermore, we get
\begin{eqnarray}
J_{\mu ,\lambda }\left( t_{a}^{-}w_{\lambda ,\Omega }\right) &=&\frac{\left(
t_{a}^{-}\right) ^{2}}{4}\left( \left\Vert w_{\lambda ,\Omega }\right\Vert
_{\mu }^{2}-\lambda \int_{\mathbb{R}^{N}}fw_{\lambda ,\Omega }^{2}dx\right) -%
\frac{\left( 4-p\right) \left( t_{a}^{-}\right) ^{p}}{4p}\int_{\mathbb{R}%
^{N}}Q\left\vert w_{\lambda ,\Omega }\right\vert ^{p}dx  \notag \\
&=&\frac{p}{2\left( p-2\right) }\left[ \left( t_{a}^{-}\right) ^{2}-\frac{%
\left( 4-p\right) \left( t_{a}^{-}\right) ^{p}}{p}\right] \alpha _{\lambda
,Q}^{\infty }(\Omega )  \notag \\
&<&\frac{1}{4}\left[ \left( t_{a}^{-}\right) ^{2}-\frac{\left( 4-p\right)
\left( t_{a}^{-}\right) ^{p}}{p}\right] \left( \frac{S_{p}^{p}\left( \Omega
\right) }{Q_{\Omega ,\min }}\right) ^{2/(p-2)}  \notag \\
&\rightarrow &\frac{p-2}{2p}\left( \frac{S_{p}^{p}\left( \Omega \right) }{%
Q_{\Omega ,\min }}\right) ^{2/(p-2)}\text{ as }a\rightarrow 0^{+}.
\label{5-14}
\end{eqnarray}%
Since $0<\lambda <\left[ 1-2\left( \frac{4-p}{4}\right) ^{2/p}\right]
\lambda _{1}\left( f_{\Omega }\right) ,$ by Lemma \ref{g3-0} and $\left( \ref%
{5-14}\right) ,$ we can conclude that there exists $a_{0}>0$ independent of $%
\lambda ,\mu $ such that for every $0<a<a_{0},$%
\begin{equation*}
J_{\mu ,\lambda }\left( t_{a}^{-}w_{\lambda }\right) <\frac{p-2}{4p}\left(
\frac{\widetilde{\lambda }_{1,\mu }\left( f\right) -\lambda }{\widetilde{%
\lambda }_{1,\mu }\left( f\right) }\right) ^{p/\left( p-2\right) }\left(
\frac{2S_{p}^{p}\left( \Omega \right) }{Q_{\Omega ,\min }\left( 4-p\right) }%
\right) ^{2/(p-2)}
\end{equation*}%
for $\mu >0$ sufficiently large.\newline
$\left( ii\right) $ Since $\lambda \geq \lambda _{1}\left( f_{\Omega
}\right) ,$ we have $\left\Vert \phi _{1}\right\Vert _{\mu }^{2}-\lambda
\int_{\Omega }f_{\Omega }\phi _{1}^{2}dx\leq 0,$ this implies that $\frac{%
\phi _{1}}{\left\Vert \phi _{1}\right\Vert _{\mu }}\in \overline{\Lambda
_{\mu }^{-}}.$ Then by, Lemma \ref{g2-6} $\left( i\right) ,$ for each $a>0$
there exists $t_{a}^{+}>0$ such that $t_{a}^{+}\phi _{1}\in \mathbf{N}_{\mu
,\lambda }^{+}.$ Moreover, $h_{\phi _{1}}^{\prime }\left( t\right) <0$ for
all $t\in \left( 0,t_{a}^{+}\right) $ and $h_{\phi _{1}}^{\prime }\left(
t\right) >0$ for all $t>t_{a}^{+},$ which leads to
\begin{equation*}
J_{\mu ,\lambda }\left( t_{a}^{+}\phi _{1}\right) =\inf_{t\geq 0}J_{\mu
,\lambda }\left( t\phi _{1}\right) <0.
\end{equation*}%
This completes the proof.
\end{proof}

\section{The case when $N=3$ and $4<p<6$}

\subsection{The subcase: $\protect\lambda <\protect\lambda _{1}\left(
f_{\Omega }\right) $}

We need the following results.

\begin{lemma}
\label{g3-1}Suppose that $4<p<6.$ Then for each $0<\lambda <\lambda
_{1}\left( f_{\Omega }\right) $ there exists $\overline{\mu }_{0}\left(
\lambda \right) \geq \mu _{0}$ with $\overline{\mu }_{0}\left( \lambda
\right) \rightarrow \infty $ as $\lambda \rightarrow \lambda _{1}^{-}\left(
f_{\Omega }\right) $\ such that for every $\mu >\overline{\mu }_{0}\left(
\lambda \right) ,$ the energy functional $J_{\mu ,\lambda }$ is coercive and
bounded below on $\mathbf{N}_{\mu ,\lambda }^{-}.$ Furthermore,
\begin{equation}
\inf_{u\in \mathbf{N}_{\mu ,\lambda }^{-}}J_{\mu ,\lambda }(u)\geq \frac{1}{4%
}\left( \frac{\widetilde{\lambda }_{1,\mu }\left( f\right) -\lambda }{%
\widetilde{\lambda }_{1,\mu }\left( f\right) }\right) ^{p/\left( p-2\right)
}\left( \frac{S^{p}}{\left\Vert Q\right\Vert _{\infty }\left\vert \left\{
V<c\right\} \right\vert ^{\frac{6-p}{6}}}\right) ^{2/\left( p-2\right) }>0
\label{13}
\end{equation}%
for all $u\in \mathbf{N}_{\mu ,\lambda }^{-}.$
\end{lemma}

\begin{proof}
By $\left( \ref{11}\right) $ and $\left( \ref{10}\right) ,$ for each $\mu >%
\overline{\mu }_{0}\left( \lambda \right) $ and $u\in \mathbf{N}_{\mu
,\lambda }^{-},$ we obtain%
\begin{eqnarray*}
\frac{\widetilde{\lambda }_{1,\mu }\left( f\right) -\lambda }{\widetilde{%
\lambda }_{1,\mu }\left( f\right) }\left\Vert u\right\Vert _{\mu }^{2} &\leq
&\left\Vert u\right\Vert _{\mu }^{2}-\lambda \int_{\mathbb{R}%
^{3}}fu^{2}dx+a\left\Vert u\right\Vert _{D^{1,2}}^{4} \\
&=&\int_{\mathbb{R}^{3}}Q\left\vert u\right\vert ^{p}dx\leq \left\Vert
Q\right\Vert _{\infty }\left\vert \left\{ V<c\right\} \right\vert ^{\frac{6-p%
}{6}}S^{-p}\left\Vert u\right\Vert _{\mu }^{p},
\end{eqnarray*}%
which indicates that%
\begin{equation*}
\left\Vert u\right\Vert _{\mu }\geq \left( \frac{S^{p}\left( \widetilde{%
\lambda }_{1,\mu }\left( f\right) -\lambda \right) }{\widetilde{\lambda }%
_{1,\mu }\left( f\right) \left\Vert Q\right\Vert _{\infty }\left\vert
\left\{ V<c\right\} \right\vert ^{\frac{6-p}{6}}}\right) ^{1/\left(
p-2\right) }.
\end{equation*}%
Thus,%
\begin{eqnarray*}
J_{\mu ,\lambda }(u) &\geq &\frac{1}{4}\left( \left\Vert u\right\Vert _{\mu
}^{2}-\lambda \int_{\mathbb{R}^{3}}fu^{2}dx\right) \geq \frac{\widetilde{%
\lambda }_{1,\mu }\left( f\right) -\lambda }{4\widetilde{\lambda }_{1,\mu
}\left( f\right) }\left\Vert u\right\Vert _{\mu }^{2} \\
&\geq &\frac{1}{4}\left( \frac{\widetilde{\lambda }_{1,\mu }\left( f\right)
-\lambda }{\widetilde{\lambda }_{1,\mu }\left( f\right) }\right) ^{p/\left(
p-2\right) }\left( \frac{S^{p}}{\left\Vert Q\right\Vert _{\infty }\left\vert
\left\{ V<c\right\} \right\vert ^{\frac{6-p}{6}}}\right) ^{2/\left(
p-2\right) }>0,
\end{eqnarray*}%
this implies that the energy functional $J_{\mu ,\lambda }$ is coercive and
bounded below on $\mathbf{N}_{\mu ,\lambda }^{-}.$ This completes the proof.
\end{proof}

We now show that there exists a minimizer on $\mathbf{N}_{\mu ,\lambda }^{-}$
which is a critical point of $J_{\mu ,\lambda }(u)$ and so a nontrivial
solution of Equation $\left( E_{\mu ,\lambda }\right) .$ First, we define%
\begin{equation*}
\theta _{a,\lambda }(\Omega )=\inf_{u\in \mathbf{M}_{\mu ,\lambda }(\Omega
)}J_{\mu ,\lambda }|_{H_{0}^{1}(\Omega )\cap H^{2}(\mathbb{R}^{3})}(u),
\end{equation*}%
where
\begin{equation*}
\mathbf{M}_{\lambda }(\Omega )=\{u\in H_{0}^{1}(\Omega )\cap H^{1}(\mathbb{R}%
^{3}):\left\langle J_{\mu ,\lambda }^{\prime }|_{H_{0}^{1}(\Omega )}\left(
u\right) ,u\right\rangle =0\}.
\end{equation*}%
Note that
\begin{equation*}
J_{\mu ,\lambda }|_{H_{0}^{1}\left( \Omega \right) }(u)=\frac{a}{4}\left(
\int_{\Omega }\left\vert \nabla u\right\vert ^{2}dx\right) ^{2}+\frac{1}{2}%
\int_{\Omega }\left\vert \nabla u\right\vert ^{2}dx-\frac{1}{p}\int_{\Omega
}Q\left\vert u\right\vert ^{p}dx-\frac{\lambda }{2}\int_{\Omega }f_{\Omega
}u^{2}dx,
\end{equation*}%
a restriction of $J_{\mu ,\lambda }$ on $H_{0}^{1}(\Omega ),$ and $\theta
_{a,\lambda }(\Omega )$ independent of $\mu .$ Since $\max \left\{
Q,0\right\} \not\equiv 0$ in\ $\Omega ,$ $\mathbf{M}_{\lambda }(\Omega )\neq
\emptyset .$ Thus, similar to the argument of $\left( \ref{13}\right) $, we
can conclude that $J_{\mu ,\lambda }|_{H_{0}^{1}(\Omega )}$ is bounded below
on $\mathbf{M}_{\lambda }(\Omega ).$ Moreover, $H_{0}^{1}(\Omega )\subset
X_{\mu }$ for all $\mu >0$, one can see that
\begin{equation}
0<\eta \leq \inf_{u\in \mathbf{N}_{\mu ,\lambda }^{-}}J_{\mu ,\lambda
}(u)\leq \theta _{a,\lambda }(\Omega )\text{ for all }\mu \geq \overline{\mu
}_{0}\left( \lambda \right) .  \label{5-0}
\end{equation}%
Taking $D_{a,\lambda }>\theta _{a,\lambda }(\Omega ).$ Then we have
\begin{equation}
\inf_{u\in \mathbf{N}_{\mu ,\lambda }^{-}}J_{\mu ,\lambda }(u)\leq \theta
_{a,\lambda }(\Omega )<D_{a,\lambda }\text{ for all }\mu \geq \overline{\mu }%
_{0}\left( \lambda \right) .  \label{5-1}
\end{equation}

\bigskip

\textbf{We are now ready to prove Theorem \ref{t1}: }When $0<\lambda
<\lambda _{1}\left( f_{\Omega }\right) .$ By Lemma $\ref{g3-1},$ $\left( \ref%
{5-1}\right) $ and the Ekeland variational principle \cite{E}, for each $\mu
>\overline{\mu }_{0}\left( \lambda \right) $ there exists a minimizing
sequence $\left\{ u_{n}\right\} \subset \mathbf{N}_{\mu ,\lambda }^{-}$ such
that%
\begin{equation*}
D_{a,\lambda }>\lim_{n\rightarrow \infty }J_{\mu ,\lambda
}(u_{n})=\inf_{u\in \mathbf{N}_{\mu ,\lambda }^{-}}J_{\mu ,\lambda }(u)>0%
\text{ and }J_{\mu ,\lambda }^{\prime }(u_{n})=o\left( 1\right) .
\end{equation*}%
Since $\inf_{u\in \mathbf{N}_{\mu ,\lambda }^{-}}J_{\mu ,\lambda
}(u)<D_{a,\lambda },$ again using Lemma \ref{g3-1}, there exists $%
c_{a,\lambda }>0$ such that $\left\Vert u_{n}\right\Vert _{\mu }\leq
c_{a,\lambda }.$ By Proposition \ref{p1}, there exist a subsequence $\left\{
u_{n}\right\} $ and $u_{0}\in X$ such that $J_{\mu ,\lambda }^{\prime
}(u_{0})=0$ and $u_{n}\rightarrow u_{0}$ strongly in $X_{\mu }$ for $\mu >0$
sufficiently large, which implies that $J_{\mu ,\lambda }$ has minimizer $%
u_{0}$ in $\mathbf{N}_{\mu ,\lambda }=\mathbf{N}_{\mu ,\lambda }^{-}$ for $%
\mu $ sufficiently large. Since $J_{\mu ,\lambda }(u_{0})=J_{\mu ,\lambda
}(|u_{0}|),$ by Lemma \ref{g2-1}, we may assume that $u_{0}$ is a positive
solution of Equation $\left( E_{\mu ,\lambda }\right) $ such that $J_{\mu
,\lambda }(u_{0})=\inf_{u\in \mathbf{N}_{\mu ,\lambda }^{-}}J_{\mu ,\lambda
}(u)>0.$

\subsection{The subcase: $\protect\lambda \geq \protect\lambda _{1}\left(
f_{\Omega }\right) $}

By Theorem \ref{g4-1}, for each $a>0$ there exists $\delta _{0}>0$ such that
for every $\lambda _{1}\left( f_{\Omega }\right) \leq \lambda <\lambda
_{1}\left( f_{\Omega }\right) +\delta _{0},$ there holds $\overline{\Lambda
_{\mu }^{-}}\cap \overline{\Theta _{\mu }^{+}}\left( p\right) =\emptyset $
for $\mu >0$ sufficiently large, it is possible to obtain more information
about the nature of the Nehari manifold as follows.

\begin{lemma}
\label{g4-2}Suppose that $4<p<6$ and $\int_{\Omega }Q\phi _{1}^{p}dx<0.$ Let
$\delta _{0}>0$ be as in Theorem \ref{g4-1}. Then for every $\lambda
_{1}\left( f_{\Omega }\right) \leq \lambda <\lambda _{1}\left( f_{\Omega
}\right) +\delta _{0},$ we have the following results.\newline
$(i)$ There exists $c_{0}>0$ such that $\left\Vert u\right\Vert _{\mu }\geq
c_{0}$ for all $u\in \mathbf{N}_{\mu ,\lambda }^{-}$ and $\mu >0$
sufficiently large.\newline
$(ii)$ $\mathbf{N}_{\mu ,\lambda }^{-}\;$and$\;\mathbf{N}_{\mu ,\lambda
}^{+}\;$are\ separated for $\mu >0$ sufficiently large, i.e.$,\overline{%
\mathbf{N}_{\mu ,\lambda }^{-}}\cap \overline{\mathbf{N}_{\mu ,\lambda }^{+}}%
=\emptyset $ for $\mu >0$ sufficiently large.\newline
$(iii)$ $\mathbf{N}_{\mu ,\lambda }^{+}\;$is\ uniform bounded for $\mu >0$
sufficiently large.
\end{lemma}

\begin{proof}
$(i)$ Suppose on the contrary. Then there exist $\left\{ \mu _{n}\right\}
\subset \mathbb{R}^{+}$ and $\{u_{n}\}\subset \mathbf{N}_{\mu _{n},\lambda
}^{-}$ such that $\mu _{n}\rightarrow \infty $ and $\left\Vert
u_{n}\right\Vert _{\mu _{n}}\rightarrow 0.$ Hence, by $\left( \ref{11}%
\right) ,$%
\begin{eqnarray}
0 &<&\left\Vert u_{n}\right\Vert _{\mu _{n}}^{2}-\lambda \int_{\mathbb{R}%
^{3}}fu_{n}^{2}dx  \notag \\
&<&\left\Vert u_{n}\right\Vert _{\mu _{n}}^{2}-\lambda \int_{\mathbb{R}%
^{3}}fu_{n}^{2}dx+a\left\Vert u_{n}\right\Vert _{D^{1,2}}^{4}=\int_{\mathbb{R%
}^{3}}Q|u_{n}|^{p}dx\rightarrow 0\text{ as }n\rightarrow \infty .  \label{8}
\end{eqnarray}%
Let $v_{n}=\frac{u_{n}}{\Vert u_{n}\Vert _{\mu _{n}}}.$ Then, by Theorem $%
\ref{g4-0}$, there exist subsequence $\left\{ v_{n}\right\} $ and $v_{0}\in
H_{0}^{1}\left( \Omega \right) $ such that
\begin{equation}
v_{n}\rightharpoonup v_{0}\text{ in }X_{\mu };\text{ }v_{n}\rightarrow v_{0}%
\text{ in }L^{r}\left( \mathbb{R}^{3}\right) \text{ for all }2\leq r<6.
\label{12}
\end{equation}%
Thus, by $\left( \ref{8}\right) $ and $\left( \ref{12}\right) ,$%
\begin{equation}
\lim_{n\rightarrow \infty }\int_{\mathbb{R}^{3}}fv_{n}^{2}dx=\int_{\mathbb{R}%
^{3}}fv_{0}^{2}dx  \label{9}
\end{equation}%
and%
\begin{equation}
\left\Vert v_{n}\right\Vert _{\mu _{n}}^{2}-\lambda \int_{\mathbb{R}%
^{3}}fv_{n}^{2}dx\rightarrow 0\text{ as\ }n\rightarrow \infty .  \label{15}
\end{equation}%
Moreover, by $\left( \ref{9}\right) ,\left( \ref{15}\right) ,v_{0}\in
H_{0}^{1}\left( \Omega \right) $ and Fatou's Lemma, we can obtain that%
\begin{equation*}
0=\lim\limits_{n\rightarrow \infty }\left( \left\Vert v_{n}\right\Vert _{\mu
_{n}}^{2}-\lambda \int_{\mathbb{R}^{3}}fv_{n}^{2}dx\right) =1-\lambda \int_{%
\mathbb{R}^{3}}fv_{0}^{2}dx,
\end{equation*}%
and for every $\mu >0$%
\begin{equation*}
\left\Vert v_{0}\right\Vert _{\mu }^{2}-\lambda \int_{\mathbb{R}%
^{3}}fv_{0}^{2}dx\leq \liminf_{n\rightarrow \infty }\left( \left\Vert
v_{n}\right\Vert _{\mu _{n}}^{2}-\lambda \int_{\mathbb{R}^{3}}fv_{n}^{2}dx%
\right) =0,
\end{equation*}%
this implies that $v_{0}\neq 0$ and $\frac{v_{0}}{\Vert v_{0}\Vert _{\mu }}%
\in \overline{\Lambda _{\mu }^{-}}$ for all $\mu >0.\ $Since $\int_{\mathbb{R%
}^{3}}Q|v_{n}|^{p}dx>0$ and $v_{n}\rightarrow v_{0}$ in $L^{p}\left( \mathbb{%
R}^{3}\right) $, it follows that $\int_{\mathbb{R}^{3}}Q|v_{0}|^{p}dx\geq 0$
and so $\frac{v_{0}}{\Vert v_{0}\Vert _{\mu }}\in \overline{\Theta _{\mu
}^{+}}\left( p\right) $ for all $\mu >0.$ Hence, $\frac{v_{0}}{\Vert
v_{0}\Vert _{\mu }}\in \overline{\Lambda _{\mu }^{-}}\cap \overline{\Theta
_{\mu }^{+}}\left( p\right) $ for all $\mu >0,$ which a contradiction. Thus,
$0\notin \overline{\mathbf{N}_{\mu ,\lambda }^{-}}$ for $\mu >0$
sufficiently large. Moreover, by Lemma \ref{g4-3}, $\overline{\mathbf{N}%
_{\mu ,\lambda }^{-}}\subset \mathbf{N}_{\mu ,\lambda }^{-}\cup \{0\}.$
Since $0\notin \overline{\mathbf{N}_{\mu ,\lambda }^{-}}$, it follows that $%
\overline{\mathbf{N}_{\mu ,\lambda }^{-}}=\mathbf{N}_{\mu ,\lambda }^{-}$,
i.e., $\mathbf{N}_{\mu ,\lambda }^{-}$ is closed.\newline
$(ii)$ By Theorem \ref{g4-1} and part $(i),$%
\begin{eqnarray*}
\overline{\mathbf{N}_{\mu ,\lambda }^{-}}\cap \overline{\mathbf{N}_{\mu
,\lambda }^{+}} &\subseteq &\mathbf{N}_{\mu ,\lambda }^{-}\cap (\mathbf{N}%
_{\mu ,\lambda }^{+}\cup \mathbf{N}_{\mu ,\lambda }^{0})=\mathbf{N}_{\mu
,\lambda }^{-}\cap (\mathbf{N}_{\mu ,\lambda }^{+}\cup \emptyset ) \\
&=&(\mathbf{N}_{\mu ,\lambda }^{-}\cap \mathbf{N}_{\mu ,\lambda }^{+})\cup (%
\mathbf{N}_{\mu ,\lambda }^{-}\cap \emptyset )=\emptyset ,
\end{eqnarray*}%
and so $\mathbf{N}_{\mu ,\lambda }^{-}$ and $\mathbf{N}_{\mu ,\lambda }^{+}$
are separated for $\mu >0$ sufficiently large.\newline
$(iii)$ Suppose on the contrary. Then there exist sequences $\left\{ \mu
_{n}\right\} \subset \mathbb{R}^{+}$ and $\{u_{n}\}\subset \mathbf{N}_{\mu
_{n},\lambda }^{+}$ such that $\mu _{n}\rightarrow \infty $ and $\Vert
u_{n}\Vert _{\mu _{n}}\rightarrow \infty $ as $n\rightarrow \infty .$
Clearly,%
\begin{equation}
\left\Vert u_{n}\right\Vert _{\mu _{n}}^{2}-\lambda \int_{\mathbb{R}%
^{3}}fu_{n}^{2}dx+a\left\Vert u_{n}\right\Vert _{D^{1,2}}^{4}=\int_{\mathbb{R%
}^{3}}Q|u_{n}|^{p}dx<0.  \label{4-1}
\end{equation}%
Let $v_{n}=\frac{u_{n}}{\Vert u_{n}\Vert _{\mu _{n}}}.$ Then by Theorem \ref%
{g4-0}, we may assume that there exists $v_{0}\in H_{0}^{1}\left( \Omega
\right) $ such that
\begin{equation*}
v_{n}\rightharpoonup v_{0}\text{ in }X;v_{n}\rightarrow v_{0}\text{ in }%
L^{r}\left( \mathbb{R}^{3}\right) \text{ for all }2<r\leq 6.
\end{equation*}%
Thus,
\begin{equation}
\lim_{n\rightarrow \infty }\int_{\mathbb{R}^{3}}fv_{n}^{2}dx=\int_{\mathbb{R}%
^{3}}fv_{0}^{2}dx  \label{4-2}
\end{equation}%
and%
\begin{equation}
\lim_{n\rightarrow \infty }\int_{\mathbb{R}^{3}}Q\left\vert v_{n}\right\vert
^{p}dx=\int_{\mathbb{R}^{3}}Q\left\vert v_{0}\right\vert ^{p}dx.  \label{4-8}
\end{equation}%
Moreover, by Fatou's Lemma,%
\begin{equation}
\int_{\mathbb{R}^{3}}|\nabla v_{0}|^{2}dx\leq \liminf_{n\rightarrow \infty
}\int_{\mathbb{R}^{3}}|\nabla v_{n}|^{2}dx.  \label{4-4}
\end{equation}%
Dividing $\left( \ref{4-1}\right) $ by $\Vert u_{n}\Vert _{\mu _{n}}^{2}$
gives%
\begin{equation}
\left\Vert v_{n}\right\Vert _{\mu _{n}}^{2}-\lambda \int_{\mathbb{R}%
^{3}}fv_{n}^{2}dx+a\Vert u_{n}\Vert _{\mu _{n}}^{2}\left\Vert
v_{n}\right\Vert _{D^{1,2}}^{4}=\Vert u_{n}\Vert _{\mu _{n}}^{p-2}\int_{%
\mathbb{R}^{3}}Q|v_{n}|^{p}dx<0.  \label{4-5}
\end{equation}%
Since
\begin{equation*}
\lim_{n\rightarrow \infty }\left( \left\Vert v_{n}\right\Vert _{\mu
_{n}}^{2}-\lambda \int_{\mathbb{R}^{3}}fv_{n}^{2}dx\right) =1-\lambda
\lim_{n\rightarrow \infty }\int_{\mathbb{R}^{3}}fv_{n}^{2}dx=1-\lambda \int_{%
\mathbb{R}^{3}}fv_{0}^{2}dx
\end{equation*}%
and $\Vert u_{n}\Vert _{\mu _{n}}\rightarrow \infty ,$ it obtain that $\int_{%
\mathbb{R}^{3}}Q|v_{0}|^{p}dx=0$ and $\int_{\mathbb{R}^{3}}fv_{0}^{2}dx>0$
from the conclusions $\left( \ref{4-8}\right) $ and $\left( \ref{4-5}\right)
.$ Moreover, by $v_{0}\in H_{0}^{1}\left( \Omega \right) ,\left( \ref{4-2}%
\right) $ and $\left( \ref{4-4}\right) ,$ for every $\mu >0,$%
\begin{equation*}
\left\Vert v_{0}\right\Vert _{\mu }^{2}-\lambda \int_{\mathbb{R}%
^{3}}fv_{0}^{2}dx\leq \liminf_{n\rightarrow \infty }\left( \left\Vert
v_{n}\right\Vert _{\mu _{n}}^{2}-\lambda \int_{\mathbb{R}^{3}}fv_{n}^{2}dx%
\right) \leq 0.
\end{equation*}%
We now show that $v_{n}\rightarrow v_{0}$ in $X_{\mu }.$ Suppose on the
contrary. Then%
\begin{eqnarray*}
\left\Vert v_{0}\right\Vert _{\mu }^{2}-\lambda \int_{\mathbb{R}%
^{3}}fv_{0}^{2}dx &=&\int_{\mathbb{R}^{3}}|\nabla v_{0}|^{2}dx-\lambda \int_{%
\mathbb{R}^{3}}fv_{0}^{2}dx \\
&<&\liminf_{n\rightarrow \infty }\left( \left\Vert v_{n}\right\Vert _{\mu
_{n}}^{2}-\lambda \int_{\mathbb{R}^{3}}fv_{n}^{2}dx\right) \leq 0,
\end{eqnarray*}%
since $\int_{\mathbb{R}^{3}}Vv_{0}^{2}dx=0.$ Hence $\frac{v_{0}}{\left\Vert
v_{0}\right\Vert _{\mu }}\in \overline{\Lambda _{\mu }^{-}}\cap \overline{%
\Theta _{\mu }^{+}}\left( p\right) $ which is impossible. Since $%
v_{n}\rightarrow v_{0}$ in $X_{\mu }$, then $\left\Vert v_{0}\right\Vert
_{\mu }=1$. Hence $v_{0}\in \Theta _{\mu }^{0}\left( p\right) $ and so $%
v_{0}\in \overline{\Theta _{\mu }^{+}}\left( p\right) $. Moreover,%
\begin{equation*}
\left\Vert v_{0}\right\Vert _{\mu }^{2}-\lambda \int_{\mathbb{R}%
^{3}}fv_{0}^{2}dx=\lim\limits_{n\rightarrow \infty }\left( \left\Vert
v_{n}\right\Vert _{\mu }^{2}-\lambda \int_{\mathbb{R}^{3}}fv_{n}^{2}dx%
\right) \leq 0.
\end{equation*}%
and so $v_{0}\in \overline{\Lambda _{\mu }^{-}}.$ Thus, $v_{0}\in \overline{%
\Lambda _{\mu }^{-}}\cap \overline{\Theta _{\mu }^{+}}\left( p\right) $
which is impossible. Therefore, we can conclude that $\mathbf{N}_{\mu
,\lambda }^{+}\;$is\ uniform bounded for $\mu >0$ sufficiently large. This
completes the proof.
\end{proof}

When $\mathbf{N}_{\mu ,\lambda }^{+}$ and $\mathbf{N}_{\mu ,\lambda }^{-}$
are separated and $\mathbf{N}_{\mu ,\lambda }^{0}=\emptyset ,$ any non-zero
minimizer for $J_{\mu ,\lambda }$ on $\mathbf{N}_{\mu ,\lambda }^{+}$ (or on
$\mathbf{N}_{\mu ,\lambda }^{-}$ ) is also a local minimizer on $\mathbf{N}%
_{\mu ,\lambda }$ and so will be a critical point for $J_{\mu ,\lambda }$ on
$\mathbf{N}_{\mu ,\lambda }$ and a nontrivial solution of Equation $\left(
E_{\mu ,\lambda }\right) .$ Since $\int_{\mathbb{R}^{3}}Q\phi _{1}^{p}dx<0$,
we can obtain that $\Lambda _{\mu }^{-}\cap \Theta _{\mu }^{-}\left(
p\right) \neq \emptyset $ for all $\mu >0.$ Furthermore, we have the
following result.

\begin{theorem}
\label{t4-2}Suppose that $4<p<6$ and $\int_{\Omega }Q\phi _{1}^{p}dx<0.$
Then for each $\lambda _{1}\left( f_{\Omega }\right) \leq \lambda <\lambda
_{1}\left( f_{\Omega }\right) +\delta _{0},$ there exists a minimizer of $%
J_{\mu ,\lambda }(u)$ on $\mathbf{N}_{\mu ,\lambda }^{+}$ for $\mu >0$
sufficiently large.
\end{theorem}

\begin{proof}
By Lemmas \ref{g4-3} and \ref{g4-2}, $\mathbf{N}_{\mu ,\lambda }^{+}\neq
\emptyset $ and $\mathbf{N}_{\mu ,\lambda }^{+}$ is uniformly bounded for $%
\mu >0$ sufficiently large. Then there exists $C_{a,\lambda }>0$ such that $%
\left\Vert u\right\Vert _{\mu }\leq C_{a,\lambda }$ for all $u\in \mathbf{N}%
_{\mu ,\lambda }^{+}$. Hence, making use of $\left( \ref{11}\right) $, for $%
u\in \mathbf{N}_{\mu ,\lambda }^{+}$ we have%
\begin{align}
J_{\mu ,\lambda }(u)& \geq -\frac{a}{4}\left\Vert u\right\Vert _{\mu }^{4}-%
\frac{\left( p-2\right) \left\Vert Q\right\Vert _{\infty }}{2p}\int_{\mathbb{%
R}^{3}}\left\vert u\right\vert ^{p}dx  \notag \\
& \geq -\frac{a}{4}C_{a,\lambda }^{4}-\frac{\left( p-2\right) \left\Vert
Q\right\Vert _{\infty }}{2pS^{p}}\left\vert \left\{ V<c\right\} \right\vert
^{\frac{6-p}{6}}C_{a,\lambda }^{p}.  \label{3-0}
\end{align}%
Thus, $J_{\mu ,\lambda }$ is bounded from below on $\mathbf{N}_{\mu ,\lambda
}^{+}$ and so $\inf_{u\in \mathbf{N}_{\mu ,\lambda }^{+}}J_{\mu ,\lambda
}(u) $ is finite. Since $\int_{\Omega }Q\phi _{1}^{p}dx<0$ and $\int_{\Omega
}\left\vert \nabla \phi _{1}\right\vert ^{2}dx-\lambda \int_{\Omega
}f_{\Omega }\phi _{1}^{2}dx<0,$ which indicates that the function $h_{\phi
_{1}}\left( t\right) =J_{\mu ,\lambda }\left( t\phi _{1}\right) $ have $%
t_{0}^{+}>0$ and $\kappa _{0}<0$ are independent of $\mu $ such that $%
t_{0}^{+}\phi _{1}\in \mathbf{N}_{\mu ,\lambda }^{+}$ and
\begin{equation*}
\inf_{0<t<\infty }h_{\phi _{1}}\left( t\right) =h_{\phi _{1}}\left(
t_{0}^{+}\right) =\kappa _{0}<0.
\end{equation*}%
This implies that
\begin{equation}
\inf_{u\in \mathbf{N}_{\mu ,\lambda }^{+}}J_{\mu ,\lambda }\left( u\right)
\leq \kappa _{0}<0\text{ for }\mu >0\text{ sufficiently large.}  \label{32}
\end{equation}%
Then by the Ekeland variational principle \cite{E}, there exists a
minimizing sequence $\left\{ u_{n}\right\} \subset \mathbf{N}_{\mu ,\lambda
}^{+}$ such that%
\begin{equation*}
\lim_{n\rightarrow \infty }J_{\mu ,\lambda }(u_{n})=\inf_{u\in \mathbf{N}%
_{\mu ,\lambda }^{+}}J_{\mu ,\lambda }(u)\leq \kappa _{0}\text{ and }J_{\mu
,\lambda }^{\prime }(u_{n})=o\left( 1\right) .
\end{equation*}%
Since $\left\Vert u_{n}\right\Vert _{\mu }\leq C_{a,\lambda }.$ Thus, by
Proposition \ref{p1}, there exist a subsequence $\left\{ u_{n}\right\} $ and
$u_{0}\in X$ such that $J_{\mu ,\lambda }^{\prime }(u_{0})=0$ and $%
u_{n}\rightarrow u_{0}$ strongly in $X_{\mu }$ for $\mu >0$ sufficiently
large, which implies that $J_{\mu ,\lambda }$ has minimizer $u_{0}$ in $%
\mathbf{N}_{\mu ,\lambda }^{+}$ for $\mu $ sufficiently large, and so
\begin{equation*}
J_{\mu ,\lambda }(u_{0})=\lim_{n\rightarrow \infty }J_{\mu ,\lambda
}(u_{n})=\inf_{u\in \mathbf{N}_{\mu ,\lambda }^{+}}J_{\mu ,\lambda }(u)\leq
\kappa _{0}<0,
\end{equation*}%
which implies that $u_{0}$ is a minimizer on $\mathbf{N}_{\mu ,\lambda }^{+}$
for $\mu >0$ sufficiently large.
\end{proof}

We now turn our attention to $\mathbf{N}_{\mu ,\lambda }^{-}.$ Since
\begin{equation}
J_{\mu ,\lambda }(u)=\frac{1}{4}\left( \left\Vert u\right\Vert _{\mu
}^{2}-\lambda \int_{\mathbb{R}^{3}}fu^{2}dx\right) +\left( \frac{1}{4}-\frac{%
1}{p}\right) \int_{\mathbb{R}^{3}}Q\left\vert u\right\vert ^{p}dx>0\text{
for all }u\in \mathbf{N}_{\mu ,\lambda }^{-},  \label{4-3}
\end{equation}%
we have $\inf_{u\in \mathbf{N}_{\mu ,\lambda }^{-}}J_{\mu ,\lambda }(u)\geq
0 $ for all $\mu >0.$ Since $\max \left\{ Q,0\right\} \not\equiv 0$ in\ $%
\Omega ,$ similar to the arguments in $\left( \ref{5-1}\right) ,$ there
exists $\overline{D}_{a,\lambda }>0$ independent of $\mu $ such that $%
\inf_{u\in \mathbf{N}_{\mu ,\lambda }^{-}}J_{\mu ,\lambda }(u)<\overline{D}%
_{a,\lambda }$ and the set
\begin{equation*}
\left\{ J_{\mu ,\lambda }<\overline{D}_{a,\lambda }\right\} :=\left\{ u\in
\mathbf{N}_{\mu ,\lambda }^{-}:J_{\mu ,\lambda }(u)<\overline{D}_{a,\lambda
}\right\} \neq \emptyset \text{ for }\mu >0\text{ sufficiently large.}
\end{equation*}%
Furthermore, we have the following results.

\begin{lemma}
\label{t4-1}Suppose that $4<p<6$ and $\int_{\Omega }Q\phi _{1}^{p}dx<0.$
Then for each $\lambda _{1}\left( f_{\Omega }\right) \leq \lambda <\lambda
_{1}\left( f_{\Omega }\right) +\delta _{0},$ we have the following results.%
\newline
$(i)$ There exists $C_{a,\lambda }>0$ such that $\left\Vert u\right\Vert
_{\mu }\leq C_{a,\lambda }$ for all $u\in \left\{ J_{\mu ,\lambda }<%
\overline{D}_{a,\lambda }\right\} $ and for $\mu >0$ sufficiently large.%
\newline
$(ii)$ We have
\begin{equation*}
\inf_{u\in \mathbf{N}_{\mu ,\lambda }^{-}}J_{\mu ,\lambda }(u)=\inf_{u\in
\left\{ J_{\mu ,\lambda }<\overline{D}_{a,\lambda }\right\} }J_{\mu ,\lambda
}(u)>0
\end{equation*}%
for $\mu >0$ sufficiently large.
\end{lemma}

\begin{proof}
$(i)$ Suppose on the contrary. Then there exist a sequences $\left\{ \mu
_{n}\right\} \subset \mathbb{R}^{+}$ with $\mu _{n}\rightarrow \infty $ and
a sequence $u_{n}\in \left\{ J_{\mu _{n},\lambda }<\overline{D}_{a,\lambda
}\right\} $ such that $\Vert u_{n}\Vert _{\mu _{n}}\rightarrow \infty $ as $%
n\rightarrow \infty .$ Let $v_{n}=\frac{u_{n}}{\Vert u_{n}\Vert _{\mu _{n}}}%
. $ Then by Theorem \ref{g4-0}, we may assume that there exist subsequence $%
\left\{ v_{n}\right\} $ and $v_{0}\in H_{0}^{1}\left( \Omega \right) $ such
that $v_{n}\rightharpoonup v_{0}$ in $X$ and $v_{n}\rightarrow v_{0}$ in $%
L^{r}\left( \mathbb{R}^{3}\right) $ for all $2\leq r<6.$ Then%
\begin{equation}
\lim_{n\rightarrow \infty }\int_{\mathbb{R}^{3}}Q|v_{n}|^{p}dx=\int_{\mathbb{%
R}^{3}}Q|v_{0}|^{p}dx  \label{4-9}
\end{equation}%
and%
\begin{equation}
\lim_{n\rightarrow \infty }\int_{\mathbb{R}^{3}}fv_{n}^{2}dx=\int_{\mathbb{R}%
^{3}}fv_{0}^{2}dx.  \label{4-6}
\end{equation}%
Dividing $\left( \ref{4-3}\right) $ by $\Vert u_{n}\Vert _{\mu _{n}}^{2}$
gives%
\begin{equation}
\frac{J_{\mu _{n},\lambda }(u_{n})}{\Vert u_{n}\Vert _{\mu _{n}}^{2}}=\frac{1%
}{4}\left( 1-\lambda \int_{\mathbb{R}^{3}}fv_{n}^{2}dx\right) +\left( \frac{1%
}{4}-\frac{1}{p}\right) \Vert u_{n}\Vert _{\mu _{n}}^{p-2}\int_{\mathbb{R}%
^{3}}Q\left\vert v_{n}\right\vert ^{p}dx.  \label{4-7}
\end{equation}%
Since $\Vert u_{n}\Vert _{\mu _{n}}\rightarrow +\infty $ and $\frac{J_{\mu
_{n},\lambda }(u_{n})}{\Vert u_{n}\Vert _{\mu _{n}}^{2}}\rightarrow 0$, by $%
\left( \ref{4-9}\right) -\left( \ref{4-7}\right) ,$ we have that $\int_{%
\mathbb{R}^{3}}Q|v_{n}|^{p}dx\rightarrow 0$ and so $\int_{\mathbb{R}%
^{3}}Q|v_{0}|^{p}dx=0.$ We now show that for each $\mu >0,$ we have $%
v_{n}\rightarrow v_{0}$ in $X_{\mu }$. Suppose otherwise, then by $\left( %
\ref{4-6}\right) $ and $\left( \ref{4-7}\right) ,$ there exists $\mu >0$
such that
\begin{eqnarray*}
\int_{\mathbb{R}^{3}}|\nabla v_{0}|^{2}dx-\lambda \int_{\mathbb{R}%
^{3}}fv_{0}^{2}dx &=&\Vert v_{0}\Vert _{\mu }^{2}-\lambda \int_{\mathbb{R}%
^{3}}fv_{0}^{2}dx \\
&<&\liminf_{n\rightarrow \infty }\left( \Vert v_{n}\Vert _{\mu
_{n}}^{2}-\lambda \int_{\mathbb{R}^{3}}fv_{n}^{2}dx\right) =0.
\end{eqnarray*}%
Thus, $v_{0}\neq 0$ and
\begin{equation*}
\frac{v_{0}}{\left\Vert v_{0}\right\Vert _{\mu }}=\frac{v_{0}}{\left(
\int_{\Omega }|\nabla v_{0}|^{2}dx\right) ^{1/2}}\in \overline{\Lambda _{\mu
}^{-}}\cap \overline{\Theta _{\mu }^{+}}\left( p\right) ,
\end{equation*}%
which is impossible. Hence $v_{n}\rightarrow v_{0}$ in $X_{\mu }.$ It
follows that $\Vert v_{0}\Vert _{\mu }=1,\int_{\mathbb{R}^{3}}Vv_{0}^{2}dx=0$
and%
\begin{equation*}
\Vert v_{0}\Vert _{\mu }^{2}-\lambda \int_{\mathbb{R}^{3}}fv_{0}^{2}dx+a%
\left\Vert v_{0}\right\Vert _{D^{1,2}}^{4}=\int_{\mathbb{R}%
^{3}}Q|v_{0}|^{p}dx=0.
\end{equation*}%
Thus, for every $\mu >0,$ there holds $v_{0}\in \Lambda _{\mu }^{0}\cap
\Theta _{\mu }^{0}\left( p\right) $ which is impossible as $\overline{%
\Lambda _{\mu }^{-}}\cap \overline{\Theta _{\mu }^{+}}\left( p\right)
=\emptyset .$ Hence there exists $C_{a,\lambda }>0$ such that $\left\Vert
u\right\Vert _{\mu }\leq C_{a,\lambda }$ for all $u\in \left\{ J_{\mu
,\lambda }<\overline{D}_{a,\lambda }\right\} $ and for $\mu >0$ sufficiently
large.\newline
$(ii)$ Since $\inf_{u\in \mathbf{N}_{\mu ,\lambda }^{-}}J_{\mu ,\lambda
}(u)\geq 0,$ by Lemma \ref{g4-2} and the Ekeland variational principle \cite%
{E}, there exists a minimizing sequence $\left\{ u_{n}\right\} \subset
\left\{ J_{\mu ,\lambda }<\overline{D}_{a,\lambda }\right\} \subset \mathbf{N%
}_{\mu ,\lambda }^{-}$ such that%
\begin{equation*}
\lim_{n\rightarrow \infty }J_{\mu ,\lambda }(u_{n})=\inf_{u\in \mathbf{N}%
_{\mu ,\lambda }^{-}}J_{\mu ,\lambda }(u)\text{ and }J_{\mu ,\lambda
}^{\prime }(u_{n})=o\left( 1\right) .
\end{equation*}%
By part $\left( i\right) $, there exists $C_{a,\lambda }>0$ such that $%
\left\Vert u_{n}\right\Vert _{\mu }\leq C_{a,\lambda }.$ Then by Proposition %
\ref{p1}, and so there exist a subsequence $\left\{ u_{n}\right\} $ and $%
u_{0}\in \mathbf{N}_{\mu ,\lambda }^{-}$ such that $u_{n}\rightarrow u_{0}$
in $X_{\mu }$ and $J_{\mu ,\lambda }^{\prime }(u_{0})=0$ for $\mu >0$
sufficiently large. If $\inf_{u\in \mathbf{N}_{\mu ,\lambda }^{-}}J_{\mu
,\lambda }(u)=0,$ then
\begin{equation*}
J_{\mu ,\lambda }(u_{0})=\lim_{n\rightarrow \infty }J_{\mu ,\lambda
}(u_{n})=\inf_{u\in \mathbf{N}_{\mu ,\lambda }^{-}}J_{\mu ,\lambda }(u)=0.
\end{equation*}%
It then follows exactly as in the proof in part $(i)$ that $\frac{u_{0}}{%
\Vert u_{0}\Vert _{\mu }}\in \Lambda _{\mu }^{0}\cap \Theta _{\mu
}^{0}\left( p\right) $ which is impossible as $\overline{\Lambda _{\mu }^{-}}%
\cap \overline{\Theta _{\mu }^{+}}\left( p\right) =\emptyset .$ This
completes the proof.
\end{proof}

\begin{theorem}
\label{t4-3}Suppose that $4<p<6$ and $\int_{\Omega }Q\phi _{1}^{p}dx<0.$
Then for each $\lambda _{1}\left( f_{\Omega }\right) \leq \lambda <\lambda
_{1}\left( f_{\Omega }\right) +\delta _{0},$ there exists a minimizer of $%
J_{\mu ,\lambda }(u)$ on $\mathbf{N}_{\mu ,\lambda }^{-}$ for $\mu >0$
sufficiently large.
\end{theorem}

\begin{proof}
By Lemmas \ref{g4-2}, \ref{t4-1} and the Ekeland variational principle \cite%
{E}, for each $\mu >\overline{\mu }_{0}\left( \lambda \right) $ there exists
a minimizing sequence $\left\{ u_{n}\right\} \subset \left\{ J_{\mu ,\lambda
}<\overline{D}_{a,\lambda }\right\} \subset \mathbf{N}_{\mu ,\lambda }^{-}$
such that%
\begin{equation*}
\lim_{n\rightarrow \infty }J_{\mu ,\lambda }(u_{n})=\inf_{u\in \mathbf{N}%
_{\mu ,\lambda }^{-}}J_{\mu ,\lambda }(u)>0\text{ and }J_{\mu ,\lambda
}^{\prime }(u_{n})=o\left( 1\right) .
\end{equation*}%
Since $\inf_{u\in \mathbf{N}_{\mu ,\lambda }^{-}}J_{\mu ,\lambda }(u)<%
\overline{D}_{a,\lambda },$ by Lemma \ref{t4-1} $\left( i\right) $, there
exists a positive constant $C_{a,\lambda }$ independent of $\mu $ such that $%
\left\Vert u_{n}\right\Vert _{\mu }\leq C_{a,\lambda }.$ Thus, by
Proposition \ref{p1}, there exist a subsequence $\left\{ u_{n}\right\} $ and
$u_{0}\in X$ such that $J_{\mu ,\lambda }^{\prime }(u_{0})=0$ and $%
u_{n}\rightarrow u_{0}$ strongly in $X_{\mu }$ for $\mu >0$ sufficiently
large, which implies that $J_{\mu ,\lambda }$ has minimizer $u_{0}$ in $\in
\mathbf{N}_{\mu ,\lambda }^{-}$ for $\mu $ sufficiently large, and so
\begin{equation*}
J_{\mu ,\lambda }(u_{0})=\lim_{n\rightarrow \infty }J_{\mu ,\lambda
}(u_{n})=\inf_{u\in \mathbf{N}_{\mu ,\lambda }^{-}}J_{\mu ,\lambda }(u)<%
\overline{D}_{a,\lambda },
\end{equation*}%
which implies that $u_{0}$ is a minimizer on $\mathbf{N}_{\mu ,\lambda
}^{-}. $ This completes the proof.
\end{proof}

\bigskip

\textbf{We are now ready to prove Theorem \ref{t2}: }By Theorems \ref{t4-2}
and \ref{t4-3}, there exists $\delta _{0}>0$ such that, when $\lambda
_{1}\left( f_{\Omega }\right) \leq \lambda <\lambda _{1}\left( f_{\Omega
}\right) +\delta _{0},J_{\mu ,\lambda }$ has minimizers in each of $\mathbf{N%
}_{\mu ,\lambda }^{+}$ and $\mathbf{N}_{\mu ,\lambda }^{-}$ for $\mu $
sufficiently large. Since $J_{\mu ,\lambda }(u)=J_{\mu ,\lambda }(|u|),$ we
may assume that these minimizers are positive. Moreover, by Lemma \ref{g4-3}
we may assume that $\mathbf{N}_{\mu ,\lambda }^{+}$ and $\mathbf{N}_{\mu
,\lambda }^{-}$ are separated and $\mathbf{N}_{\mu ,\lambda }^{0}=\emptyset $%
. It follows that the minimizers are local minimizers in $\mathbf{N}_{\mu
,\lambda }$ which do not lie in $\mathbf{N}_{\mu ,\lambda }^{0}$ and so by
Lemma \ref{g2-1}, they are positive solutions of Equation $\left( E_{\mu
,\lambda }\right) .$

\section{The case when $N=3$ and $p=4$}

\textbf{We are now ready to prove Theorem \ref{t3}: }$\left( i\right) $ When
$0<\lambda <\lambda _{1}\left( f_{\Omega }\right) .$ Similar to the argument
of proofs in Lemma \ref{g3-1} and Theorem \ref{t1}, $J_{\mu ,\lambda }$ has
minimizer $u_{0}$ in $\mathbf{N}_{\mu ,\lambda }=\mathbf{N}_{\mu ,\lambda
}^{-}$ for $\mu $ sufficiently large. Since $J_{\mu ,\lambda }(u)=J_{\mu
,\lambda }(|u|),$ by Lemma \ref{g2-1}, we may assume that $u_{0}$ is a
positive solution of Equation $\left( E_{\mu ,\lambda }\right) $ such that $%
J_{\mu ,\lambda }(u_{\mu }^{-})=\inf_{u\in \mathbf{N}_{\mu ,\lambda
}^{-}}J_{\mu ,\lambda }(u)>0.$

$\left( ii\right) $ Since $\Gamma _{0}<\infty ,$ by Lemma \ref{g2-5} $\left(
ii\right) $ for each $a>\Gamma _{0},$ we have $\mathbf{N}_{\mu ,\lambda
}=\emptyset $ for $\mu $ sufficiently large, this implies that for each $%
a>\Gamma _{0}$ and $0<\lambda <\lambda _{1}\left( f_{\Omega }\right) ,$
Equation $\left( E_{\mu ,\lambda }\right) $ does not admits nontrivial
solution.

$\left( iii\right) $ Since $\Gamma _{0}<\infty ,$ by Lemma \ref{g2-5} $%
\left( iii\right) ,$ for each $a>\Gamma _{0}$ and $\lambda \geq \lambda
_{1}\left( f_{\Omega }\right) ,$ we have $\mathbf{N}_{\mu ,\lambda }=\mathbf{%
N}_{\mu ,\lambda }^{+}$ and $\mathbf{N}_{\mu ,\lambda }^{+}=\{t_{\min
}(u)u:u\in \Theta _{\mu }^{-}\left( p\right) \}$ for $\mu >0$ sufficiently
large. Now, we will show that $\mathbf{N}_{\mu ,\lambda }^{+}\;$is\ uniform
bounded for $\mu >0$ sufficiently large. Suppose on the contrary. Then there
exist sequences $\left\{ \mu _{n}\right\} \subset \mathbb{R}$ and $%
\{u_{n}\}\subset \mathbf{N}_{\mu _{n},\lambda }^{+}$ such that $\mu
_{n}\rightarrow \infty $ and $\Vert u_{n}\Vert _{\mu _{n}}\rightarrow \infty
$ as $n\rightarrow \infty .$ Clearly,%
\begin{equation}
\left\Vert u_{n}\right\Vert _{\mu _{n}}^{2}-\lambda \int_{\mathbb{R}%
^{3}}fu_{n}^{2}dx=\int_{\mathbb{R}^{3}}Q|u_{n}|^{4}dx-a\left\Vert
u_{n}\right\Vert _{D^{1,2}}^{4}<0.  \label{5-2}
\end{equation}%
Let $v_{n}=\frac{u_{n}}{\Vert u_{n}\Vert _{\mu _{n}}}.$ Then by Theorem \ref%
{g4-0}, we may assume that there exists $v_{0}\in H_{0}^{1}\left( \Omega
\right) $ such that
\begin{equation*}
v_{n}\rightharpoonup v_{0}\text{ in }X;v_{n}\rightarrow v_{0}\text{ in }%
L^{r}\left( \mathbb{R}^{3}\right) \text{ for all }2<r\leq 6.
\end{equation*}%
Thus,
\begin{equation}
\lim_{n\rightarrow \infty }\int_{\mathbb{R}^{3}}fv_{n}^{2}dx=\int_{\mathbb{R}%
^{3}}fv_{0}^{2}dx  \label{5-3}
\end{equation}%
and%
\begin{equation}
\lim_{n\rightarrow \infty }\int_{\mathbb{R}^{3}}Q\left\vert v_{n}\right\vert
^{p}dx=\int_{\mathbb{R}^{3}}Q\left\vert v_{0}\right\vert ^{p}dx.  \label{5-4}
\end{equation}%
Moreover, by Fatou's Lemma,%
\begin{equation}
\int_{\mathbb{R}^{3}}|\nabla v_{0}|^{2}dx\leq \liminf_{n\rightarrow \infty
}\int_{\mathbb{R}^{3}}|\nabla v_{n}|^{2}dx.  \label{5-5}
\end{equation}%
Dividing $\left( \ref{5-2}\right) $ by $\Vert u_{n}\Vert _{\mu _{n}}^{2}$
gives%
\begin{equation}
\left\Vert v_{n}\right\Vert _{\mu _{n}}^{2}-\lambda \int_{\mathbb{R}%
^{3}}fv_{n}^{2}dx=\Vert u_{n}\Vert _{\mu _{n}}^{2}\left( \int_{\mathbb{R}%
^{3}}Q|v_{n}|^{4}dx-a\left\Vert v_{n}\right\Vert _{D^{1,2}}^{4}\right) <0.
\label{5-6}
\end{equation}%
Since
\begin{equation}
\lim_{n\rightarrow \infty }\left( \left\Vert v_{n}\right\Vert _{\mu
_{n}}^{2}-\lambda \int_{\mathbb{R}^{3}}fv_{n}^{2}dx\right) =1-\lambda
\lim_{n\rightarrow \infty }\int_{\mathbb{R}^{3}}fv_{n}^{2}dx=1-\lambda \int_{%
\mathbb{R}^{3}}fv_{0}^{2}dx  \label{5-7}
\end{equation}%
and $\Vert u_{n}\Vert _{\mu _{n}}\rightarrow \infty ,$ by $\left( \ref{5-4}%
\right) -\left( \ref{5-7}\right) ,$ it obtain that $\int_{\mathbb{R}%
^{3}}Q|v_{0}|^{4}dx-a\left\Vert v_{0}\right\Vert _{D^{1,2}}^{4}\geq 0$ and $%
\int_{\mathbb{R}^{3}}fv_{0}^{2}dx>0.$ Moreover, by $v_{0}\in H_{0}^{1}\left(
\Omega \right) ,\left( \ref{5-3}\right) $ and $\left( \ref{5-6}\right) ,$
for every $\mu >0,$%
\begin{eqnarray*}
\left\Vert v_{0}\right\Vert _{\mu }^{2}-\lambda \int_{\mathbb{R}%
^{3}}fv_{0}^{2}dx &=&\int_{\mathbb{R}^{3}}|\nabla v_{0}|^{2}dx-\lambda \int_{%
\mathbb{R}^{3}}fv_{0}^{2}dx \\
&\leq &\liminf_{n\rightarrow \infty }\left( \left\Vert v_{n}\right\Vert
_{\mu _{n}}^{2}-\lambda \int_{\mathbb{R}^{3}}fv_{n}^{2}dx\right) \leq 0.
\end{eqnarray*}%
We now show that $v_{n}\rightarrow v_{0}$ in $X_{\mu }.$ Suppose on the
contrary. Then%
\begin{eqnarray*}
\left\Vert v_{0}\right\Vert _{\mu }^{2}-\lambda \int_{\mathbb{R}%
^{3}}fv_{0}^{2}dx &=&\int_{\mathbb{R}^{3}}|\nabla v_{0}|^{2}dx-\lambda \int_{%
\mathbb{R}^{3}}fv_{0}^{2}dx \\
&<&\liminf_{n\rightarrow \infty }\left( \left\Vert v_{n}\right\Vert _{\mu
_{n}}^{2}-\lambda \int_{\mathbb{R}^{3}}fv_{n}^{2}dx\right) \leq 0,
\end{eqnarray*}%
since $\int_{\mathbb{R}^{3}}Vv_{0}^{2}dx=0.$ Hence $\frac{v_{0}}{\left\Vert
v_{0}\right\Vert _{\mu }}\in \overline{\Lambda _{\mu }^{-}}\cap \overline{%
\Theta _{\mu }^{+}}\left( p\right) $ which is impossible. Since $%
v_{n}\rightarrow v_{0}$ in $X_{\mu }$, then $\left\Vert v_{0}\right\Vert
_{\mu }=1$. Hence $v_{0}\in \Theta _{\mu }^{0}\left( p\right) $ and so $%
v_{0}\in \overline{\Theta _{\mu }^{+}}\left( p\right) $. Moreover,%
\begin{equation*}
\left\Vert v_{0}\right\Vert _{\mu }^{2}-\lambda \int_{\mathbb{R}%
^{3}}fv_{0}^{2}dx=\lim\limits_{n\rightarrow \infty }\left( \left\Vert
v_{n}\right\Vert _{\mu }^{2}-\lambda \int_{\mathbb{R}^{3}}fv_{n}^{2}dx%
\right) \leq 0.
\end{equation*}%
and so $v_{0}\in \overline{\Lambda _{\mu }^{-}}.$ Thus, $v_{0}\in \overline{%
\Lambda _{\mu }^{-}}\cap \overline{\Theta _{\mu }^{+}}\left( p\right) $
which is impossible. Therefore, we can conclude that $\mathbf{N}_{\mu
,\lambda }^{+}\;$is\ uniform bounded for $\mu >0$ sufficiently large. Then
similar to the argument of proof in Theorem \ref{t4-2}, $J_{\mu ,\lambda }$
has minimizer $u_{\mu }^{+}$ in $\mathbf{N}_{\mu ,\lambda }=\mathbf{N}_{\mu
,\lambda }^{+}$ for $\mu $ sufficiently large such that $J_{\mu ,\lambda
}(u_{\mu }^{+})<0$. Since $J_{\mu ,\lambda }(u_{\mu }^{+})=J_{\mu ,\lambda
}(|u_{\mu }^{+}|),$ by Lemma \ref{g2-1}, we may assume that $u_{\mu }^{+}$
is a positive solution of Equation $\left( E_{\mu ,\lambda }\right) .$

\bigskip

\textbf{We are now ready to prove Theorem \ref{t4}: }Since $\lambda
_{1}^{-2}\left( f_{\Omega }\right) \int_{\Omega }Q\phi _{1}^{4}dx<a<\Gamma
_{0},$%
\begin{equation*}
\Phi _{p}\left( \phi _{1}\right) =\int_{\Omega }Q|\phi _{1}|^{p}dx-a\left(
\int_{\Omega }\left\vert \nabla \phi _{1}\right\vert ^{2}dx\right) ^{2}<0%
\text{ for }p=4.
\end{equation*}%
By Lemma \ref{g4-4}, there exists $\delta _{0}>0$ such that for every $%
\lambda _{1}\left( f_{\Omega }\right) \leq \lambda <\lambda _{1}\left(
f_{\Omega }\right) +\delta _{0}$, $\mathbf{N}_{\mu ,\lambda }^{\pm }$ are
nonempty sets and $\mathbf{N}_{\mu ,\lambda }=\mathbf{N}_{\mu ,\lambda
}^{+}\cup \mathbf{N}_{\mu ,\lambda }^{-}$ for $\mu >0$ sufficiently large.
Then similar to the argument of proof in Theorem \ref{t2}, Equation $\left(
E_{\mu ,\lambda }\right) $ has two positive solutions $u_{\mu }^{-}$ and $%
u_{\mu }^{+}$ satisfying $J_{\mu ,\lambda }\left( u_{\mu }^{+}\right)
<0<J_{\mu ,\lambda }\left( u_{\mu }^{-}\right) $ for $\mu >0$ sufficiently
large.

\section{The case when $N\geq 3$ and $2<p<\min \left\{ 4,2^{\ast }\right\} $}

\subsection{The proof of Theorem \protect\ref{t5-0}}

\textbf{We are now ready to prove Theorem \ref{t5-0}:}

For $0<\lambda <\lambda _{1}\left( f\right) $ and $u\in X\backslash \left\{
0\right\} ,$ we know that $tu\in \mathbf{N}_{\mu }^{0}$ if and only if $%
h_{tu}^{\prime }\left( 1\right) =h_{tu}^{\prime \prime }\left( 1\right) =0,$
i.e., the following system of equations is satisfied:%
\begin{equation}
\left\{
\begin{array}{c}
at^{3}\left\Vert u\right\Vert _{D^{1,2}}^{4}+t\left( \left\Vert u\right\Vert
_{\mu }^{2}-\lambda \int_{\mathbb{R}^{N}}fu^{2}dx\right) -t^{p-1}\int_{%
\mathbb{R}^{N}}Q\left\vert u\right\vert ^{p}dx=0, \\
3at^{2}\left\Vert u\right\Vert _{D^{1,2}}^{4}+\left( \left\Vert u\right\Vert
_{\mu }^{2}-\lambda \int_{\mathbb{R}^{N}}fu^{2}dx\right) -(p-1)t^{p-2}\int_{%
\mathbb{R}^{N}}Q\left\vert u\right\vert ^{p}dx=0.%
\end{array}%
\right.  \label{15-2}
\end{equation}%
By solving the system $(\ref{15-2})$ with respect to the variables $t$ and $%
a,$ we have%
\begin{equation*}
t(u)=\left( \frac{2\left( \left\Vert u\right\Vert _{\mu }^{2}-\lambda \int_{%
\mathbb{R}^{N}}fu^{2}dx\right) }{(4-p)\int_{\mathbb{R}^{N}}Q\left\vert
u\right\vert ^{p}dx}\right) ^{1/(p-2)}
\end{equation*}%
and%
\begin{equation*}
a(u)=\frac{p-2}{4-p}\left( \frac{4-p}{2}\right) ^{2/(p-2)}\overline{A}%
_{\lambda }(u),
\end{equation*}%
where%
\begin{equation}
\overline{A}_{\lambda }(u)=\frac{\left( \int_{\mathbb{R}^{N}}Q|u|^{p}dx%
\right) ^{2/(p-2)}}{\left\Vert u\right\Vert _{D^{1,2}}^{4}\left( \left\Vert
u\right\Vert _{\mu }^{2}-\lambda \int_{\mathbb{R}^{N}}fu^{2}dx\right)
^{(4-p)/(p-2)}}.  \label{15-1}
\end{equation}%
We conclude that $a(u)$ is the unique parameter $a>0$ for which the fibering
map $h_{u}$ has a critical point with second derivative zero at $t(u)$.
Hence, if $a>a(u)$, then $h_{u}$ is increasing on $(0,\infty )$ and has no
critical point. Moreover, for $0<\lambda <\lambda _{1}\left( f\right) ,$ we
define%
\begin{equation}
\overline{\mathbf{A}}_{\lambda }=\frac{p-2}{4-p}\left( \frac{4-p}{p}\right)
^{2/(p-2)}\sup_{u\in X\backslash \left\{ 0\right\} }\overline{A}_{\lambda
}(u).  \label{15-0}
\end{equation}%
Note that by $\left( \ref{10}\right) $ and the H\"{o}lder and Sobolev
inequalities,%
\begin{eqnarray*}
\overline{A}_{\lambda }(u) &\leq &\frac{\left( \left\Vert Q\right\Vert
_{\infty }\left( \int_{\left\{ V\geq c\right\} }u^{2}dx+\int_{\left\{
V<c\right\} }u^{2}dx\right) ^{\frac{2^{\ast }-p}{2^{\ast }-2}}\left( \frac{%
\left\Vert u\right\Vert _{D^{1,2}}^{2^{\ast }}}{S^{2^{\ast }}}\right) ^{%
\frac{p-2}{2^{\ast }-2}}\right) ^{2/(p-2)}}{\left\Vert u\right\Vert
_{D^{1,2}}^{4}\left( \left\Vert u\right\Vert _{\mu }^{2}-\lambda \int_{%
\mathbb{R}^{N}}fu^{2}dx\right) ^{(4-p)/(p-2)}} \\
&\leq &\left( \frac{\widetilde{\lambda }_{1,\mu }\left( f\right) }{%
\widetilde{\lambda }_{1,\mu }\left( f\right) -\lambda }\right)
^{(4-p)/(p-2)}\left( \frac{\left\Vert Q\right\Vert _{\infty }\left\vert
\left\{ V<c\right\} \right\vert ^{\frac{2^{\ast }-p}{2^{\ast }}}\left\Vert
u\right\Vert _{D^{1,2}}^{\frac{N\left( p-2\right) }{2}}\left\Vert
u\right\Vert _{\mu }^{\frac{\left( N-2\right) \left( 2^{\ast }-p\right) }{2}}%
}{S^{p}\left\Vert u\right\Vert _{D^{1,2}}^{2\left( p-2\right) }\left\Vert
u\right\Vert _{\mu }^{4-p}}\right) ^{2/\left( p-2\right) } \\
&\leq &\left( \frac{\widetilde{\lambda }_{1,\mu }\left( f\right) }{%
\widetilde{\lambda }_{1,\mu }\left( f\right) -\lambda }\right)
^{(4-p)/(p-2)}\left( \frac{\left\Vert Q\right\Vert _{\infty }\left\vert
\left\{ V<c\right\} \right\vert ^{\frac{2^{\ast }-p}{2^{\ast }}}}{S^{p}}%
\right) ^{2/(p-2)}\text{ for al }\mu >\overline{\mu }_{0}\left( \lambda
\right) ,
\end{eqnarray*}%
which implies that for each $0<\lambda <\lambda _{1}\left( f\right) ,$%
\begin{equation*}
\overline{\mathbf{A}}_{\lambda }<\frac{1}{2}\left( \frac{\left( 4-p\right)
\lambda _{1}\left( f_{\Omega }\right) }{p\left( \lambda _{1}\left( f_{\Omega
}\right) -\lambda \right) }\right) ^{(4-p)/(p-2)}\left( \frac{\left\Vert
Q\right\Vert _{\infty }\left\vert \left\{ V<c\right\} \right\vert ^{\frac{%
2^{\ast }-p}{2^{\ast }}}}{S^{p}}\right) ^{2/(p-2)},
\end{equation*}%
for $\mu >0$ sufficiently large. Hence, the energy functional $J_{\mu
,\lambda }$ has no any nontrivial critical points for $a>\overline{\mathbf{A}%
}_{\lambda }$ for $\mu >0$ sufficiently large. Consequently, we complete the
proof.

\subsection{The proofs of Theorems \protect\ref{t5}, \protect\ref{t5-2}}

First, we define
\begin{equation*}
\alpha _{\mu ,\lambda }^{+}=\inf_{u\in \mathbf{N}_{\mu ,\lambda }^{+}}J_{\mu
,\lambda }\left( u\right) .
\end{equation*}%
Then we have the following results.

\begin{proposition}
\label{p2}Suppose that $N=3,2<p<4$ and conditions $\left( V_{1}\right)
-\left( V_{3}\right) $ and $\left( D_{1}\right) -\left( D_{3}\right) $ hold.
Then the following statements are true.\newline
$\left( i\right) $ For each $\lambda >0$ and $a>0$, we have $\mathbf{N}_{\mu
,\lambda }^{+}$ is uniformly bounded for $\mu >0$ sufficiently large;\newline
$\left( ii\right) $ For each $\lambda >0$ and $a>0,$there exist two numbers $%
d_{0},D_{0}>0$ such that
\begin{equation*}
\inf_{u\in \mathbf{N}_{\mu ,\lambda }^{-}\cup \mathbf{N}_{\mu ,\lambda
}^{0}}J_{\mu ,\lambda }(u)\geq 0>-d_{0}>\alpha _{\mu ,\lambda }^{+}>-D_{0}%
\text{ for }\mu >0\text{ sufficiently large.}
\end{equation*}
\end{proposition}

\begin{proof}
$\left( i\right) $ Let $u\in \mathbf{N}_{\mu ,\lambda }^{+}.$ Then by $%
\left( \ref{2.2}\right) $ and the H\"{o}lder and Sobolev inequalities,%
\begin{equation}
\left\Vert u\right\Vert _{\mu }^{2}<\frac{a\left( 4-p\right) }{\left(
p-2\right) }\left( \int_{\mathbb{R}^{3}}\left\vert \nabla u\right\vert
^{2}dx\right) ^{2}+\frac{\lambda \left\Vert f\right\Vert _{L^{3/2}}}{S^{2}}%
\int_{\mathbb{R}^{3}}\left\vert \nabla u\right\vert ^{2}dx.  \label{8-14}
\end{equation}%
Moreover, using the Sobolev, H\"{o}lder and Hardy inequalities, condition $%
\left( D_{3}\right) $ and $\left( \ref{8-14}\right) $ gives%
\begin{eqnarray*}
1 &=&\frac{\int_{\mathbb{R}^{3}}Q\left\vert u\right\vert ^{p}dx+\lambda
\int_{\mathbb{R}^{3}}fu^{2}dx}{\left\Vert u\right\Vert _{\mu
}^{2}+a\left\Vert u\right\Vert _{D^{1,2}}^{4}}<\frac{\int_{\mathbb{R}%
^{3}}Q\left\vert u\right\vert ^{p}dx+\lambda \int_{\mathbb{R}^{3}}fu^{2}dx}{%
a\left\Vert u\right\Vert _{D^{1,2}}^{4}} \\
&=&\frac{1}{aS^{\frac{3(p-2)}{2}}}\left[ \frac{\int_{\left\vert x\right\vert
>R_{\ast }}Q^{\frac{4}{6-p}}u^{2}dx+\left\Vert Q\right\Vert _{\infty }^{%
\frac{4}{6-p}}\left\vert B_{R_{\ast }}\left( 0\right) \right\vert ^{\frac{2}{%
3}}S^{-2}\int_{\mathbb{R}^{3}}\left\vert \nabla u\right\vert ^{2}dx}{%
\left\Vert u\right\Vert _{D^{1,2}}^{\frac{2\left( 14-3p\right) }{6-p}}}%
\right] ^{\frac{6-p}{4}}+\frac{\lambda \left\Vert f\right\Vert _{L^{3/2}}}{%
aS^{2}\left\Vert u\right\Vert _{D^{1,2}}^{2}} \\
&\leq &\frac{1}{aS^{\frac{3(p-2)}{2}}}\left[ \frac{\int_{\left\vert
x\right\vert >R_{\ast }}\left( Vu^{2}\right) ^{\frac{2\left( 4-p\right) }{6-p%
}}\left( \frac{\left\vert u\right\vert }{\left\vert x\right\vert }\right)
^{^{\frac{2\left( p-2\right) }{6-p}}}dx}{\left\Vert u\right\Vert _{D^{1,2}}^{%
\frac{2\left( 14-3p\right) }{6-p}}}+\frac{\left\Vert Q\right\Vert _{\infty
}^{\frac{4}{6-p}}\left\vert B_{R_{\ast }}\left( 0\right) \right\vert ^{\frac{%
2}{3}}S^{-2}}{\left\Vert u\right\Vert _{D^{1,2}}^{\frac{4\left( 4-p\right) }{%
6-p}}}\right] ^{\frac{6-p}{4}}+\frac{\lambda \left\Vert f\right\Vert
_{L^{3/2}}}{aS^{2}\left\Vert u\right\Vert _{D^{1,2}}^{2}} \\
&\leq &\frac{1}{aS^{\frac{3(p-2)}{2}}}\left[ \overline{C}_{0}\left( \frac{%
\int_{\left\vert x\right\vert >R_{\ast }}V\left( x\right) u^{2}dx}{%
\left\Vert u\right\Vert _{D^{1,2}}^{4}}\right) ^{\frac{2\left( 4-p\right) }{%
6-p}}+\frac{\left\Vert Q\right\Vert _{\infty }^{\frac{4}{6-p}}\left\vert
B_{R_{\ast }}\left( 0\right) \right\vert ^{\frac{2}{3}}S^{-2}}{\left\Vert
u\right\Vert _{D^{1,2}}^{\frac{4\left( 4-p\right) }{6-p}}}\right] ^{\frac{6-p%
}{4}}+\frac{\lambda \left\Vert f\right\Vert _{L^{3/2}}}{aS^{2}\left\Vert
u\right\Vert _{D^{1,2}}^{2}} \\
&<&\frac{1}{aS^{\frac{3(p-2)}{2}}}\left[ \overline{C}_{0}\left( \frac{%
a\left( 4-p\right) }{\mu \left( p-2\right) }+\frac{\lambda \left\Vert
f\right\Vert _{L^{3/2}}}{\mu S^{2}\left\Vert u\right\Vert _{D^{1,2}}^{2}}%
\right) ^{\frac{2\left( 4-p\right) }{6-p}}+\frac{Q_{\max }^{\frac{4}{6-p}%
}\left\vert B_{R_{\ast }}\left( 0\right) \right\vert ^{\frac{2}{3}}}{%
S^{2}\left\Vert u\right\Vert _{D^{1,2}}^{\frac{4\left( 4-p\right) }{6-p}}}%
\right] ^{\frac{6-p}{4}}+\frac{\lambda \left\Vert f\right\Vert _{L^{3/2}}}{%
aS^{2}\left\Vert u\right\Vert _{D^{1,2}}^{2}}
\end{eqnarray*}%
where $\overline{C}_{0}$ is the sharp constant of Caffarelli-Kohn-Nirenberg
inequality. This implies that there exists a constant $d_{1}>0,$ dependent
on $a$ and $\lambda $ such that
\begin{equation}
\int_{\mathbb{R}^{3}}\left\vert \nabla u\right\vert ^{2}dx\leq d_{1}\text{
for all }u\in \mathbf{N}_{\mu ,\lambda }^{+}\text{ and for }\mu >0\text{
sufficiently large.}  \label{8-13}
\end{equation}%
Thus, by $(\ref{8-14})$ and $(\ref{8-13}),$ we have%
\begin{equation*}
\left\Vert u\right\Vert _{\mu }^{2}<\frac{a(4-p)}{p-2}d_{1}^{2}+\frac{%
\lambda \left\Vert f\right\Vert _{L^{3/2}}}{S^{2}}d_{1}\text{ for all }u\in
\mathbf{N}_{\mu ,\lambda }^{+}.
\end{equation*}%
$\left( ii\right) $ By Theorem \ref{t6-1} $\left( ii\right) ,$ there exists $%
d_{0}>0$ such that $\alpha _{\mu ,\lambda }^{+}<-d_{0}:=J_{\mu ,\lambda
}\left( t_{a}^{+}\phi _{1}\right) .$ Next, we prove that there exist
constants $D_{0},\mu _{2}>0$ such that
\begin{equation*}
\alpha _{\mu ,\lambda }^{+}>-D_{0}\text{ for all }\mu \geq \mu _{2}\text{
and }a>0.
\end{equation*}%
Let $u\in \mathbf{N}_{\mu ,\lambda }^{+}.$ Similar to $(\ref{8-14})$, we
obtain%
\begin{equation*}
\int_{\mathbb{R}^{3}}fu^{2}dx\leq \frac{\lambda \left\Vert f\right\Vert
_{L^{3/2}}}{S^{2}}\int_{\mathbb{R}^{3}}\left\vert \nabla u\right\vert ^{2}dx
\end{equation*}%
and%
\begin{equation*}
\int_{\mathbb{R}^{3}}Q\left\vert u\right\vert ^{p}dx\leq \frac{\overline{C}%
_{0}^{\frac{6-p}{4}}}{S^{\frac{3(p-2)}{2}}}\left( \frac{a\left( 4-p\right) }{%
2\lambda \left( p-2\right) }\right) ^{\frac{4-p}{2}}\left\Vert u\right\Vert
_{D^{1,2}}^{4}+\frac{\left\Vert Q\right\Vert _{\infty }\left\vert B_{R_{\ast
}}\left( 0\right) \right\vert ^{\frac{6-p}{6}}}{S^{p}}\left\Vert
u\right\Vert _{D^{1,2}}^{p}.
\end{equation*}%
Using the above inequalities gives%
\begin{eqnarray*}
J_{\mu ,\lambda }\left( u\right) &=&\frac{1}{2}\left( \left\Vert
u\right\Vert _{\mu }^{2}-\lambda \int_{\mathbb{R}^{3}}fu^{2}dx\right) +\frac{%
a}{4}\left\Vert u\right\Vert _{D^{1,2}}^{4}-\frac{1}{p}\int_{\mathbb{R}%
^{3}}Q|u|^{p}dx \\
&>&\left[ \frac{a}{4}-\frac{\overline{C}_{0}^{\frac{6-p}{4}}}{S^{\frac{3(p-2)%
}{2}}}\left( \frac{a\left( 4-p\right) }{2\lambda \left( p-2\right) }\right)
^{\frac{4-p}{2}}\right] \left\Vert u\right\Vert _{D^{1,2}}^{4}-\frac{%
\left\Vert Q\right\Vert _{\infty }\left\vert B_{R_{\ast }}\left( 0\right)
\right\vert ^{\frac{6-p}{6}}}{S^{p}}\left\Vert u\right\Vert _{D^{1,2}}^{p} \\
&&-\frac{\lambda \left\Vert f\right\Vert _{L^{3/2}}}{2S^{2}}\left\Vert
u\right\Vert _{D^{1,2}}^{2}.
\end{eqnarray*}%
This implies that there exists a constant $D_{a,\lambda }>0$ such that $%
\alpha _{\mu ,\lambda }^{+}>-D_{a,\lambda }$ for $\mu >0$ sufficiently
large. Moreover, for $u\in \mathbf{N}_{\mu ,\lambda }^{-}\cup \mathbf{N}%
_{\mu ,\lambda }^{0},$ by $\left( \ref{2.2}\right) ,$
\begin{eqnarray*}
J_{\mu ,\lambda }(u) &=&\frac{1}{4}\left( \left\Vert u\right\Vert _{\mu
}^{2}-\lambda \int_{\mathbb{R}^{3}}fu^{2}dx\right) -\frac{4-p}{4p}\int_{%
\mathbb{R}^{3}}Q|u|^{p}dx \\
&\geq &\frac{\left( 4-p\right) \left( p-2\right) }{8p}\int_{\mathbb{R}%
^{3}}Q|u|^{p}dx>0.
\end{eqnarray*}%
Therefore,
\begin{equation*}
\inf_{u\in \mathbf{N}_{\mu ,\lambda }^{-}\cup \mathbf{N}_{\mu ,\lambda
}^{0}}J_{\mu ,\lambda }(u)\geq 0>-d_{0}>\alpha _{\mu ,\lambda
}^{+}>-D_{a,\lambda },
\end{equation*}%
for $\mu >0$ sufficiently large. This completes the proof.
\end{proof}

\begin{proposition}
\label{p3}Suppose that $N\geq 4,2<p<2^{\ast }$ and conditions $\left(
V_{1}\right) -\left( V_{3}\right) $ and $\left( D_{1}\right) -\left(
D_{2}\right) $ hold. Then the following statements are true.\newline
$\left( i\right) $ For each $\lambda >0$ and $a>0$, we have $\mathbf{N}_{\mu
,\lambda }^{+}$ is uniformly bounded for $\mu >0$ sufficiently large;\newline
$\left( ii\right) $ For each $\lambda >0$ and $a>0,$there exist two numbers $%
d_{0},D_{0}>0$ such that
\begin{equation*}
\inf_{u\in \mathbf{N}_{\mu ,\lambda }^{-}\cup \mathbf{N}_{\mu ,\lambda
}^{0}}J_{\mu ,\lambda }(u)\geq 0>-d_{0}>\alpha _{\mu ,\lambda }^{+}>-D_{0}%
\text{ for }\mu >0\text{ sufficiently large.}
\end{equation*}
\end{proposition}

\begin{proof}
$\left( i\right) $ Let $u\in \mathbf{N}_{\mu ,\lambda }^{+}.$ Then by $%
\left( \ref{2.2}\right) $ and the H\"{o}lder and Sobolev inequalities,%
\begin{equation}
\left\Vert u\right\Vert _{\mu }^{2}<\frac{a\left( 4-p\right) }{\left(
p-2\right) }\left\Vert u\right\Vert _{D^{1,2}}^{4}+\frac{\lambda \left\Vert
f\right\Vert _{L^{N/2}}}{S^{2}}\int_{\mathbb{R}^{N}}\left\vert \nabla
u\right\vert ^{2}dx.  \label{8-15}
\end{equation}%
Moreover, using the Sobolev and H\"{o}lder inequalities and $\left( \ref%
{8-15}\right) $ gives%
\begin{eqnarray*}
1 &=&\frac{\int_{\mathbb{R}^{N}}Q\left\vert u\right\vert ^{p}dx+\lambda
\int_{\mathbb{R}^{N}}fu^{2}dx}{\left\Vert u\right\Vert _{\mu
}^{2}+a\left\Vert u\right\Vert _{D^{1,2}}^{4}}<\frac{\int_{\mathbb{R}%
^{N}}Q\left\vert u\right\vert ^{p}dx+\lambda \int_{\mathbb{R}^{N}}fu^{2}dx}{%
a\left\Vert u\right\Vert _{D^{1,2}}^{4}} \\
&\leq &\frac{\left\Vert Q\right\Vert _{\infty }\left( \frac{1}{\mu c}%
\left\Vert u\right\Vert _{\mu }^{2}+\frac{\left\vert \left\{ V<c\right\}
\right\vert ^{2/N}}{S^{2}}\int_{\mathbb{R}^{N}}\left\vert \nabla
u\right\vert ^{2}dx\right) ^{\frac{2p-N\left( p-2\right) }{4}}}{aS^{N\left(
p-2\right) /2}\left\Vert u\right\Vert _{D^{1,2}}^{\frac{8-N\left( p-2\right)
}{2}}}+\frac{\lambda \left\Vert f\right\Vert _{L^{N/2}}}{aS^{2}\left\Vert
u\right\Vert _{D^{1,2}}^{2}} \\
&<&\frac{\left\Vert Q\right\Vert _{\infty }\left[ \frac{a\left( 4-p\right) }{%
\mu c\left( p-2\right) }\left\Vert u\right\Vert _{D^{1,2}}^{4}+\left( \frac{%
\lambda \left\Vert f\right\Vert _{L^{N/2}}}{\mu cS^{2}}+\frac{\left\vert
\left\{ V<c\right\} \right\vert ^{2/N}}{S^{2}}\right) \left\Vert
u\right\Vert _{D^{1,2}}^{2}\right] ^{\frac{2p-N\left( p-2\right) }{4}}}{%
aS^{N\left( p-2\right) /2}\left\Vert u\right\Vert _{D^{1,2}}^{\frac{%
8-N\left( p-2\right) }{2}}}+\frac{\lambda \left\Vert f\right\Vert _{L^{N/2}}%
}{aS^{2}\left\Vert u\right\Vert _{D^{1,2}}^{2}}.
\end{eqnarray*}%
Since%
\begin{equation*}
\frac{8-N\left( p-2\right) }{2}\geq 2p-N\left( p-2\right) \text{ for }N\geq
4,
\end{equation*}%
this implies that there exists a constant $d_{1}>0,$ dependent on $a$ and $%
\lambda $ such that
\begin{equation}
\left\Vert u\right\Vert _{D^{1,2}}\leq d_{1}\text{ for all }u\in \mathbf{N}%
_{\mu ,\lambda }^{+}\text{ and for }\mu >0\text{ sufficiently large.}
\label{8-16}
\end{equation}%
Thus, by $(\ref{8-15})$ and $(\ref{8-16}),$ we have%
\begin{equation*}
\left\Vert u\right\Vert _{\mu }^{2}<\frac{a(4-p)}{p-2}d_{1}^{4}+\frac{%
\lambda \left\Vert f\right\Vert _{L^{N/2}}}{S^{2}}d_{1}^{2}\text{ for all }%
u\in \mathbf{N}_{\mu ,\lambda }^{+}.
\end{equation*}%
$\left( ii\right) $ The proof is essentially same as that in Proposition \ref%
{p2} $\left( ii\right) $, so we omit it here.
\end{proof}

\bigskip

\textbf{We are now ready to prove Theorem \ref{t5}: }$\left( i\right) $ By
the Ekeland variational principle \cite{E}, Lemma \ref{g7} and Proposition %
\ref{p2}, for each $0<\lambda <\lambda _{1}\left( f_{\Omega }\right) $ and $%
0<a<a_{0}$ there exists a bounded sequence $\{u_{n}\}\subset \mathbf{N}_{\mu
,\lambda }^{+}$ such that%
\begin{equation*}
J_{\mu ,\lambda }\left( u_{n}\right) =\alpha _{\mu ,\lambda }^{+}+o(1)\text{
and }J_{\mu ,\lambda }^{\prime }\left( u_{n}\right) =o(1)\text{ in }X_{\mu
}^{-1}.
\end{equation*}%
It follows from Propositions \ref{p1}, \ref{p2} that $J_{\mu ,\lambda }$
satisfies the (PS)$_{\alpha _{\mu ,\lambda }^{+}}$--condition in $\mathbf{N}%
_{\mu ,\lambda }^{+}$ for $\mu >0$ sufficiently large. Thus, there exist a
subsequence $\left\{ u_{n}\right\} $ and $u_{\mu ,\lambda }^{+}\in \mathbf{N}%
_{\mu ,\lambda }^{+}$ such that $u_{n}\rightarrow u_{\mu ,\lambda }^{+}$
strongly in $X_{\mu }$ for $\mu >0$ sufficiently large. Note that $\alpha
_{\mu ,\lambda }^{+}=J_{\mu ,\lambda }\left( u_{\mu ,\lambda }^{+}\right)
<0. $ Hence, $u_{\mu ,\lambda }^{+}\in \mathbf{N}_{\mu ,\lambda }^{+}$ is a
minimizer for $J_{\mu ,\lambda }$ on $\mathbf{N}_{\mu ,\lambda }^{+}.$ Since
$|u_{\mu ,\lambda }^{+}|\in \mathbf{N}_{\mu ,\lambda }^{+}$ and $J_{\mu
,\lambda }\left( |u_{\mu ,\lambda }^{+}|\right) =J_{\mu ,\lambda }\left(
u_{\mu ,\lambda }^{+}\right) =\alpha _{\mu ,\lambda }^{+}<0,$ one can see
that $u_{\mu ,\lambda }^{+}$ is a positive solution for Equation $(E_{\mu
,\lambda })$ by Lemma $\ref{g1}.$\newline
$\left( ii\right) $ By the Ekeland variational principle \cite{E}, Theorem %
\ref{t6-1} $\left( ii\right) $ and Proposition \ref{p2}, for each $a>0$ and $%
\lambda \geq \lambda _{1}\left( f_{\Omega }\right) $ there exists a bounded
sequence $\{u_{n}\}\subset \mathbf{N}_{\mu ,\lambda }^{+}$ with $J_{\mu
,\lambda }\left( u_{n}\right) <-d_{0}<\inf_{u\in \mathbf{N}_{\mu ,\lambda
}^{-}\cup \mathbf{N}_{\mu ,\lambda }^{0}}J_{\mu ,\lambda }(u)$ such that%
\begin{equation*}
J_{\mu ,\lambda }\left( u_{n}\right) =\alpha _{\mu ,\lambda }^{+}+o(1)\text{
and }J_{\mu ,\lambda }^{\prime }\left( u_{n}\right) =o(1)\text{ in }X_{\mu
}^{-1}.
\end{equation*}%
By Proposition \ref{p1}, we can establish a compactness conclusion for $%
\left\{ u_{n}\right\} ,$ this means that there exist a subsequence $\left\{
u_{n}\right\} $ and $u_{\mu ,\lambda }^{+}\in \mathbf{N}_{\mu ,\lambda }^{+}$
such that $u_{n}\rightarrow u_{\mu ,\lambda }^{+}$ strongly in $X_{\mu }$
for $\mu >0$ sufficiently large. In fact that $u_{\mu ,\lambda }^{+}$ is a
positive solution for Equation $(E_{\mu ,\lambda }).$

\bigskip

\textbf{We are now ready to prove Theorem \ref{t5-2}: }The proof is
essentially same as that in Theorem \ref{t5}, so we omit it here.

\subsection{The proof of Theorem \protect\ref{t6}}

Note that $u\in \mathbf{N}_{\mu ,\lambda }$ if and only if $a\left\Vert
u\right\Vert _{D^{1,2}}^{4}+\left\Vert u\right\Vert _{\mu }^{2}=\int_{%
\mathbb{R}^{N}}Q|u|^{p}dx+\lambda \int_{\mathbb{R}^{N}}fu^{2}dx.$ It follows
from Lemma \ref{g3-0}, $\left( \ref{10}\right) $ and the Sobolev inequality
that%
\begin{eqnarray*}
\frac{\widetilde{\lambda }_{1,\mu }\left( f\right) -\lambda }{\widetilde{%
\lambda }_{1,\mu }\left( f\right) }\left\Vert u\right\Vert _{\mu }^{2} &\leq
&\left\Vert u\right\Vert _{\mu }^{2}-\lambda \int_{\mathbb{R}%
^{N}}fu^{2}dx+a\left\Vert u\right\Vert _{D^{1,2}}^{4}=\int_{\mathbb{R}%
^{N}}Q|u|^{p}dx \\
&\leq &\left\Vert Q\right\Vert _{\infty }S^{-p}\left\vert \left\{
V<c\right\} \right\vert ^{\frac{6-p}{6}}\left\Vert u\right\Vert _{\mu }^{p}%
\text{ for all }u\in \mathbf{N}_{\mu ,\lambda }\text{ and }\mu >\overline{%
\mu }_{0}\left( \lambda \right) .
\end{eqnarray*}%
Thus, it leads to
\begin{equation}
\frac{\widetilde{\lambda }_{1,\mu }\left( f\right) }{\widetilde{\lambda }%
_{1,\mu }\left( f\right) -\lambda }\int_{\mathbb{R}^{N}}Q|u|^{p}dx\geq
\left\Vert u\right\Vert _{\mu }^{2}\geq \left( \frac{S^{p}\left( \widetilde{%
\lambda }_{1,\mu }\left( f\right) -\lambda \right) }{\widetilde{\lambda }%
_{1,\mu }\left( f\right) \left\Vert Q\right\Vert _{\infty }\left\vert
\left\{ V<c\right\} \right\vert ^{\frac{2^{\ast }-p}{2^{\ast }}}}\right)
^{2/\left( p-2\right) }  \label{5-8}
\end{equation}%
for all $u\in \mathbf{N}_{\mu ,\lambda }$ and $\mu >\overline{\mu }%
_{0}\left( \lambda \right) .$ Moreover, by $\left( \ref{2.2}\right) $ and $%
\left( \ref{5-8}\right) ,$%
\begin{eqnarray*}
J_{\mu ,\lambda }(u) &=&\frac{1}{4}\left( \left\Vert u\right\Vert _{\mu
}^{2}-\lambda \int_{\mathbb{R}^{N}}fu^{2}dx\right) -\frac{4-p}{4p}\int_{%
\mathbb{R}^{N}}Q|u|^{p}dx \\
&\geq &\frac{\left( p-2\right) \left( \widetilde{\lambda }_{1,\mu }\left(
f\right) -\lambda \right) }{4p\widetilde{\lambda }_{1,\mu }\left( f\right) }%
\left\Vert u\right\Vert _{\mu }^{2} \\
&\geq &\frac{p-2}{4p}\left( \frac{S^{p}}{\left\Vert Q\right\Vert _{\infty
}\left\vert \left\{ V<c\right\} \right\vert ^{\frac{2^{\ast }-p}{2^{\ast }}}}%
\right) ^{2/\left( p-2\right) }\left( \frac{\widetilde{\lambda }_{1,\mu
}\left( f\right) -\lambda }{\widetilde{\lambda }_{1,\mu }\left( f\right) }%
\right) ^{p/\left( p-2\right) }\text{ for all }u\in \mathbf{N}_{\mu ,\lambda
}^{-}.
\end{eqnarray*}%
Hence, the following statement is true.

\begin{lemma}
\label{g1}Suppose that $2<p<\min \left\{ 4,2^{\ast }\right\} $ and condition
$(V1)$ hold. Then $J_{\mu ,\lambda }$ is coercive and bounded below on $%
\mathbf{N}_{\mu ,\lambda }^{-}.$ Furthermore, for all $u\in \mathbf{N}_{\mu
,\lambda }^{-},$ there holds
\begin{equation*}
J_{\mu ,\lambda }(u)>d_{\mu }:=\frac{\left( p-2\right) K^{p}\left( \mu
\right) }{4p}\left( \frac{S^{p}}{\left\Vert Q\right\Vert _{\infty
}\left\vert \left\{ V<c\right\} \right\vert ^{\frac{2^{\ast }-p}{2^{\ast }}}}%
\right) ^{2/\left( p-2\right) },
\end{equation*}%
where $K\left( \mu \right) :=\left( \frac{\widetilde{\lambda }_{1,\mu
}\left( f\right) -\lambda }{\widetilde{\lambda }_{1,\mu }\left( f\right) }%
\right) ^{1/\left( p-2\right) }\leq \left( \frac{\lambda _{1}\left(
f_{\Omega }\right) -\lambda }{\lambda _{1}\left( f_{\Omega }\right) }\right)
^{1/\left( p-2\right) }$ for all $\mu \geq \mu _{0}.$
\end{lemma}

Let $C\left( p\right) :=\left( \frac{2S_{p}^{p}\left( \Omega \right) }{%
Q_{\Omega ,\min }\left( 4-p\right) }\right) ^{2/(p-2)}.$ Then for any $u\in
\mathbf{N}_{\mu ,\lambda }$ with $J_{\mu ,\lambda }(u)<\frac{p-2}{4p}C\left(
p\right) K^{p}\left( \mu \right) ,$ we deduce that%
\begin{eqnarray}
\frac{p-2}{4p}C\left( p\right) K^{p}\left( \mu \right) &>&J_{\mu ,\lambda
}(u)  \notag \\
&=&\frac{1}{2}\left( \left\Vert u\right\Vert _{\mu }^{2}-\lambda \int_{%
\mathbb{R}^{N}}fu^{2}dx\right) +\frac{a}{4}\left\Vert u\right\Vert
_{D^{1,2}}^{4}-\frac{1}{p}\int_{\mathbb{R}^{N}}Q|u|^{p}dx  \notag \\
&\geq &\frac{\left( p-2\right) \left( \widetilde{\lambda }_{1,\mu }\left(
f\right) -\lambda \right) }{2p\widetilde{\lambda }_{1,\mu }\left( f\right) }%
\left\Vert u\right\Vert _{\mu }^{2}-\frac{a(4-p)}{4p}\left\Vert u\right\Vert
_{D^{1,2}}^{4}  \notag \\
&\geq &\frac{\left( p-2\right) \left( \widetilde{\lambda }_{1,\mu }\left(
f\right) -\lambda \right) }{2p\widetilde{\lambda }_{1,\mu }\left( f\right) }%
\left\Vert u\right\Vert _{\mu }^{2}-\frac{a(4-p)}{4p}\left\Vert u\right\Vert
_{\mu }^{4}.  \label{6-4}
\end{eqnarray}%
It implies that if $0<a<C^{-1}\left( p\right) ,$ then there exist two
positive numbers $\widehat{D}_{1}\left( \mu \right) $ and $\widehat{D}%
_{2}\left( \mu \right) $ satisfying%
\begin{equation*}
0<\widehat{D}_{1}\left( \mu \right) <C\left( p\right) K\left( \mu \right) <%
\widehat{D}_{2}\left( \mu \right)
\end{equation*}%
such that%
\begin{equation}
\left\Vert u\right\Vert _{\mu }<\widehat{D}_{1}\left( \mu \right) \text{ or }%
\left\Vert u\right\Vert _{\mu }>\widehat{D}_{2}\left( \mu \right) .
\label{6-6}
\end{equation}%
Thus, we have
\begin{eqnarray}
\mathbf{N}_{\mu ,\lambda }\left[ \frac{p-2}{4p}C\left( p\right) K^{p}\left(
\mu \right) \right] &=&\left\{ u\in \mathbf{N}_{\mu ,\lambda }\text{ }|\text{
}J_{\mu ,\lambda }\left( u\right) <\frac{p-2}{4p}C\left( p\right)
K^{p}\left( \mu \right) \right\}  \notag \\
&=&\mathbf{N}_{\mu ,\lambda }^{(1)}\cup \mathbf{N}_{\mu ,\lambda }^{(2)},
\label{6-5}
\end{eqnarray}%
where
\begin{equation*}
\mathbf{N}_{\mu ,\lambda }^{(1)}:=\left\{ u\in \mathbf{N}_{\mu ,\lambda }%
\left[ \frac{p-2}{4p}C\left( p\right) K^{p}\left( \mu \right) \right]
:\left\Vert u\right\Vert _{\mu }<\widehat{D}_{1}\left( \mu \right) \right\}
\end{equation*}%
and
\begin{equation*}
\mathbf{N}_{\mu ,\lambda }^{(2)}:=\left\{ u\in \mathbf{N}_{\mu ,\lambda }%
\left[ \frac{p-2}{4p}C\left( p\right) K^{p}\left( \mu \right) \right]
:\left\Vert u\right\Vert _{\mu }>\widehat{D}_{2}\left( \mu \right) \right\} .
\end{equation*}%
For $0<a<a_{0}:=\frac{p-2}{4-p}C^{-1}\left( p\right) ,$ we further have
\begin{equation}
\left\Vert u\right\Vert _{\mu }<\widehat{D}_{1}\left( \mu \right)
<C^{1/2}\left( p\right) K\left( \mu \right) \text{ for all }u\in \mathbf{N}%
_{\mu ,\lambda }^{(1)}  \label{5-9}
\end{equation}%
and%
\begin{equation}
\left\Vert u\right\Vert _{\mu }>\widehat{D}_{2}\left( \mu \right)
>C^{1/2}\left( p\right) K\left( \mu \right) \text{ for all }u\in \mathbf{N}%
_{\mu ,\lambda }^{(2)}.  \label{5-10}
\end{equation}%
Using $(\ref{2.2}),(\ref{6-6}),$ condition $\left( D_{4}\right) $ and the
Sobolev inequality gives%
\begin{eqnarray*}
h_{u}^{\prime \prime }\left( 1\right) &=&-2\left( \left\Vert u\right\Vert
_{\mu }^{2}-\lambda \int_{\mathbb{R}^{N}}fu^{2}dx\right) +(4-p)\int_{\mathbb{%
R}^{N}}Q|u|^{p}dx \\
&\leq &\frac{-2\left( \widetilde{\lambda }_{1,\mu }\left( f\right) -\lambda
\right) }{\widetilde{\lambda }_{1,\mu }\left( f\right) }\left\Vert
u\right\Vert _{\mu }^{2}+\frac{\left\Vert Q\right\Vert _{\infty
}(4-p)\left\vert \left\{ V<c\right\} \right\vert ^{\left( 2^{\ast }-p\right)
/2^{\ast }}}{S^{p}}\left\Vert u\right\Vert _{\mu }^{p} \\
&<&0\text{ for all }u\in \mathbf{N}_{\mu ,\lambda }^{(1)}.
\end{eqnarray*}%
By $\left( \ref{6-4}\right) ,$ one has
\begin{eqnarray*}
\frac{1}{4}\left( \left\Vert u\right\Vert _{\mu }^{2}-\lambda \int_{\mathbb{R%
}^{N}}fu^{2}dx\right) -\frac{4-p}{4p}\int_{\mathbb{R}^{N}}Q|u|^{p}dx
&=&J_{\mu ,\lambda }\left( u\right) <\frac{p-2}{4p}C\left( p\right)
K^{p}\left( \mu \right) \\
&<&\frac{p-2}{4p}\left( \left\Vert u\right\Vert _{\mu }^{2}-\lambda \int_{%
\mathbb{R}^{N}}fu^{2}dx\right)
\end{eqnarray*}%
for all $u\in \mathbf{N}_{\mu ,\lambda }^{(2)},$ which implies that%
\begin{equation}
2\left( \left\Vert u\right\Vert _{\mu }^{2}-\lambda \int_{\mathbb{R}%
^{N}}fu^{2}dx\right) <\left( 4-p\right) \int_{\mathbb{R}^{N}}Q|u|^{p}dx\text{
for all }u\in \mathbf{N}_{\mu ,\lambda }^{(2)}.  \label{eqq16}
\end{equation}%
Applying (\ref{2.2}) and (\ref{eqq16}) leads to
\begin{equation*}
h_{u}^{\prime \prime }\left( 1\right) =-2\left( \left\Vert u\right\Vert
_{\mu }^{2}-\lambda \int_{\mathbb{R}^{N}}fu^{2}dx\right) +(4-p)\int_{\mathbb{%
R}^{N}}Q|u|^{p}dx>0\text{ for all }u\in \mathbf{N}_{\mu ,\lambda }^{(2)}.
\end{equation*}%
Moreover, by Theorem \ref{t6-1} $\left( i\right) $, there exist $t_{a}^{\pm
}>0$ such that $t_{a}^{-}w_{\lambda ,\Omega }\in \mathbf{N}_{\mu ,\lambda
}^{\left( 1\right) }$ and $t_{a}^{+}w_{\lambda ,\Omega }\in \mathbf{N}_{\mu
,\lambda }^{(2)}.$ Namely, $\mathbf{N}_{\mu ,\lambda }^{\left( i\right) }$
are nonempty. Hence, we obtain the following result.

\begin{lemma}
\label{g7}Suppose that $2<p<\min \left\{ 4,2^{\ast }\right\} $ and
conditions $\left( V_{1}\right) -\left( V_{3}\right) ,\left( D_{1}\right)
-\left( D_{2}\right) $ and $\left( D_{4}\right) $ hold. Then there exists $%
a_{0}>0$ such that for every $0<a<a_{0}$ and $0<\lambda <\left[ 1-2\left(
\frac{4-p}{4}\right) ^{2/p}\right] \lambda _{1}\left( f_{\Omega }\right) ,$ $%
\mathbf{N}_{\mu ,\lambda }^{(1)}\subset \mathbf{N}_{\mu ,\lambda }^{-}$ and $%
\mathbf{N}_{\mu ,\lambda }^{(2)}\subset \mathbf{N}_{\mu ,\lambda }^{+}$ are $%
C^{1}$ nonempty sub-manifolds. Furthermore, each local minimizer of the
functional $J_{\mu ,\lambda }$ in the sub-manifolds $\mathbf{N}_{\mu
,\lambda }^{(1)}$ and $\mathbf{N}_{\mu ,\lambda }^{(2)}$ is a critical point
of $J_{\mu ,\lambda }$ in $X_{\mu }.$
\end{lemma}

Define
\begin{equation*}
\alpha _{\mu ,\lambda }^{-}=\inf_{u\in \mathbf{N}_{\mu ,\lambda
}^{(1)}}J_{\mu ,\lambda }\left( u\right) =\inf_{u\in \mathbf{N}_{\mu
,\lambda }}J_{\mu ,\lambda }\left( u\right) .
\end{equation*}%
It follows from Lemma \ref{g1} and $(\ref{6-5})$ that
\begin{equation}
0<d_{\mu }<\alpha _{\mu ,\lambda }^{-}<\frac{p-2}{4p}C\left( p\right)
K^{p}\left( \mu \right) \leq \frac{p-2}{4p}C\left( p\right) \left( \frac{%
\lambda _{1}\left( f_{\Omega }\right) -\lambda }{\lambda _{1}\left(
f_{\Omega }\right) }\right) ^{p/\left( p-2\right) }.  \label{8-11}
\end{equation}%
By the Ekeland variational principle \cite{E}, there exists a sequence $%
\{u_{n}\}\subset \mathbf{N}_{\mu ,\lambda }^{(1)}$ such that%
\begin{equation}
J_{\mu ,\lambda }\left( u_{n}\right) =\alpha _{\mu ,\lambda }^{-}+o(1)\text{
and }J_{\mu ,\lambda }^{\prime }\left( u_{n}\right) =o(1)\text{ in }X_{\mu
}^{-1}.  \label{8-12}
\end{equation}

\bigskip

\textbf{We are now ready to prove Theorem \ref{t6}: }By $\left( \ref{6-6}%
\right) ,(\ref{8-11}),(\ref{8-12})$ and Proposition \ref{p1}, for each $%
0<a<a_{0}$ we can obtain that $J_{\mu ,\lambda }$ satisfies the (PS)$%
_{\alpha _{\mu ,\lambda }^{-}}$--condition in $X_{\mu }$ for $\mu >0$
sufficiently large. Thus, there exist a subsequence $\left\{ u_{n}\right\} $
and $u_{\mu ,\lambda }^{-}\in X_{\mu }$ such that $u_{n}\rightarrow u_{\mu
,\lambda }^{-}$ strongly in $X_{\mu }$ for $\mu >0$ sufficiently large.
Hence, $u_{\mu ,\lambda }^{-}$ is a minimizer for $J_{\mu ,\lambda }$ on $%
\mathbf{N}_{\mu ,\lambda }^{\left( 1\right) }.$ Note that
\begin{equation*}
0<\alpha _{\mu ,\lambda }^{-}=J_{\mu ,\lambda }\left( u_{\mu ,\lambda
}^{-}\right) <\frac{p-2}{4p}C\left( p\right) \left( \frac{\lambda _{1}\left(
f_{\Omega }\right) -\lambda }{\lambda _{1}\left( f_{\Omega }\right) }\right)
^{p/\left( p-2\right) },
\end{equation*}%
which implies that $u_{\mu ,\lambda }^{-}\in \mathbf{N}_{\mu ,\lambda
}^{\left( 1\right) }.$ Since $|u_{\mu ,\lambda }^{-}|\in \mathbf{N}_{\mu
,\lambda }^{\left( 1\right) }$ and $J_{\mu ,\lambda }\left( \left\vert
u_{\mu ,\lambda }^{-}\right\vert \right) =J_{\mu ,\lambda }\left( u_{\mu
,\lambda }^{-}\right) =\alpha _{\mu ,\lambda }^{-},$ one can see that $%
u_{\mu ,\lambda }^{-}$ is a positive solution for Equation $(E_{\mu ,\lambda
})$ by Lemma $\ref{g1}.$

\section{Acknowledgments}

T. F. Wu was supported in part by the Ministry of Science and Technology,
Taiwan (Grant No. 108-2115-M-390-007-MY2) and the National Center for
Theoretical Sciences, Taiwan. T. Li was supported by the National Natural
Science Foundation of China (Grant No. 11971105).

\end{document}